\documentclass[11pt]{amsart}
\usepackage{amssymb}
\usepackage{mathrsfs}
\usepackage{cite}
\usepackage{amscd}
\usepackage{amsmath}
\usepackage[all]{xy} 
\usepackage{hyperref}

{\theoremstyle{definition} 
\newtheorem{dfn}{Definition}
}
\newtheorem{pro}{Proposition}
\newtheorem{lmm}{Lemma}
\newtheorem{thm}{Theorem}
\newtheorem{cor}{Corollary}

\newcommand{\HHi}{{\mathcal{H}{\rm ilb}}}
\newcommand{\Hi}{{{\rm Hilb}}}
\newcommand{\et}{\text{\rm \'et}}

\begin{document} 

\title{Components of Gr\"obner strata in the Hilbert scheme of points}
\author{Mathias Lederer}
\email{mlederer@math.cornell.edu}
\thanks{This research was partly supported by Collaborative Research Center 701 (SFB 701)
``Spectral Structures and Topological Methods in Mathematics'' in Bielefeld
and partly by a Marie Curie International Outgoing Fellowship of the EU Seventh Framework Program}
\address{Department of Mathematics, Cornell University, Ithaca, New York 14853, USA}
\date{May 2010}
\keywords{Hilbert scheme of points, Gr\"obner bases, standard sets, irreducible components, connected components}
\subjclass[2010]{14C05; 14N10; 13P10; 13B25}

% 05A17 Partitions of integers 
% 05A19 Combinatorial identities, bijective combinatorics 
% 13B25 Polynomials over commutative rings
% 13F20 Polynomial rings and ideals; rings of integer-valued polynomials
% 14C05 Parametrization (Chow and Hilbert schemes)
% 05C30 Enumeration in graph theory
% 05C78 Graph labelling (graceful graphs, bandwidth, etc.)
% 13P10 Gršbner bases; other bases for ideals and modules (e.g., Janet and border bases)
% 14Q20 Effectivity, complexity
% 14Q99 Computational aspects in algebraic geometry; None of the above, but in this section
% 14N10 Enumerative problems (combinatorial problems)
% 14R10 Affine spaces (automorphisms, embeddings, exotic structures, cancellation problem)
% 68Q25 Analysis of algorithms and problem complexity

\begin{abstract} 
  We fix the lexicographic order $\prec$ on the polynomial ring $S=k[x_{1},\ldots,x_{n}]$ over a ring $k$. 
  We define $\Hi^{\prec\Delta}_{S/k}$, 
  the moduli space of reduced Gr\"obner bases with a given finite standard set $\Delta$, 
  and its open subscheme $\Hi^{\prec\Delta,\et}_{S/k}$, 
  the moduli space of families of $\#\Delta$ points whose attached ideal has the standard set $\Delta$. 
  We determine the number of irreducible and connected components of the latter scheme; 
  we show that it is equidimensional over ${\rm Spec}\,k$; 
  and we determine its relative dimension over ${\rm Spec}\,k$. 
  We show that analogous statements do not hold for the scheme $\Hi^{\prec\Delta}_{S/k}$. 
  Our results prove a version of a conjecture by Bernd Sturmfels. 
\end{abstract}

\maketitle

%%%%%%%%%%%%%%%%%%%%%%%%%%%%%%%%%%%%%%%%%%%%%%%%%%%%%%%%

\section{Introduction}\label{intro}

Let $k$ be a ring and let $S: = k[x]: = k[x_{1},\ldots,x_{n}]$ be the polynomial ring in $n$ variables over $k$. 
Let $\prec$ be the lexicographic order on $S$ such that $x_1 \succ \ldots \succ x_n$. 
Let $\Delta \subset \mathbb{N}^n$ be a standard set of size $r$. 
This is a set with $r$ elements such that 
$\mathbb{N}^n \setminus \Delta$ is closed w.r.t. addition of elements of $\mathbb{N}^n$. 
The minimal generators of the $\mathbb{N}^n$-module $\mathbb{N}^n \setminus \Delta$ are called the {\it corners} of $\Delta$;
we denote the set of corners of $\Delta$ by $\mathscr{C}(\Delta)$. 
In the paper \cite{strata}, we study the functor 
\begin{equation*}
  \HHi^{\prec\Delta}_{S/k}:(k{\rm-Alg})\to({\rm Sets}), 
\end{equation*}
where $(k{\rm-Alg})$ is the category of $k$-algebras and $({\rm Sets})$ is the category of sets. 
This functor attaches to each $k$-algebra $B$ the set of all reduced Gr\"obner bases w.r.t. $\prec$
of monic ideals in $S\otimes_k B$ with standard set $\Delta$. 
For the definition of {\it monic ideals} and {\it reduced Gr\"obner bases with standard set $\Delta$}, 
see \cite{pauer}, \cite{wibmer} or \cite{strata}. 
At this point, we only address that the reduced Gr\"obner basis of a monic ideal $I \subset B$ 
with standard set $\Delta$ is a unique collection of polynomials $f_{\alpha}\in I$,
for all $\alpha\in\mathscr{C}(\Delta)$, which take the shape
  \begin{equation*}
    f_{\alpha}=x^\alpha+\sum_{\beta\in\Delta,\,\beta\prec\alpha}d_{\alpha,\beta}x^\beta .
\end{equation*}
In \cite{strata}, we prove that the functor $\HHi^{\prec\Delta}_{S/k}$ is representable 
by an affine scheme $\Hi^{\prec\Delta}_{S/k}$. 
This scheme is therefore the moduli space of all reduced Gr\"obner bases in $S$ with standard set $\Delta$.
It turns out that the coordinate ring of $\Hi^{\prec\Delta}_{S/k}$ is a $k$-algebra of finite type. 
In particular, if the ring $k$ is noetherian, 
the topological space $\Hi^{\prec\Delta}_{S/k}$ has only finitely many irreducible components
and only finitely many connected components. 

In the cited paper, we also show that $\Hi^{\prec\Delta}_{S/k}$ is a closed subscheme of an open subscheme
$\Hi^{\Delta}_{S/k}$ of $\Hi^{r}_{S/k}$, 
the {\it Hilbert scheme of $r$ points}. (Remember that $r = \#\Delta$.)
The latter scheme is a classical object of study, see \cite{huibregtse}, \cite{norge}, \cite{bertin} and the references therein.
The intermediate scheme $\Hi^{\Delta}_{S/k}$ is the moduli space of all {\it $\Delta$-border bases}
in the terminology of \cite{kk1}, \cite{krbook},\cite{kk2} and \cite{braunpokutta}.
The schemes $\Hi^{\Delta}_{S/k}$, where $\Delta$ runs through all standard sets of size $r$, 
form an open cover of $\Hi^{r}_{S/k}$. 

In the paper \cite{jpaa}, 
we defined an addition operation on the set of all standard sets in $\mathbb{N}^n$, 
which is reminiscent of the popular game {\it Connect Four}. 
Accordingly, if $\Delta$ and $\Delta^\prime$ are two standard sets, 
we call $\Delta+\Delta^\prime$ their {\it Connect Four sum}, or \emph{C4 sum}. 
We will study this operation in Section \ref{combinatorics}. 
In particular, we will define what a {\it Connect Four decomposition}, 
or \emph{C4 decomposition} $\{\Delta_i: i\in I\}$ of a standard set $\Delta$ is. 
For the sake of simplicity, we will identify a C4 decomposition as above with its indexing set $I$. 
Moreover, we will introduce the {\it C4 decomposition number} of $\Delta$, 
which measures all possibilities to iteratively decompose $\Delta$ into a C4 sum of indecomposable standard sets. 

Based on the combinatorial structures introduced in Section \ref{combinatorics}, 
we will construct a number of auxiliary schemes in Section \ref{affines}. 
All of them will be based on the affine scheme $\Hi^{\prec\Delta}_{S/k}$;
the most important one of them will be denoted by $Y^\Delta$. 
This scheme will be the coproduct of schemes $Y^I$, where $I$ runs through all C4 decompositions of $\Delta$. 
The scheme $Y^\Delta$, and the functor it represents, will be used in the rest of the paper. 

In Section \ref{connectfour}, a large portion of the technical work of the present paper is carried out. 
We introduce a morphism $\tau:Y^{\Delta}\to\Hi^{\prec\Delta}_{S/k}$,
which is a universal version of the Connect Four operation of Gr\"obner bases we introduced in our paper \cite{jpaa},
and will therefore be called the {\it Connect Four morphism}, or \emph{C4 morphism}. 
At the end of Section \ref{connectfour}, we will prove the following Theorem: 

\begin{thm}\label{immersion}
  The morphism $\tau:Y^{\Delta}\to\Hi^{\prec\Delta}_{S/k}$ is an immersion. 
\end{thm}

This theorem is interesting in its own right, as it reveals some of the structure of the Hilbert scheme of points. 
Indeed, as a topological space, $\Hi^{r}_{S/k}$ is the coproduct of all $\Hi^{\prec\Delta}_{S/k}$, 
where $\Delta$ runs through all standard sets of size $r$ in $\mathbb{N}^n$, see \eqref{rdelta} below. 
Theorem \ref{immersion} sheds some light on subschemes of $\Hi^{\prec\Delta}_{S/k}$, 
given by the summands of $Y^\Delta$. 

However, the C4 morphism will turn out to be particularly interesting 
when restricted to a certain subscheme $Y^{\Delta,\et}$ of $Y^\Delta$. 
More precisely, in Section \ref{reducedmoduli} we will study a subfunctor 
$\HHi^{\prec\Delta,\et}_{S/k}$ of $\HHi^{\prec\Delta}_{S/k}$. 
Here is a brief description of that subfunctor. 
Given a $k$-algebra $B$, an element of $\HHi^{\prec\Delta}_{S/k}(B)$ 
is the reduced Gr\"obner basis $f_{\alpha}$, for $\alpha\in\mathscr{C}(\Delta)$, of an ideal $I \subset S \otimes_k B$. 
Let us write $Q$ for the $B$-algebra $S \otimes_k B/I$. Then $Q$ is a free $B$-module. 
In particular, the morphism ${\rm Spec}\,Q\to{\rm Spec}\,B$ corresponding to the canonical ring homomorphism 
$B\to Q$ is surjective and flat of degree $r$. 
Now the functor $\HHi^{\prec\Delta,\et}_{S/k}$ attaches to a $k$-algebra $B$ 
the set of all reduced Gr\"obner basis $f_{\alpha}$, for $\alpha\in\mathscr{C}(\Delta)$, 
such that the morphism ${\rm Spec}\,Q\to{\rm Spec}\,B$ is surjective and \'etale of degree $r$. 
We will show that $\HHi^{\prec\Delta,\et}_{S/k}$ is representable by an open subscheme 
$\Hi^{\prec\Delta,\et}_{S/k}$ of $\Hi^{\prec\Delta}_{S/k}$. 
This scheme is the moduli space of families of $r$ distinct points in affine $n$-space
whose defining ideal has standard set $\Delta$. 

In Section \ref{connectonreduced}, we define an open subscheme $Y^{\Delta,\et}$ of $Y^\Delta$. 
We will show that the restriction to $Y^{\Delta,\et}$ of the C4 morphism yields a morphism 
$\tau^\et:Y^{\Delta,\et}\to\Hi^{\prec\Delta,\et}_{S/k}$. 
Upon considering the projections
\begin{equation}\label{qj}
  q_{j}:\mathbb{N}^n\to\mathbb{N}^{n-j+1}:
  (\beta_{1},\ldots,\beta_{n})\mapsto(\beta_{j},\ldots,\beta_{n}) ,
\end{equation}
for $j=1,\ldots,n$, we will prove the following Theorem: 

\begin{thm}\label{mainthm}
  \begin{enumerate}
    \item[(i)] At the level of topological spaces, the C4 morphism $\tau^\et$ is a homeomorphism. 
    \item[(ii)] If ${\rm Spec}\,k$ is irreducible, then $\Hi^{\prec\Delta,\et}_{S/k}$ is equidimensional over ${\rm Spec}\,k$ 
      of dimension $\sum_{j=1}^n\#q_{j}(\Delta)$. 
   \item[(iii)] In this case, the number of irreducible components of $\Hi^{\prec\Delta,\et}_{S/k}$ 
     equals the C4 decomposition number $d(\Delta)$. 
    \item[(iv)] If ${\rm Spec}\,k$ is irreducible and connected, 
      then also the number of connected components of $\Hi^{\prec\Delta,\et}_{S/k}$
      equals the C4 decomposition number $d(\Delta)$. 
  \end{enumerate}
\end{thm}

This theorem proves a version of a conjecture of Bernd Sturmfels, 
which he made on the basis of an example which we shall present in Section \ref{combinatorics}. 

\subsection*{Acknowledgments}

The author wishes to thank the following large group of people: 
Graham Ellis and G\"otz Pfeiffer, and all other people at the {\it First de Br\'{u}n Workshop} in Galway, Ireland;
Bernd Sturmfels for sharing his deep insight; 
my colleagues Michael Spie{\ss}, Thomas Zink, Eike Lau (thanks for the support in \'etale questions!) 
and Vytautas Paskunas in Bielefeld for so generously teaching me a lot of beautiful mathematics; 
Mike Stillman for his warm hospitality and lots of support when I moved to Cornell University; 
Roy Skjelnes for inviting me to Stockholm; 
Robin Hartshorne, Diane Maclagan, Gregory G. Smith, and all participants of the workshop 
{\it Components of Hilbert Schemes} held at American Institute of Mathematics in Palo Alto, California. 
Special thanks go to Laurent Evain for sketching Theorem \ref{immersion} --- his comments 
greatly improved the quality of the research presented here. 
Many thanks to the anonymous referee for his or her careful reading and pointing out my mistakes. 

%%%%%%%%%%%%%%%%%%%%%%%%%%%%%%%%%%%%%%%%%%%%%%%%%%%%%%%%

\section{Combinatorics of lexicographic standard sets}\label{combinatorics}

%%%%%%%

\subsection{Connect Four addition of standard sets}

We start by reviewing the most essential definition of our paper \cite{jpaa}. 
Let $\Delta$ and $\Delta^\prime$ be standard sets in $\mathbb{N}^n$. 
In addition to the projection $q_{n}$ we introduced in \eqref{qj}, we consider its complement, $q^n$, hence
\begin{equation*}
  \begin{split}
    q^n:\mathbb{N}^n&\to\mathbb{N}^{n-1}:(\beta_{1},\ldots,\beta_{n})\mapsto(\beta_{1},\ldots,\beta_{n-1}) ,\\
    q_{n}:\mathbb{N}^n&\to\mathbb{N}:(\beta_{1},\ldots,\beta_{n})\mapsto\beta_{n} .
  \end{split}
\end{equation*}
We define {\it addition of standard sets} by the formula
\begin{equation*}
  \Delta + \Delta^\prime := \left\lbrace
  \begin{array}{c}
    \beta\in\mathbb{N}^n: \\
    q_{n}(\beta) < \#((q^n)^{-1}((q^n)(\beta))\cap\Delta) + \#((q^n)^{-1}((q^n)(\beta))\cap\Delta^\prime) 
  \end{array}
  \right\rbrace .
\end{equation*}
This addition is map best visualized as an analogue of dropping discs in the popular game {\it Connect Four} 
(see Figure \ref{connectfigure}). We define $\mathcal{D}_{n}$ to be the set of all finite standard sets in $\mathbb{N}^n$.
The following facts hold true (see \cite{jpaa}).
\begin{enumerate}
  \item[(i)] $(\mathcal{D}_{n},+)$ is a commutative and associative monoid with neutral element $\emptyset$. 
  \item[(ii)] For all $\Delta,\Delta^\prime\in\mathcal{D}_{n}$, we have $\#(\Delta+\Delta^\prime)=\#\Delta+\#\Delta^\prime$. 
  \item[(iii)] We embed $\mathcal{D}_{n-1}$ into $\mathcal{D}_{n}$ by sending $\Delta \subset \mathbb{N}^{n-1}$
  to the subset $\Delta\times\{0\}$ of $\mathbb{N}^n$. This yields an addition map
  $+:\mathcal{D}_{n-1}\times\mathcal{D}_{n-1}\to\mathcal{D}_{n}$, 
  and more generally, $+:\mathcal{D}_{n-1}\times\mathcal{D}_{n}\to\mathcal{D}_{n}$.
  \item[(iv)] For the time being, let $k$ be a field. 
  Let $A \subset k^n$ be a finite set of closed $k$-rational points of $\mathbb{A}^n_k$. 
  Denote by $D(A)\in\mathcal{D}_{n}$ the standard set of the ideal $I \subset S$ defining the closed subscheme 
  $A \subset \mathbb{A}^n_k$. 
  For all $\lambda\in k$, denote by $A_{\lambda}$ the intersection of $A$ with the hyperplane $\{x_{n}=\lambda\}$
  in $\mathbb{A}^n_k$. We understand $A_{\lambda}$ to be a closed subscheme of $\mathbb{A}^{n-1}$, 
  and accordingly, define $D(A_{\lambda})\in\mathcal{D}_{n-1}$ to be the attached standard set. 
  In particular, almost all $A_{\lambda}$, therefore also almost all $D(A_{\lambda})$, are empty. 
  By the main theorem of \cite{jpaa}, we have $D(A)=\sum_{\lambda\in k}D(A_{\lambda})$, 
  where the addition of standard sets is defined as in (iii). 
  \item[(v)] By induction over $n$, this gives a complete description of $D(A)$. 
\end{enumerate}

%%%%%%%

\subsection{C4 Decompositions of standard sets}\label{subsectiondecomps}

\begin{dfn}
  Let $\Delta\in\mathcal{D}_n$. 
  A finite multiset $\{\Delta_i: i\in I\}$ of elements of $\mathcal{D}_{n-1}$
  is called a {\it Connect Four decomposition}, or, for short, a {\it C4 decomposition} of $\Delta$ if 
  \begin{equation}\label{sum}
    \Delta=\sum_{i \in I}\Delta_i ,
  \end{equation}
  where the sum is defined as in (iii) above. 
  In what follows, we will often write the indexing set of the C4 decomposition as a coproduct 
  \begin{equation}\label{ij}
    I = I_1 \coprod \ldots \coprod I_m
  \end{equation}
  such that for all $j$, all elements of the multiset $\{\Delta_i: i \in I_j\}$ agree, 
  and elements of of $I_j$ and $I_b$ do not agree if $j \neq b$. 
\end{dfn}

For $\Delta\in\mathcal{D}_{n}$ as above, let $h = \#q_{n}(\Delta)$, the {\it height} of $\Delta$. 
Then clearly each C4 decomposition of $\Delta$ is indexed by a set $I$ of size $h$. 
Moreover, each $\Delta\in\mathcal{D}_{n}$ admits at least one C4 decomposition. 
That decomposition is given by slicing $\Delta$ horizontally into $h$ standard subsets --- more
precisely, to consider
\begin{equation*}
  \Delta_i := q^n(\Delta\cap\{\beta\in\mathbb{N}^n: q_{n}(\beta)=i\}) ,
\end{equation*}
for $i = 0, \ldots, h-1$. Then \eqref{sum} holds true for that collection of $\Delta_i$. 
Figure \ref{connectfigure} shows an example in which we add eight elements of $\mathcal{D}_{3}$
(which we can also view as the embeddings of elements of $\mathcal{D}_{2}$)
and get another element of $\mathcal{D}_{3}$. 
The C4 decomposition depicted there is, however, not obtained by the trivial slicing process just described. 
In particular, the standard set of Figure \ref{connectfigure} admits more than one C4 decomposition. 

\begin{center}
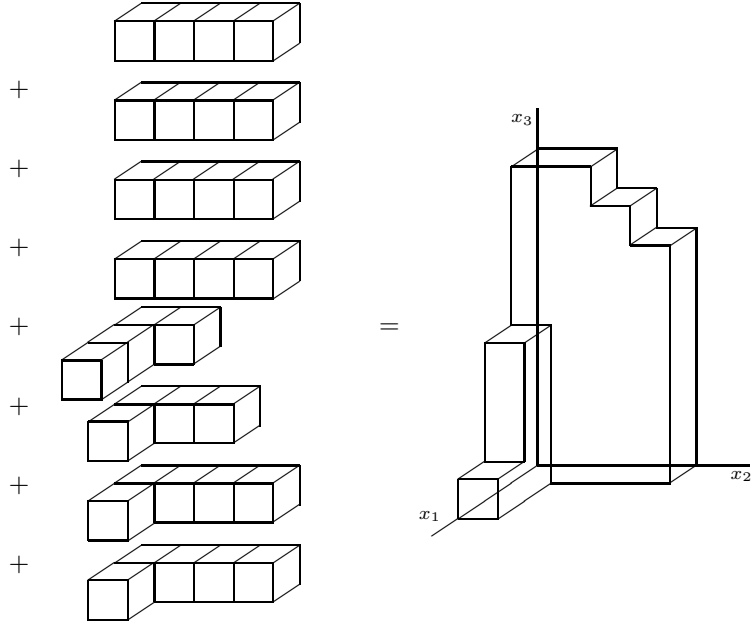
\begin{figure}[ht]
  \begin{picture}(320,250)
    % lowest
    \put(20,20){\small $+$}
    \put(70,30){\line(1,0){60}}
    \put(60,23.2){\line(1,0){60}}
    \put(50,16.6){\line(1,0){15}}
    \put(50,16.6){\line(0,-1){15}}
    \put(50,1.6){\line(1,0){15}}
    \put(65,1.6){\line(0,1){15}}
    \put(65,1.6){\line(3,2){10}}
    \put(75,8.4){\line(0,1){15}}
    \put(75,8.4){\line(1,0){45}}
    \put(90,8.4){\line(0,1){15}}
    \put(105,8.4){\line(0,1){15}}
    \put(120,8.4){\line(0,1){15}}
    \put(120,8.5){\line(3,2){10}}
    \put(130,30){\line(0,-1){15}}
    \put(70,30){\line(-3,-2){20}}
    \put(85,30){\line(-3,-2){20}}
    \put(100,30){\line(-3,-2){10}}
    \put(115,30){\line(-3,-2){10}}
    \put(130,30){\line(-3,-2){10}}
    % one up
    \put(20,50){\small $+$}
    \put(70,60){\line(1,0){60}}
    \put(60,53.2){\line(1,0){60}}
    \put(50,46.6){\line(1,0){15}}
    \put(50,46.6){\line(0,-1){15}}
    \put(50,31.6){\line(1,0){15}}
    \put(65,31.6){\line(0,1){15}}
    \put(65,31.6){\line(3,2){10}}
    \put(75,38.4){\line(0,1){15}}
    \put(75,38.4){\line(1,0){45}}
    \put(90,38.4){\line(0,1){15}}
    \put(105,38.4){\line(0,1){15}}
    \put(120,38.4){\line(0,1){15}}
    \put(120,38.5){\line(3,2){10}}
    \put(130,60){\line(0,-1){15}}
    \put(70,60){\line(-3,-2){20}}
    \put(85,60){\line(-3,-2){20}}
    \put(100,60){\line(-3,-2){10}}
    \put(115,60){\line(-3,-2){10}}
    \put(130,60){\line(-3,-2){10}}
    % two up
    \put(20,80){\small $+$}
    \put(70,90){\line(1,0){45}}
    \put(60,83.2){\line(1,0){45}}
    \put(50,76.6){\line(1,0){15}}
    \put(50,76.6){\line(0,-1){15}}
    \put(50,61.6){\line(1,0){15}}
    \put(65,61.6){\line(0,1){15}}
    \put(65,61.6){\line(3,2){10}}
    \put(75,68.4){\line(0,1){15}}
    \put(75,68.4){\line(1,0){30}}
    \put(90,68.4){\line(0,1){15}}
    \put(105,68.4){\line(0,1){15}}
    \put(105,68.5){\line(3,2){10}}
    \put(115,90){\line(0,-1){15}}
    \put(70,90){\line(-3,-2){20}}
    \put(85,90){\line(-3,-2){20}}
    \put(100,90){\line(-3,-2){10}}
    \put(115,90){\line(-3,-2){10}}
    % three up
    \put(20,110){\small $+$}
    \put(70,120){\line(1,0){30}}
    \put(60,113.2){\line(1,0){30}}
    \put(50,106.4){\line(1,0){15}}
    \put(40,99.8){\line(1,0){15}}
    \put(40,99.8){\line(0,-1){15}}
    \put(40,85){\line(1,0){15}}
    \put(55,99.8){\line(0,-1){15}}
    \put(55,85){\line(3,2){20}}
    \put(65,106.4){\line(0,-1){15}}
    \put(75,113.2){\line(0,-1){15}}
    \put(70,120){\line(-3,-2){30}}
    \put(85,120){\line(-3,-2){30}}
    \put(100,120){\line(-3,-2){10}}
    \put(75,98.2){\line(1,0){15}}
    \put(90,98.2){\line(0,1){15}}
    \put(90,98.5){\line(3,2){10}}
    \put(100,120){\line(0,-1){15}}
    % four up
    \put(20,140){\small $+$}
    \put(70,145){\line(1,0){60}}
    \put(60,138.2){\line(1,0){60}}
    \put(60,138.2){\line(0,-1){15}}
    \put(75,138.2){\line(0,-1){15}}
    \put(90,138.2){\line(0,-1){15}}
    \put(105,138.2){\line(0,-1){15}}
    \put(120,138.2){\line(0,-1){15}}
    \put(130,145){\line(0,-1){15}}
    \put(60,123.2){\line(1,0){60}}
    \put(70,145){\line(-3,-2){10}}
    \put(85,145){\line(-3,-2){10}}
    \put(100,145){\line(-3,-2){10}}
    \put(115,145){\line(-3,-2){10}}
    \put(130,145){\line(-3,-2){10}}
    \put(120,123.2){\line(3,2){10}}
    % five up
    \put(20,170){\small $+$}
    \put(70,175){\line(1,0){60}}
    \put(60,168.2){\line(1,0){60}}
    \put(60,168.2){\line(0,-1){15}}
    \put(75,168.2){\line(0,-1){15}}
    \put(90,168.2){\line(0,-1){15}}
    \put(105,168.2){\line(0,-1){15}}
    \put(120,168.2){\line(0,-1){15}}
    \put(130,175){\line(0,-1){15}}
    \put(60,153.2){\line(1,0){60}}
    \put(70,175){\line(-3,-2){10}}
    \put(85,175){\line(-3,-2){10}}
    \put(100,175){\line(-3,-2){10}}
    \put(115,175){\line(-3,-2){10}}
    \put(130,175){\line(-3,-2){10}}
    \put(120,153.2){\line(3,2){10}}
    % six up
    \put(20,200){\small $+$}
    \put(70,205){\line(1,0){60}}
    \put(60,198.2){\line(1,0){60}}
    \put(60,198.2){\line(0,-1){15}}
    \put(75,198.2){\line(0,-1){15}}
    \put(90,198.2){\line(0,-1){15}}
    \put(105,198.2){\line(0,-1){15}}
    \put(120,198.2){\line(0,-1){15}}
    \put(130,205){\line(0,-1){15}}
    \put(60,183.2){\line(1,0){60}}
    \put(70,205){\line(-3,-2){10}}
    \put(85,205){\line(-3,-2){10}}
    \put(100,205){\line(-3,-2){10}}
    \put(115,205){\line(-3,-2){10}}
    \put(130,205){\line(-3,-2){10}}
    \put(120,183.2){\line(3,2){10}}
    % seven up
    \put(70,235){\line(1,0){60}}
    \put(60,228.2){\line(1,0){60}}
    \put(60,228.2){\line(0,-1){15}}
    \put(75,228.2){\line(0,-1){15}}
    \put(90,228.2){\line(0,-1){15}}
    \put(105,228.2){\line(0,-1){15}}
    \put(120,228.2){\line(0,-1){15}}
    \put(130,235){\line(0,-1){15}}
    \put(60,213.2){\line(1,0){60}}
    \put(70,235){\line(-3,-2){10}}
    \put(85,235){\line(-3,-2){10}}
    \put(100,235){\line(-3,-2){10}}
    \put(115,235){\line(-3,-2){10}}
    \put(130,235){\line(-3,-2){10}}
    \put(120,213.2){\line(3,2){10}}
    % coordinate system
    \put(220,60){\line(1,0){80}}
    \put(220,60){\line(-3,-2){40}}
    \put(220,60){\line(0,1){135}}
    \put(175,40){\tiny $x_{1}$}
    \put(293,55){\tiny $x_{2}$}
    \put(210,190){\tiny $x_{3}$}
    % all at once
    \put(160,110){\small $=$}
    \put(225,53.2){\line(1,0){45}}
    \put(225,53.2){\line(0,1){60}}
    \put(225,53.2){\line(-3,-2){20}}
    \put(280,60){\line(0,1){90}}
    \put(280,60){\line(-3,-2){10}}
    \put(270,53.2){\line(0,1){90}}
    \put(210,173.2){\line(1,0){30}}
    \put(210,173.2){\line(3,2){10}}
    \put(220,180){\line(1,0){30}}
    \put(210,113.2){\line(0,1){60}}
    \put(210,113.2){\line(1,0){15}}
    \put(210,113.2){\line(-3,-2){10}}
    \put(200,61.4){\line(0,1){45}}
    \put(200,61.4){\line(1,0){15}}
    \put(215,106.4){\line(-1,0){15}}
    \put(215,106.4){\line(0,-1){45}}
    \put(215,106.4){\line(3,2){10}}
    \put(190,39.8){\line(1,0){15}}
    \put(190,39.8){\line(0,1){15}}
    \put(190,54.8){\line(1,0){15}}
    \put(190,54.8){\line(3,2){10}}
    \put(205,54.8){\line(3,2){10}}
    \put(205,54.8){\line(0,-1){15}}
    \put(270,143.2){\line(-1,0){15}}
    \put(270,143.2){\line(3,2){10}}
    \put(255,143.2){\line(0,1){15}}
    \put(255,143.4){\line(3,2){10}}
    \put(280,150){\line(-1,0){15}}
    \put(265,150){\line(0,1){15}}
    \put(255,158.2){\line(-1,0){15}}
    \put(255,158.2){\line(3,2){10}}
    \put(265,165){\line(-1,0){15}}
    \put(240,158.4){\line(3,2){10}}
    \put(240,158.2){\line(0,1){15}}
    \put(250,165){\line(0,1){15}}
    \put(240,173.2){\line(3,2){10}}
  \end{picture}
\caption{C4 addition of standard sets}
\label{connectfigure}
\end{figure}
\end{center}

Evidently, for $n=1$ or $n=2$, each $\Delta\in\mathcal{D}_{n}$ admits only one C4 decomposition, 
namely, the one given by the above-described slicing process. 
(We understand $\mathcal{D}_{0}$ to consist of exactly one point.) 
Bernd Sturmfels observed that elements of $\mathcal{D}_n$ in general admit more than one C4 decomposition if $n\geq3$, 
the smallest example being 
\begin{equation*}
  \Delta := \{0, e_{1}, e_{2}, e_{3}\}\in\mathcal{D}_{3} .
\end{equation*}
The two C4 decompositions of $\Delta$ are $\{\{0,e_{1},e_{2}\},\{0\}\}$ and $\{\{0,e_{1}\},\{0,e_{2}\}\}$,
see Figure \ref{nonuniqueness}. 
For better visibility, that picture does not show the elements $\Delta_i\in\mathcal{D}_{2}$ 
whose C4 sum is $\Delta$, yet rather their embeddings $\Delta_i\times\{0\}$ into $\mathcal{D}_{3}$. 
We will return to this standard set later, see Section \ref{original} below. 
As we have seen, also the standard set of Figure \ref{connectfigure} admits several different C4 decompositions. 
Nonuniqueness of the C4 decomposition of $\Delta$ motivates the following notion.

\begin{center}
\begin{figure}
  \begin{picture}(280,110)
    % coordinate system one
    \put(40,50){\line(1,0){45}}
    \put(40,50){\line(-3,-2){30}}
    \put(40,50){\line(0,1){45}}
    \put(5,37){\tiny $x_{1}$}
    \put(78,44){\tiny $x_{2}$}
    \put(30,89){\tiny $x_{3}$}
    % all at once
    \put(55,65){\line(1,0){15}}
    \put(40,80){\line(1,0){15}}
    \put(55,65){\line(0,1){15}}
    \put(70,50){\line(0,1){15}}
    \put(29.8,58.2){\line(-3,-2){10}}
    \put(40,80){\line(-3,-2){10}}
    \put(44.8,43.2){\line(-3,-2){10}}
    \put(55,65){\line(-3,-2){20}}
    \put(55,80){\line(-3,-2){10}}
    \put(70,50){\line(-3,-2){10}}
    \put(70,65){\line(-3,-2){10}}
    \put(44.8,43.2){\line(1,0){15}}
    \put(29.8,58.2){\line(1,0){30}}
    \put(29.8,73.2){\line(1,0){15}}
    \put(29.8,58.2){\line(0,1){15}}
    \put(44.8,43.2){\line(0,1){30}}
    \put(59.8,43.2){\line(0,1){15}}
    \put(19.8,36.6){\line(1,0){15}}
    \put(19.8,36.6){\line(0,1){15}}
    \put(34.8,51.6){\line(-1,0){15}}
    \put(34.8,51.6){\line(0,-1){15}}
    % equals
    \put(100,50){\small $=$}
    % first decomposition
    % the lower part
    \put(140,45){\line(1,0){30}}
    \put(170,30){\line(0,1){15}}
    \put(140,45){\line(-3,-2){20}}
    \put(144.8,23.2){\line(-3,-2){10}}
    \put(155,45){\line(-3,-2){20}}
    \put(170,30){\line(-3,-2){10}}
    \put(170,45){\line(-3,-2){10}}
    \put(144.8,23.2){\line(1,0){15}}
    \put(129.8,38.2){\line(1,0){30}}
    \put(144.8,23.2){\line(0,1){15}}
    \put(159.8,23.2){\line(0,1){15}}
    \put(119.8,16.6){\line(1,0){15}}
    \put(119.8,16.6){\line(0,1){15}}
    \put(134.8,31.6){\line(-1,0){15}}
    \put(134.8,31.6){\line(0,-1){15}}
    % plus
    \put(138,57){\small $+$}
    % the upper part
    \put(155,80){\line(-3,-2){10}}
    \put(155,95){\line(-3,-2){10}}
    \put(140,95){\line(-3,-2){10}}
    \put(155,95){\line(-1,0){15}}
    \put(155,95){\line(0,-1){15}}
    \put(129.8,73.2){\line(1,0){15}}
    \put(129.8,73.2){\line(0,1){15}}
    \put(144.8,88.2){\line(-1,0){15}}
    \put(144.8,88.2){\line(0,-1){15}}
    % equals
    \put(188,50){\small $=$}
    % second decomposition
    % the lower part
    \put(245,45){\line(-1,0){15}}
    \put(245,45){\line(0,-1){15}}
    \put(245,30){\line(-3,-2){20}}
    \put(230,45){\line(-3,-2){20}}
    \put(245,45){\line(-3,-2){20}}
    \put(234.8,38.2){\line(-1,0){15}}
    \put(234.8,38.2){\line(0,-1){15}}
    \put(209.8,16.6){\line(1,0){15}}
    \put(209.8,16.6){\line(0,1){15}}
    \put(224.8,31.6){\line(-1,0){15}}
    \put(224.8,31.6){\line(0,-1){15}}
    % plus
    \put(228,57){\small $+$}
    % the upper part
    \put(230,95){\line(-3,-2){10}}
    \put(245,95){\line(-3,-2){10}}
    \put(260,80){\line(-3,-2){10}}
    \put(260,95){\line(-3,-2){10}}
    \put(230,95){\line(1,0){30}}
    \put(260,95){\line(0,-1){15}}
    \put(219.8,73.2){\line(1,0){30}}
    \put(219.8,73.2){\line(0,1){15}}
    \put(234.8,73.2){\line(0,1){15}}
    \put(249.8,73.2){\line(0,1){15}}
    \put(219.8,88.2){\line(1,0){30}}
  \end{picture}
\caption{A standard set admitting two C4 decompositions}
\label{nonuniqueness}
\end{figure}
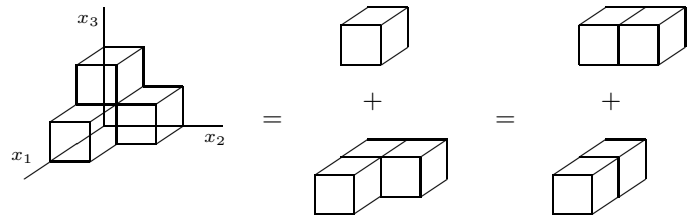
\end{center}

\begin{dfn}\label{graphs}
  Let $\Delta\in\mathcal{D}_n$. We define two labeled graphs attached to $\Delta$. 
  \begin{enumerate}
    \item[(i)] The {\it C4 decomposition graph of $\Delta$} is the rooted tree with the following nodes,
    \begin{itemize}
      \item the root $\Delta$, whose label is $n$;
      \item one node for each C4 decomposition $\{\Delta_i: i\in I\}$ of $\Delta$, each having the label $n-\frac{1}{2}$;
      \item one node for each $\Delta_i$ appearing in a C4 decomposition of $\Delta$, each having the label $n-1$;
    \end{itemize}
    and the following edges,
    \begin{itemize}
      \item one edge between $\Delta$ and each C4 decomposition $\{\Delta_i: i\in I\}$; and
      \item for all C4 decompositions $\{\Delta_i: i\in I\}$, one edge between that C4 decomposition and each standard set 
        $\Delta_i$ appearing therein. 
    \end{itemize}
    \item[(ii)] In the C4 decomposition graph of $\Delta$, we replace each leaf, i.e., 
    each node of the form $\Delta_i\in\mathcal{D}_{n-1}$, with its C4 decomposition graph. 
    In the resulting graph, a rooted tree with root $\Delta$, we replace each leaf, i.e., 
    each node of the form $\Delta_i\in\mathcal{D}_{n-2}$, 
    with its C4 decomposition graph. We repeat this process $n$ times,
    in the last step adding leaves that lie in $\mathcal{D}_{0}$. 
    We obtain a rooted tree with root $\Delta$ whose nodes are labeled 
    $n,n-\frac{1}{2},n-1,n-\frac{3}{2},\ldots,\frac{1}{2},0$. 
    We call that graph the {\it iterated C4 decomposition graph of $\Delta$}. 
  \end{enumerate}
\end{dfn}

Figure \ref{graph} shows the C4 decomposition graph of an abstract $\Delta$ lying in some $\mathcal{D}_{n}$. 
We denoted the various indexing sets by $I^1, \ldots, I^t$ 
for avoiding confusion with the sets $I_1, \ldots, I_m$ of \eqref{ij}. 
Note that the numbers $h^1 := \# I^1, h^2 := \# I^2, \ldots, h^t := \# I^t$ 
appearing in the bottom line of Figure \ref{graph} are all equal to $\# q_n(\Delta)$. 
In the figure, 
we use an indexing of the nodes which is slightly different from the indexing we used in Definition \ref{graphs} above. 
Moreover, we we write the labels $n$, $n-\frac{1}{2}$ and $n-1$ in a separate column at the right hand side, 
as the labels are the same for all nodes on the same horizontal line. 
Figure \ref{longgraph} shows the iterated C4 decomposition graph of the standard set $\Delta=\{0,e_{1},e_{2},e_{3}\}$, 
which we also considered in Figure \ref{nonuniqueness} above. 

\begin{center}
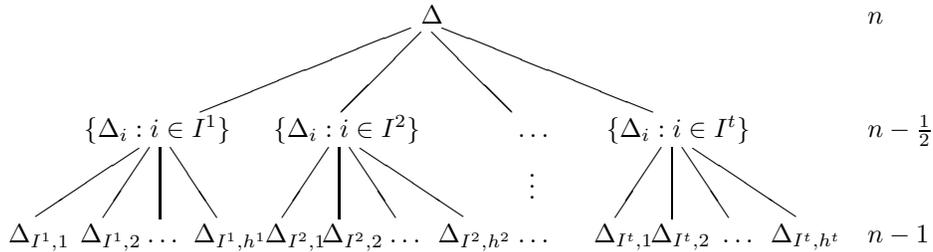
\begin{figure}[ht]
  \begin{picture}(420,140)
     \unitlength0.28mm
    \put(198,120){\small $\Delta$}
    \put(193,118){\line(-5,-2){100}}
    \put(198,116){\line(-1,-1){38}}
    \put(203,116){\line(1,-1){38}}
    \put(208,118){\line(5,-2){100}}
    \put(38,66){\small $\{\Delta_i: i\in I^1\}$}
    \put(128,66){\small $\{\Delta_i: i\in I^2\}$}
    \put(244,66){\small $\ldots$}
    \put(286,66){\small $\{\Delta_i: i\in I^t\}$}
    \put(64,61){\line(-3,-2){49}}
    \put(69,61){\line(-2,-3){22}}
    \put(74,61){\line(0,-1){34}}
    \put(79,61){\line(2,-3){22}}
    \put(2,16){\small $\Delta_{I^1,1}$}
    \put(36,16){\small $\Delta_{I^1,2}$}
    \put(68,16){\small $\ldots$}
    \put(90,16){\small $\Delta_{I^1,h^1}$}
    \put(154,61){\line(-2,-3){22}}
    \put(159,61){\line(0,-1){34}}
    \put(164,61){\line(2,-3){22}}
    \put(169,61){\line(3,-2){49}}
    \put(124,16){\small $\Delta_{I^2,1}$}
    \put(151,16){\small $\Delta_{I^2,2}$}
    \put(183,16){\small $\ldots$}
    \put(206,16){\small $\Delta_{I^2,h^2}$}
    \put(249,37){\small $\vdots$}
    \put(244,16){\small $\ldots$}
    \put(312,61){\line(-2,-3){22}}
    \put(317,61){\line(0,-1){34}}
    \put(322,61){\line(2,-3){22}}
    \put(327,61){\line(3,-2){49}}
    \put(280,16){\small $\Delta_{I^t,1}$}
    \put(307,16){\small $\Delta_{I^t,2}$}
    \put(341,16){\small $\ldots$}
    \put(364,16){\small $\Delta_{I^t,h^t}$}
    \put(410,120){\small $n$}
    \put(410,66){\small $n-\frac{1}{2}$}
    \put(410,16){\small $n-1$}
  \end{picture}
\caption{The C4 decomposition graph of a general $\Delta$}
\label{graph}
\end{figure}
\end{center}

\begin{center}
\begin{figure}
  \begin{picture}(420,330)
    \unitlength.28mm
    \put(170,310){\small $\{0,e_{1},e_{2},e_{3}\}$}
    \put(180,305){\line(-3,-2){52}}
    \put(214,305){\line(3,-2){52}}
    \put(90,260){\small $\{\{0,e_{1},e_{2}\},\{0\}\}$}
    \put(240,260){\small $\{\{0,e_{1}\},\{0,e_{2}\}\}$}
    \put(110,255){\line(-3,-2){52}}
    \put(140,255){\line(1,-2){17}}
    \put(254,255){\line(-1,-2){17}}
    \put(284,255){\line(3,-2){52}}
    \put(34,210){\small $\{0,e_{1},e_{2}\}$}
    \put(150.5,210){\small $\{0\}$}
    \put(222,210){\small $\{0,e_{1}\}$}
    \put(328,210){\small $\{0,e_{2}\}$}
    \put(51,205){\line(0,-1){35}}
    \put(158,205){\line(0,-1){35}}
    \put(236,205){\line(0,-1){35}}
    \put(343,205){\line(0,-1){35}}
    \put(25,160){\small $\{\{0,e_{1}\},\{0\}\}$}
    \put(143,160){\small $\{\{0\}\}$}
    \put(217,160){\small $\{\{0,e_{1}\}\}$}
    \put(322,160){\small $\{\{0\},\{0\}\}$}
    \put(40,155){\line(-1,-3){11.6}}
    \put(62,155){\line(2,-3){23}}
    \put(158,155){\line(0,-1){35}}
    \put(236,155){\line(0,-1){35}}
    \put(332,155){\line(-2,-3){23}}
    \put(354,155){\line(1,-3){11.6}}
    \put(12.5,110){\small $\{0,e_{1}\}$}
    \put(79,110){\small $\{0\}$}
    \put(150.5,110){\small $\{0\}$}
    \put(221.5,110){\small $\{0,e_{1}\}$}
    \put(301,110){\small $\{0\}$}
    \put(360,110){\small $\{0\}$}
    \put(25,105){\line(0,-1){35}}
    \put(86,105){\line(0,-1){35}}
    \put(158,105){\line(0,-1){35}}
    \put(236,105){\line(0,-1){35}}
    \put(308,105){\line(0,-1){35}}
    \put(369,105){\line(0,-1){35}}
    \put(3,60){\small $\{\{0\},\{0\}\}$}
    \put(71,60){\small $\{\{0\}\}$}
    \put(143,60){\small $\{\{0\}\}$}
    \put(214,60){\small $\{\{0\},\{0\}\}$}
    \put(293,60){\small $\{\{0\}\}$}
    \put(354,60){\small $\{\{0\}\}$}
    \put(16,55){\line(-1,-5){6.9}}
    \put(34,55){\line(1,-5){6.9}}
    \put(86,55){\line(0,-1){35}}
    \put(158,55){\line(0,-1){35}}
    \put(227,55){\line(-1,-5){6.9}}
    \put(245,55){\line(1,-5){6.9}}
    \put(308,55){\line(0,-1){35}}
    \put(369,55){\line(0,-1){35}}
    \put(1,10){\small $\{0\}$}
    \put(34.5,10){\small $\{0\}$}
    \put(79,10){\small $\{0\}$}
    \put(150,10){\small $\{0\}$}
    \put(212,10){\small $\{0\}$}
    \put(245.5,10){\small $\{0\}$}
    \put(301,10){\small $\{0\}$}
    \put(360.5,10){\small $\{0\}$}
    \put(410,310){\small $3$}
    \put(410,260){\small $2\frac{1}{2}$}
    \put(410,210){\small $2$}
    \put(410,160){\small $1\frac{1}{2}$}
    \put(410,110){\small $1$}
    \put(410,60){\small $\frac{1}{2}$}
    \put(410,10){\small $0$}
  \end{picture}
\caption{An iterated C4 decomposition graph}
\label{longgraph}
\end{figure}
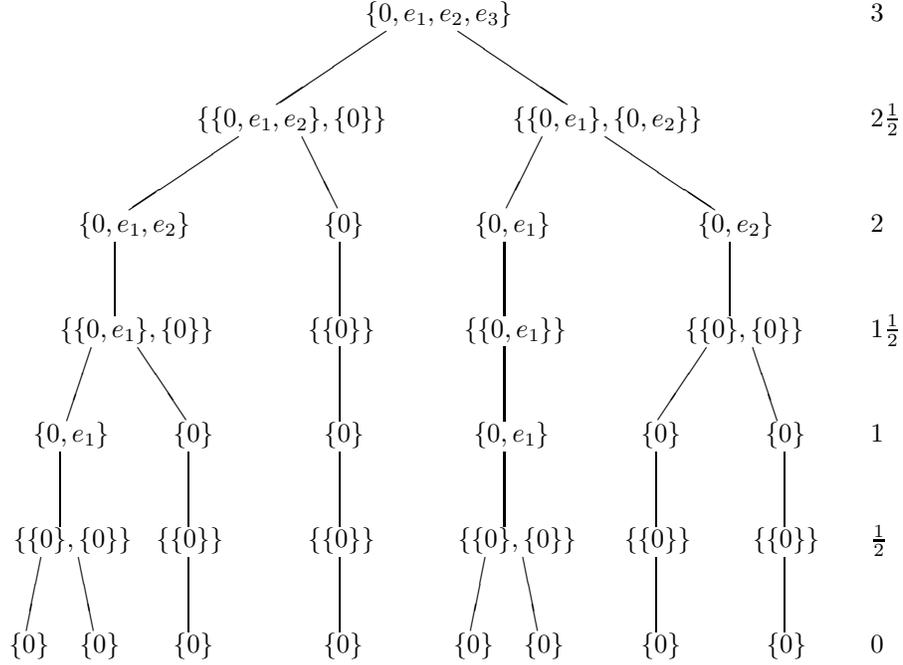
\end{center}

The iterated C4 decomposition graph encodes all possible ways to build $\Delta$ 
from $r = \#\Delta$ single points by means of iterated C4 addition of standard sets. 
For giving this observation a more concise meaning, 
we attach to the iterated C4 decomposition graph the following invariants. 

\begin{dfn}\label{numbers}
  Let $\Delta\in\mathcal{D}_{n}$, and let $\Gamma(\Delta)$ be its iterated C4 decomposition graph. 
  \begin{enumerate}
    \item[(i)] An {\it admissible subgraph of $\Gamma(\Delta)$} 
    is a subgraph $\Gamma^\prime$ of $\Gamma(\Delta)$ such that
      \begin{itemize}
        \item each node $x\in\Gamma^\prime$ of integer label $m$ 
          is connected with exactly one node of label $m-\frac{1}{2}$; and
        \item each node $x\in\Gamma^\prime$ of fractional label $m-\frac{1}{2}$ 
          is connected with all nodes of label $m-1$ from which there exists an edge to $x$ in $\Gamma(\Delta)$. 
      \end{itemize}
    \item[(ii)] We define the {\it C4 decomposition number} $d(\Delta)$ of the standard set $\Delta$ recursively as follows.
      \begin{itemize}
        \item If $n=1$ or $n=2$, we set $d(\Delta) := 1$. 
        \item If $n \geq 3$, proceed as follows. 
        We identify each C4 decomposition $\{\Delta_i: i \in I\}$ of $\Delta$ with its indexing set, $I$. 
        Let $I = I_1 \coprod \ldots \coprod I_{m(I)}$ be the C4 decomposition of the indexing set as in \eqref{ij}. 
        For all $j$, let $h_j := \#I_j$, and denote by $d(I_j)$ the C4 decomposition number of $\Delta_i$, for any $i \in I_j$. 
        Then we set 
        \begin{equation}\label{functionald}
          d(\Delta) := \sum_{I \in \mathcal{I}} \prod_{j = 1}^{m(I)} {d(I_j) + h_j - 1 \choose h_j} ,
        \end{equation}
        where the sum is taken over all C4 decompositions of $\Delta$, identified with their indexing sets, $I$. 
      \end{itemize}
  \end{enumerate}
\end{dfn}

At this point, we have to make a remark on the symmetry of the C4 decomposition graph $\Gamma(\Delta)$. 
Each node with label $n - \frac{1}{2}$ is uniquely determined, 
as each such node corresponds to a different C4 decomposition of $\Delta$. 
The nodes with label $n-1$, however, are not all uniquely determined. 
Indeed, let $\{\Delta_i: i \in I\}$ be a given node with label $n - \frac{1}{2}$, 
and let $I = I_1 \coprod \ldots \coprod I_m$ be the decomposition of $I$ as in \eqref{ij}. 
Then there is no way of distinguishing the nodes $\Delta_i$, for all $i \in I_j$, 
as they are all labeled by the same $\Delta_i$. 
Of course, nodes $\Delta_i$ and $\Delta_a$, for $i \in I_j$, $a \in I_b$, 
are different sets, and thus can be distinguished.

We may rephrase this observation in terms of a group action as follows: 
Given the C4 decomposition graph of $\Delta$, 
consider the subgraph with root $\{\Delta_i: i \in I\}$ and leaves $\Delta_i$, for all $i \in I$. 
Let $h := \#I$, and $h_j := \# I_j$, for $j = 1, \ldots, m$. 
We let the symmetric group $S_h$ act on the subgraph by permuting the leaves. 
Then the stabilizer of the graph is the subgroup
\begin{equation}\label{defg}
  G := S_{h_1} \times \ldots \times S_{h_m}
\end{equation}
of $S_h$, where each factor $S_{h_j}$ permutes the leaves $\Delta_i$, for $i \in I_j$. 

Similar assertions can be made about the symmetry of the C4 decomposition graph, 
as opposed to the above-described subgraph, and about the iterated C4 decomposition graph. 
We shall not state those assertions here, rather proving the following statement. 

\begin{lmm}\label{graphlemma}
  Let $\Delta\in\mathcal{D}_{n}$, and let $\Gamma(\Delta)$ be its iterated C4 decomposition graph. 
  Then the following natural numbers coincide:
  \begin{enumerate}
    \item[(i)] the number of admissible subgraphs of $\Gamma(\Delta)$, up to symmetry, 
    \item[(ii)] the number of possibilities to assemble $\Delta$ from $\#\Delta$ 
      points by iterated C4 addition of standard sets, and 
    \item[(iii)] the C4 decomposition number $d(\Delta)$. 
  \end{enumerate}
\end{lmm}

\begin{proof}
  First we prove by induction over $n$
  that each admissible subgraph $\Gamma^\prime$ of $\Gamma(\Delta)$ has precisely $\#\Delta$ leaves. 
  Let us take such a $\Gamma^\prime$. 
  The node $\Delta$ of $\Gamma^\prime$ is connected with exactly one node of label $n-\frac{1}{2}$, 
  i.e., with exactly one C4 decomposition $\{\Delta_i: i\in I\}$ of $\Delta$. 
  For all $i\in I$, let $\Gamma(\Delta_i)$ be the iterated C4 decomposition graph of $\Delta_i$. 
  Thus $\Gamma(\Delta_i)$ is simply the subgraph of $\Gamma(\Delta)$
  which consists of the root $\Delta_i$ and all nodes of labels 
  $n-1,n-\frac{3}{2},n-2,n-\frac{5}{2},\ldots,\frac{1}{2},0$ 
  which are connected with $\Delta_i$ by a sequence of edges. 
  It follows that the intersection $\Gamma^\prime\cap\Gamma(\Delta_i)$
  is an admissible subgraph of $\Gamma(\Delta_i)$. 
  By our induction hypothesis, $\Gamma^\prime \cap \Gamma(\Delta_i)$ has precisely $\#\Delta_i$ leaves. 
  The node $\{\Delta_i: i\in I\}$ of $\Gamma^\prime$ is connected with all nodes $\Delta_i$, $i\in I$, 
  of $\Gamma(\Delta)$. Therefore, the number of leaves of $\Gamma^\prime$ equals the sum
  over all $i\in I$ of the number of leaves of the graphs $\Gamma^\prime \cap \Gamma(\Delta_i)$.
  That sum is $\sum_{i \in I}\#\Delta_i=\#\Delta$, as was claimed.
  
  From what we have just proved, we see that each admissible subgraph $\Gamma^\prime$ of $\Gamma(\Delta)$
  encodes one possibility to assemble $\Delta$ from $\#\Delta$ points 
  (the leaves of $\Gamma^\prime$) by iterated C4 addition of standard sets.
  Moreover, it is clear that different admissible subgraphs, up to symmetry, 
  correspond to different possibilities to assemble $\Delta$ from $\#\Delta$ points. 
  Therefore, the numbers of (i) and (ii) coincide. 

  For the rest of the proof, we denote the number of admissible subgraphs of $\Gamma(\Delta)$ by $e(\Delta)$.
  We prove, again by induction over $n$, that $e(\Delta)=d(\Delta)$. 
  If $n=1$ or $n=2$, then clearly $e(\Delta)=d(\Delta)=1$. Therefore, take $n\geq3$. 
  Let us investigate what shape, up to symmetry, 
  an admissible subgraph $\Gamma^\prime$ of $\Gamma(\Delta)$ may take. 
  We start with the passage from the root, $\Delta$, 
  to nodes with label $n - \frac{1}{2}$. 
  We have seen above that $\Gamma^\prime$ contains 
  precisely one edge from $\Delta$ to one specific node $\{\Delta_i: i\in I\}$. 
  As for the passage from the nodes with label $n - \frac{1}{2}$ to nodes with label $n - 1$, 
  we have seen that in $\Gamma^\prime$, 
  the node $\{\Delta_i: i\in I\}$ is connected with each node $\Delta_i$, for $i\in I$, by an edge. 
  We have seen that for each $i \in I$, 
  the intersection $\Gamma^\prime \cap \Gamma(\Delta_i)$ is an admissible subgraph of $\Gamma(\Delta_i)$. 
  Let $I = I_1 \coprod \ldots \coprod I_m$ be the decomposition \eqref{ij} of the indexing set. 
  We have seen that as long as $i$ only runs through $I_j$, the nodes $\Delta_i$, 
  and therefore, also the subgraphs $\Gamma(\Delta_i)$, 
  and therefore, also the sets of admissible subgraphs of $\Gamma(\Delta_i)$, cannot be distinguished. 
  By assumption, there exist $e(\Delta_i)$ admissible subgraphs of $\Gamma(\Delta_i)$. 
  Define $e(I_j) := e(\Delta_i)$, for any $i \in I_j$. 
  As for the passage from nodes with label $n - 1$ to nodes with label $n - \frac{3}{2}$, 
  we have seen that the graph $\Gamma^\prime$ 
  contains precisely one edge from each node $\Delta_i$ to one C4 decomposition of $\Delta_i$. 
  Therefore, for all $i \in I_j$, 
  the intersection $\Gamma^\prime \cap \Gamma(\Delta_i)$ takes any value in a set of cardinality $e(\Gamma_i)$. 
  In other words, we are counting the number of multisets of cardinality $h_j = \# I_j$, 
  with elements taken from a set of cardinality $e(I_j)$. 
  As we are doing that independently for all $j = 1, \ldots, m(I)$, and also for all $I$ we started with, 
  we see that $e(\Delta)$ satisfies the functional equation 
  \begin{equation*}
    e(\Delta) = \sum_{I \in \mathcal{I}} \prod_{j = 1}^{m(I)} {e(I_j) + h_j - 1 \choose h_j} .
  \end{equation*}
  Here the sum is also taken over all C4 decompositions of $\Delta$, identified with their indexing sets, $I$. 
  We obtain the same functional equation as \eqref{functionald}. Hence the desired equality, $e(\Delta) = d(\Delta)$. 
\end{proof}

At this point, we wish to draw the reader's attention to the unpublished article \cite{gao}. 
The findings of that paper are closely related to our results of \cite{jpaa} --- in particular, 
the authors of \cite{gao} also exhibit what we call addition of standard sets. 
The graph in Figure 2 of that paper encodes the way in which the standard set $D(A) \subset \mathbb{N}^n$ 
of a given finite set $A \subset k^n$ is composed from $\#A$ points via iterated addition of standard sets. 
In contrast to that, our graphs in Figures \ref{longgraph} and \ref{widegraph}
encode all ways to decompose a given standard set $\Delta \subset \mathbb{N}^n$ 
into $\#\Delta$ points via iterated addition of standard sets. 

Take an arbitrary standard set $\Delta$ and its iterated C4 decomposition graph $\Gamma$. 
Let $\overline{\Gamma}$ be the graph which arises from $\Gamma$ by removing 
all nodes of labels $0,\frac{1}{2},1,1\frac{1}{2}$ and $2$. 
We call $\overline{\Gamma}$ the {\it truncated iterated C4 decomposition graph of $\Delta$}. 
We define admissible subgraphs of $\overline{\Gamma}$ by the same properties which characterize  
admissible subgraphs of $\Gamma$ in Definition \ref{numbers} above.
Then the number of admissible subgraphs of the truncated graph $\overline{\Gamma}$ 
equals the number of admissible subgraphs of $\Gamma$. 
Therefore, it suffices to consider the truncated graph for determining $d(\Delta)$. 
Figure \ref{widegraph} shows the truncated C4 decomposition graph of the standard set
$\Delta=\{0,e_{1},e_{2},e_{3},e_{4},2e_{4}\}\in\mathcal{D}_{4}$. 

\begin{center}
\begin{figure}
  \begin{picture}(350,365)
    \put(2,155){\small $\left\lbrace \begin{array}{c} 0, e_{1}, e_{2}, \\
    e_{3}, e_{4}, 2e_{4} \end{array} \right\rbrace$}
    \put(70,173){\line(1,4){26}}
    \put(70,163){\line(2,5){24}}
    \put(70,157){\line(1,0){24}}
    \put(70,151){\line(2,-5){24}}
    \put(70,141){\line(1,-4){26}}
    \put(99,283){\small $\left\lbrace \begin{array}{c} \{0,e_{1},e_{2},e_{3}\}, \\
    \{0\},\{0\} \end{array} \right\rbrace$}
    \put(103,219){\small $\left\lbrace \begin{array}{c} \{0,e_{1},e_{2}\}, \\
    \{0,e_{3}\},\{0\} \end{array} \right\rbrace$}
    \put(103,155){\small $\left\lbrace \begin{array}{c} \{0,e_{1},e_{3}\}, \\
    \{0,e_{2}\},\{0\} \end{array} \right\rbrace$}
    \put(103,90){\small $\left\lbrace \begin{array}{c} \{0,e_{2},e_{3}\}, \\
    \{0,e_{1}\},\{0\} \end{array} \right\rbrace$}
    \put(112,26){\small $\left\lbrace \begin{array}{c} \{0,e_{1}\}, \\
    \{0,e_{2}\}, \\
    \{0,e_{3}\} \end{array} \right\rbrace$}
    \put(184,295){\line(5,2){25}}
    \put(184,285){\line(1,0){24}}
    \put(184,275){\line(5,-2){25}}
    \put(184,231){\line(5,2){25}}
    \put(184,221){\line(1,0){24}}
    \put(184,211){\line(5,-2){25}}
    \put(184,167){\line(5,2){25}}
    \put(184,157){\line(1,0){24}}
    \put(184,147){\line(5,-2){25}}
    \put(184,103){\line(5,2){25}}
    \put(184,93){\line(1,0){24}}
    \put(184,83){\line(5,-2){25}}
    \put(184,39){\line(5,2){25}}
    \put(184,29){\line(1,0){24}}
    \put(184,19){\line(5,-2){25}}
    \put(211,303){\small $\{0,e_{1},e_{2},e_{3}\}$}
    \put(211,282.5){\small $\{0\}$}
    \put(211,262){\small $\{0\}$}
    \put(211,239){\small $\{0,e_{1},e_{2}\}$}
    \put(211,218.5){\small $\{0,e_{3}\}$}
    \put(211,198){\small $\{0\}$}
    \put(211,174){\small $\{0,e_{1},e_{3}\}$}
    \put(211,154.5){\small $\{0,e_{2}\}$}
    \put(211,134){\small $\{0\}$}
    \put(211,110){\small $\{0,e_{2},e_{3}\}$}
    \put(211,90.5){\small $\{0,e_{1}\}$}
    \put(211,70){\small $\{0\}$}
    \put(211,47){\small $\{0,e_{1}\}$}
    \put(211,26.5){\small $\{0,e_{2}\}$}
    \put(211,6){\small $\{0,e_{3}\}$}
    \put(268,310){\line(3,2){25}}
    \put(268,305){\line(1,0){24}}
    \put(229,285){\line(1,0){64}}
    \put(229,264.5){\line(1,0){64}}
    \put(256,241){\line(1,0){38}}
    \put(242,221){\line(1,0){52}}
    \put(229,200.5){\line(1,0){64}}
    \put(256,176){\line(1,0){38}}
    \put(242,157){\line(1,0){52}}
    \put(229,136.5){\line(1,0){64}}
    \put(256,112){\line(1,0){38}}
    \put(242,93){\line(1,0){52}}
    \put(229,72.5){\line(1,0){64}}
    \put(242,49.5){\line(1,0){52}}
    \put(242,29){\line(1,0){52}}
    \put(242,8.5){\line(1,0){52}}
    \put(295,325){\small $\{\{0,e_{1},e_{2}\},\{0\}\}$}
    \put(295,302.5){\small $\{\{0,e_{1}\},\{0,e_{2}\}\}$}
    \put(295,282.5){\small $\{\{0\}\}$}
    \put(295,262){\small $\{\{0\}\}$}
    \put(295,239){\small $\{\{0,e_{1},e_{2}\}\}$}
    \put(295,218.5){\small $\{\{0\},\{0\}\}$}
    \put(295,198){\small $\{\{0\}\}$}
    \put(295,174){\small $\{\{0,e_{1}\},\{0\}\}$}
    \put(295,154.5){\small $\{\{0,e_{2}\}\}$}
    \put(295,134){\small $\{\{0\}\}$}
    \put(295,110){\small $\{\{0,e_{2}\},\{0\}\}$}
    \put(295,90.5){\small $\{\{0,e_{1}\}\}$}
    \put(295,70){\small $\{\{0\}\}$}
    \put(295,47){\small $\{\{0,e_{1}\}\}$}
    \put(295,26.5){\small $\{\{0,e_{2}\}\}$}
    \put(295,6){\small $\{\{0\},\{0\}\}$}
    \put(40,350){\small $4$}
    \put(125,350){\small $3\frac{1}{2}$}
    \put(225,350){\small $3$}
    \put(320,350){\small $2\frac{1}{2}$}
  \end{picture}
\caption{A truncated iterated C4 decomposition graph}
\label{widegraph}
\end{figure}
\end{center}

%%%%%%%

\subsection{A generating function}\label{powerseries}

Let us briefly sketch how to express as a coefficient of a certain power series 
the number ways of a decomposing a standard set $\Delta \subset \mathbb{N}^n$ 
into a C4 sum of standard sets $\Delta_i \subset \mathbb{N}^{n-1}$. 
After that, we will make another remark on the C4 decomposition number of $\Delta$. 

For each standard set $\epsilon \subset q^n(\Delta) \subset \mathbb{N}^{n-1}$, 
we consider the vector 
\begin{equation*}
  A(\epsilon) := (A(\epsilon)_{\beta})_{\beta\in q^n(\Delta)} ,
\end{equation*}
indexed by the elements of $q^n(\Delta)$, where 
\begin{equation*}
  A(\epsilon)_{\beta} :=
  \begin{cases}
    1 & \text{ if }\beta\in\epsilon ,\\
    0 & \text{ if }\beta\notin\epsilon .
  \end{cases} 
\end{equation*}
We write $A$ for the matrix with columns $A(\epsilon)$, 
where $\epsilon$ runs through all standard sets contained in $q^n(\Delta)$. 
Moreover, let 
\begin{equation*}
  t := (t_{\beta})_{\beta \in q^n(\Delta)}
\end{equation*}
be a vector of indeterminates, also indexed by the elements of $q^n(\Delta)$. 
If $v\in\mathbb{N}^{q^n(\Delta)}$ is any vector of nonnegative integers, 
we write $t^v := \prod_{\beta \in q^n(\Delta)}t_{\beta}^{v_{\beta}}$. 

Consider the power series
\begin{equation*}
  g := \prod_{\epsilon \subset q^n(\Delta)}\frac{1}{1 - t^{A(\epsilon)}}
  = \sum_{v \in \mathbb{N}^{q^n(\Delta)}}\Phi_{A}(v)t^v .
\end{equation*}
This is the generating function of the {\it vector partition function} $\Phi_{A}:\mathbb{N}^{q^n(\Delta)}\to\mathbb{N}$
associated to the matrix $A$ (see \cite{vector} for references on vector partition functions). 
Upon considering the particular vector
\begin{equation*}
  v_{\Delta} := (\#((q_{n})^{-1}(\beta) \cap \Delta))_{\beta\in q^n(\Delta)} ,
\end{equation*}
which encodes shape of the standard set $\Delta$, 
we see that the coefficient $\Phi_{A}(v_{\Delta})$ equals 
the number of C4 decompositions of the standard set $\Delta \subset \mathbb{N}^n$
into a C4 sum of standard sets $\epsilon \subset \mathbb{N}^{n-1}$. 
Moreover, we see that for a given $v\in\mathbb{N}^{q^n(\Delta)}$, 
the coefficient $\Phi_{A}(v)$ vanishes unless the set 
\begin{equation*}
  \cup_{\beta\in q^n(\Delta)}\{(\beta,u): 0\leq u<v\} \subset \mathbb{N}^n
\end{equation*}
encoded by $v$ is a standard set. 

The power series $g$, however, only encodes the C4 decomposition number of $\Delta$, 
and not its iterated C4 decomposition number. 
The definition of a power series encoding the iterated C4 decomposition number appears much harder to find;
we shall not pursue a search for it in the present paper. 

%%%%%%%

\subsection{The original example}\label{original}

As was mentioned in the Introduction, $\Hi^{\prec\Delta}_{S/k}$ is an affine scheme. 
We denote its coordinate ring by $R^\Delta$. 
This ring is not hard to implement in a computer algebra systems such as {\it Macaulay2} (see \cite{M2}). 
For doing so, one can use the presentation of $R^\Delta$ given in \cite{strata} (cf. the discussion in the next section). 
Upon carrying this out for any field $k$ and the fixed standard set 
\begin{equation*}
  \Delta=\{0,e_{1},e_{2},e_{3}\}\in\mathcal{D}_{3} ,
\end{equation*}
one finds that the number of irreducible components of $\Hi^{\prec\Delta}_{S/k}$ equals $2$. 
Upon taking another look at Figure \ref{longgraph}, 
we see that the C4 decomposition number of that particular standard set $\Delta$ also equals $2$. 

This example was the starting point for the research presented here. 
Bernd Sturmfels observed that for this particular $\Delta$,
the number of irreducible components of $\Hi^{\prec\Delta}_{S/k}$ equals $d(\Delta)$. 
He asked whether or not this equality holds for general $\Delta$. 
We will prove that such an equality holds for a certain subscheme $\Hi^{\prec\Delta,\et}_{S/k}$ of $\Hi^{\prec\Delta}_{S/k}$, 
which we will define in Section \ref{reducedmoduli}. 
Not only will we be counting the irreducible components of that subscheme, 
we will also be counting its connected components, 
we will show that it is equidimensional, and we will compute its dimension.
The respective assertions are stated in Theorem \ref{mainthm}. 
At the very end of our paper, in Section \ref{negative}, we will return to the full scheme $\Hi^{\prec\Delta}_{S/k}$. 
We will prove that given any $\Delta$, 
the scheme $\Hi^{\prec\Delta}_{S/k}$ has more than $d(\Delta)$ irreducible components. 
In this sense, Sturmfels' question is answered in the negative for $\Hi^{\prec\Delta}_{S/k}$. 
Moreover, we will also answer Sturmfels' question in the negative for a certain scheme $\mathscr{G}^{\prec \Delta}_{S/k}$
lying in between $\Hi^{\prec\Delta,\et}_{S/k}$ and $\Hi^{\prec\Delta}_{S/k}$. 

%%%%%%%%%%%%%%%%%%%%%%%%%%%%%%%%%%%%%%%%%%%%%%%%%%%%%%%%

\section{The affine schemes under study}\label{affines}

We start this section with a review of the functor $\HHi^{\prec\Delta}_{S/k}$
we introduced in Section \ref{intro}. 
After that, we shall introduce the auxiliary schemes we are going to use. 
All schemes appearing in this section are affine; we shall describe their coordinate rings. 
For defining our auxiliary schemes, we will use the polynomial ring $\overline{S} := k[\overline{x}] := k[x_{1},\ldots,x_{n-1}]$. 
All schemes we introduce in this section are well defined if $\prec$ is an arbitrary term order on $S$ 
(thus inducing a term order on $\overline{S}$). 
However, the definitions from Sections \ref{hilbtimesa}--\ref{alldecompositions}
are crucial to our proof of Sturmfels' conjecture. 
Therefore, the schemes we define there are useful only when $\prec$ is the lexicographic order,
which is the order of choice in the present paper. 

%%%%%%%

\subsection{Gr\"obner strata}\label{groebnerscheme}

Take $\Delta\in\mathcal{D}_{n}$. 
The scheme $\Hi^{\prec\Delta}_{S/k}$, which we introduced in Section \ref{intro},
has the coordinate ring
\begin{equation}\label{generators}
  R^{\prec\Delta} := k[T_{\alpha,\beta}: \alpha\in N\cup N^{(1)}, \beta\in\Delta]/I^{\prec\Delta} ,
\end{equation}
where $N$ is a finite or infinite standard set in $\mathbb{N}^n$ containing $\Delta$,
$N^{(1)}=(\cup_{i=1}^n(N+e_i)) \setminus N$ is the {\it border} of $N$ and 
$I^{\prec\Delta}$ is a certain ideal in the indicated polynomial ring with variables $T_{\alpha,\beta}$
(see \cite{strata} for the definition of $I^{\prec\Delta}$). 

As was mentioned in Section \ref{intro}, the scheme $\Hi^{\prec\Delta}_{S/k}$
represents the {\it Gr\"obner functor} $\HHi^{\prec\Delta}_{S/k}:(k{\rm-Alg})\to({\rm Sets})$,
which sends a $k$-algebra $B$ to the set of all reduced Gr\"obner bases with standard set $\Delta$. 
The universal object over $\Hi^{\prec\Delta}_{S/k}$ is the affine scheme 
\begin{equation*}
  U^{\prec\Delta} := {\rm Spec}\,S^{\prec\Delta}
\end{equation*}
whose coordinate ring is
\begin{equation*}
  S^{\prec\Delta} := S \otimes_k R^{\prec\Delta} / J^{\prec\Delta} ,
\end{equation*}
where we factor out the ideal 
\begin{equation}\label{relationsj}
  J^{\prec\Delta} := \langle x^\alpha-\sum_{\beta\in\Delta}T_{\alpha,\beta}x^\beta: 
  \alpha\in\mathscr{C}(\Delta) \rangle .
\end{equation}
Here $\mathscr{C}(\Delta)$ is the set of corners of $\Delta$ that we introduced in Section \ref{intro}. 
Note that the ideal $J^{\prec\Delta}$ in $S \otimes_k R^{\prec\Delta}$ is monic with standard set $\Delta$. 
The universal morphism $U^{\prec\Delta}\to\Hi^{\prec\Delta}_{S/k}$ is the morphism 
corresponding to the canonical homomorphism of rings $R^{\prec\Delta}\to S^{\prec\Delta}$.

We shall frequently use equivalent formulations of the functor $\HHi^{\prec\Delta}_{S/k}$. 
More precisely, we shall identify the following objects:
\begin{itemize}
  \item a surjective $B$-algebra homomorphism $\phi:S \otimes_k B \to Q$ such that $\ker\,\phi$
  is a monic ideal with standard set $\Delta$;
  \item a monic ideal $J \subset S \otimes_k B$ with standard set $\Delta$;
  \item a reduced Gr\"obner basis of shape 
  $\lbrace x^\alpha+\sum_{\beta\in\Delta,\beta\prec\alpha}d_{\alpha,\beta}x^\beta: \alpha\in\mathscr{C}(\Delta) \rbrace$ 
  in $S \otimes_k B$;
  \item a Gr\"obner basis of shape 
  $\lbrace x^\alpha+\sum_{\beta\in\Delta,\beta\prec\alpha}d_{\alpha,\beta}x^\beta:  
  \alpha\in N\cup N^{(1)} \setminus \Delta \rbrace$ 
  in $S \otimes_k B$, where $N$ is as above;
  \item a homomorphism $g:R^{\prec\Delta}\to B$; 
  \item the extended homomorphism ${\rm id} \otimes g: S \otimes_k R^{\prec\Delta} \to S \otimes_k B$; and
  \item the ideal $\langle ({\rm id} \otimes g)(J^{\prec\Delta}) \rangle \subset S \otimes_k B$. 
\end{itemize}
The transitions between the bulleted items are established by the identities 
\begin{equation*}
  \ker\phi = J = \langle x^\alpha+\sum_{\beta\in\Delta}d_{\alpha,\beta}x^\beta: \beta\in\mathscr{C}(\Delta) \rangle
  = \langle ({\rm id} \otimes g)(J^{\prec\Delta}) \rangle .
\end{equation*}
From this we get a number of reformulations of the definition of $\HHi^{\prec\Delta}_{S/k}$:
This functor sends a $k$-algebra $B$ to the set of all surjective $B$-algebra homomorphisms 
$\phi: S \otimes_k B \to Q$ such that $\ker\,\phi$ is a monic ideal with standard set $\Delta$;
or else, it sends a $k$-algebra $B$ to the set of all 
monic ideals $J \subset S \otimes_k B$ with standard set $\Delta$;
or else, the analogous assertion for each other item.

The identifications also explain the fact that $U^{\prec\Delta}$ is the 
universal object over $\Hi^{\prec\Delta}_{S/k}$.
On the one hand, an element of $\HHi^{\prec\Delta}_{S/k}(B)$ is a quotient $Q := S \otimes_k B/J$, 
which leads to the morphism $p: X := {\rm Spec}\,Q \to {\rm Spec}\,B$
restricting the projection 
$p:\mathbb{A}^n_{B}=\mathbb{A}^n_k\times_{{\rm Spec}\,k}{\rm Spec}\,B\to{\rm Spec}\,B$.
On the other hand, an element of $\HHi^{\prec\Delta}_{S/k}(B)$
is a morphism $\psi\in{\rm Hom}({\rm Spec}\,B,\Hi^{\prec\Delta}_{S/k})$. 
Then $p$ and $\psi$ correspond to each other via the cartesian diagram
\begin{equation}\label{cdschemes}
  \xymatrix{ 
    X \ar[d]_p \ar[r] & U^{\prec\Delta} \ar[d] \\ 
    {\rm Spec}\,B \ar[r]^\psi & \Hi^{\prec\Delta}_{S/k} .
  }
\end{equation}
Equivalently, we can rephrase the universal property of $\Hi^{\prec\Delta}_{S/k}$
purely in terms of its coordinate ring $R^{\prec\Delta}$ as follows. 

\begin{lmm}\label{universalring}
  Let $B$ be a $k$-algebra and $\phi: S \otimes_k B \to Q$ a $B$-algebra homomorphism. 
  Then $\phi$ lies in $\HHi^{\prec\Delta}_{S/k}(B)$ if, and only if, there exists a $k$-algebra homomorphism
  $f:R^{\prec\Delta}\to B$ such that the diagram
  \begin{equation}\label{cdrings}
    \xymatrix{ 
      R^{\prec\Delta} \ar[d]_f \ar[r] & S^{\prec\Delta} \ar[d] \\ 
      B \ar[r]^\phi & Q 
    }
  \end{equation}
  is co-cartesian. $f$ is unique if it exists. 
\end{lmm}

\begin{proof}
  As all schemes in the cartesian diagram \eqref{cdschemes} are affine, 
  that diagram translates to a co-cartesian diagram of the corresponding coordinate rings, 
  which is precisely the diagram \eqref{cdrings}. 
\end{proof}

We call the scheme $\Hi^{\prec\Delta}_{S/k}$ a {\it Gr\"obner stratum}. 
The terminology is motivated by the fact that this scheme is a locally closed subscheme of the Hilbert scheme of $r$ points, 
$\Hi^{r}_{S/k}$, where $r=\#\Delta$. 
Gr\"obner strata $\Hi^{\prec\Delta}_{S/k}$, and related objects, have been studied by a number of authors, 
see \cite{robbianobbgb} and references therein. 
The authors of \cite{altmannbernd}, \cite{evain1}, or \cite{evain2}
refer to Gr\"obner strata as {\it Schubert schemes}, or {\it Schubert cells}. 
Their terminology is motivated by the analogy of the inclusion 
$\Hi^{\prec\Delta}_{S/k} \subset \Hi^{r}_{S/k}$ to the inclusion of a Schubert cell in the Grassmannian 
in the case where $\Delta$ is a subset of the standard basis $\{e_{1},\ldots,e_{n}\} \subset \mathbb{N}^n$, 
augmented by $0\in\mathbb{N}^n$. 

%%%%%%%

\subsection{The same in a lower dimension}\label{lowerdim}

We can replace the given $\Delta\in\mathcal{D}_{n}$ by some $\Delta_i\in\mathcal{D}_{n-1}$, 
and accordingly, replace $S$ by $\overline{S}$. 
We get an affine scheme $\Hi^{\prec\Delta_i}_{\overline{S}/k}$, 
whose coordinate ring we denote by $R^{\prec\Delta_i}$. 
The scheme $\Hi^{\prec\Delta_i}_{\overline{S}/k}$ represents a functor 
$\HHi^{\prec\Delta_i}_{\overline{S}/k}$, 
which has descriptions analogous to the descriptions of $\HHi^{\prec\Delta}_{S/k}$ we gave above. 
The universal object over $\Hi^{\prec\Delta_i}_{\overline{S}/k}$ is the scheme
\begin{equation*}
  U^{\prec\Delta_i} := {\rm Spec}\,S^{\prec\Delta_i} ,
\end{equation*}
where 
\begin{equation*}
  S^{\prec\Delta_i} := \overline{S} \otimes_k R^{\prec\Delta_i}/J^{\prec\Delta_i} ,
\end{equation*}
and the generators of $J^{\prec\Delta_i}$ are entirely analogous to those in \eqref{relationsj} above. In particular, 
the ideal $J^{\prec\Delta_i}$ in $\overline{S} \otimes_k R^{\prec\Delta_i}$ is monic with standard set $\Delta_i$. 

%%%%%%%

\subsection{An additional parameter}\label{hilbtimesa}

Consider the scheme 
\begin{equation*}
  \Hi^{\prec\Delta_i}_{\overline{S}/k}\times\mathbb{A}^1_k ,
\end{equation*}
where the product is taken over ${\rm Spec}\,k$. 
Upon using the coordinate $y_i$ on the second factor, we see that the coordinate ring of that scheme is
\begin{equation*}
  R^{\prime\prec\Delta_i} := R^{\prec\Delta_i}[y_i] .
\end{equation*}
That scheme represents the functor sending a $k$-algebra $B$ to the set of all pairs 
$(J_i,b_i)$, where $J_i \subset \overline{S} \otimes_k B$ is a monic ideal with standard set $\Delta_i$ 
and $b_i$ is an element of $B$. Equivalently, the represented functor is
\begin{equation}\label{1st}
  \begin{split}
    h_{\Hi^{\prec\Delta_i}_{\overline{S}/k}\times\mathbb{A}^1_k}:(k{\rm-Alg})&\to({\rm Sets})\\
    B&\mapsto 
    \left\lbrace
    \begin{array}{c}
      \text{pairs of ideals } (J_i, \langle x_{n}-b_i \rangle), \\
      \text{s.t. } J_i \subset \overline{S} \otimes_k B \text{ is a monic ideal} \\
      \text{ with standard set }\Delta_i, \\
      \text{ and } \langle x_{n}-b_i \rangle \subset S \otimes_k B
    \end{array}
    \right\rbrace .
  \end{split}
\end{equation}
As the ideal $J_i$ appearing in the above pair lives in the ring $\overline{S} \otimes_k B$, 
and the ideal $\langle x_{n}-b_i \rangle \subset S \otimes_k B$ is generated by a polynomial involving only $x_n$, 
we may identify the above functor with 
\begin{equation}\label{2nd}
  \begin{split}
    h_{\Hi^{\prec\Delta_i}_{\overline{S}/k}\times\mathbb{A}^1_k}:(k{\rm-Alg})&\to({\rm Sets})\\
    B&\mapsto 
    \left\lbrace
    \begin{array}{c}
      \text{ideals } \langle J_i, x_{n}-b_i \rangle \subset S \otimes_k B, \\
      \text{s.t. } J_i \subset \overline{S} \otimes_k B \text{ is a monic ideal} \\
      \text{ with standard set }\Delta_i \\
      \text{ and } b_i \in B
    \end{array}
    \right\rbrace .
  \end{split}
\end{equation}
Viewed in this way, the universal object over $\Hi^{\prec\Delta_i}_{\overline{S}/k}\times\mathbb{A}^1_k$ 
is the affine scheme 
\begin{equation*}
  U^{\prec\Delta_i}\times\mathbb{A}^1_k = {\rm Spec}\,S^{\prime\prec\Delta_i} ,
\end{equation*}
where 
\begin{equation*}
  S^{\prime\prec\Delta_i} := S \otimes_k R^{\prime\prec\Delta_i}/J^{\prime\prec\Delta_i}
\end{equation*}
and 
\begin{equation*}
  J^{\prime\prec\Delta_i} := \langle J^{\prec\Delta_i} \rangle + \langle x_{n}-y_i \rangle 
  \subset S \otimes_k R^{\prime\prec\Delta_i} .
\end{equation*}
We will return to the transition between the two descriptions of the functor later, in Section \ref{proofimmersion}. 

%%%%%%%

\subsection{A product of the former, minus bad points}\label{minusbadpoints}

Let $\{\Delta_i: i\in I\}$ be a C4 decomposition of $\Delta$, which we identify with its indexing set $I$. We define 
\begin{equation}\label{full}
  \widehat{Y}^I := \Bigl( \prod_{i \in I} \bigl( \Hi^{\prec\Delta_i}_{\overline{S}/k}\times\mathbb{A}^1_k \bigr) \Bigr) 
  \setminus \Lambda .
\end{equation}
As above, the fibered product is taken over ${\rm Spec}\,k$, 
and we denote the coordinate of the second factor of 
$\Hi^{\prec\Delta_i}_{\overline{S}/k}\times\mathbb{A}^1_k$ by $y_i$. 
The closed subscheme $\Lambda$ which we remove is the {\it large diagonal},
\begin{equation}\label{large}
  \Lambda := \cup_{i\neq j\in I}\mathbb{V}(y_i-y_{j}) .
\end{equation}
Therefore, the coordinate ring of $\widehat{Y}^I$ is
\begin{equation}\label{rprimeprime}
  R^{\prime\prec I} := (\otimes_{i \in I}R^{\prime\prec\Delta_i})[\prod_{i\neq j\in I}\frac{1}{y_i-y_{j}}] .
\end{equation}
The functor represented by $\widehat{Y}^I$ is 
\begin{equation*}
  \begin{split}
    h_{\widehat{Y}^I}:(k{\rm-Alg})&\to({\rm Sets})\\
    B&\mapsto
    \left\lbrace 
    \begin{array}{c}
      \text{tuples } 
      \bigl( \langle J_i \rangle + \langle x_{n}-b_i \rangle \bigr)_{i \in I} \\
      \text{ of ideals in } S \otimes_k B \\
      \text{s.t. } J_i \subset \overline{S} \otimes_k B \text{ is monic with standard set }\Delta_i \\
      \text{ and } b_i-b_{j}\in B^*\text{ for }i\neq j
    \end{array}
    \right\rbrace .
  \end{split}
\end{equation*}
Note that all differences $y_i-y_{j}$ are invertible in the ring $R^{\prime\prec I}$, 
which corresponds to cutting out $\Lambda$ from the product in \eqref{full}.
Invertibility of $y_i-y_{j}$ will be crucial in Section \ref{connectfour} below. 

%%%%%%%

\subsection{Symmetry}\label{mod}

Remember that a C4 decomposition $\{\Delta_i: i \in I\}$ is a multiset, 
meaning that the same $\Delta_i$ may occur multiple times. 
We write the indexing set of the C4 decomposition as a coproduct $I = I_1 \coprod \ldots \coprod I_m$ as in \eqref{ij}. 
As above, we write $h_j = \# I_j$, and consider the group
$G = S_{h_1} \times \ldots \times S_{h_m}$ of \eqref{defg}. 
We let $G$ act on the product $\prod_{i \in I} \Hi^{\prec \Delta_i}_{\overline{S}/k} \times \mathbb{A}^1_k$ 
by letting each $S_{h_j}$ permute the $h_j$ factors $\Hi^{\prec \Delta_i}_{\overline{S}/k} \times \mathbb{A}^1_k$ 
indexed by the multiset $\{\Delta_i: i \in I_j\}$. 
This action clearly induces an action of $G$ on $\widehat{Y}^I$. We consider the {\it geometric quotient}
\begin{equation}\label{quotient}
  Y^I := \widehat{Y}^I / G ,
\end{equation}
As $\widehat{Y}^I$ is quasi-projective and $G$ is finite, the quotient $Y^I$ exists as a scheme.
In fact, $Y^I$ is an affine scheme,
whose coordinate ring arises from the coordinate ring of $\widehat{Y}^I$ by taking invariants. 
(This follows from \cite{mumford}, \S12, Theorem 1; or \cite{bertin}, Section 1.3, Lemma 1.3 and Proposition 1.4; 
or \cite{dolgachevit}, Example 6.1.
In what follows, we will encounter the same situation several times: 
A finite group will be acting on a quasi-projective scheme. 
In those situations, the existence of the quotient will always follow from the cited results.)

Here is the explicit description of the action, and of the coordinate ring of $Y^I$. 
Denote by $T_{\alpha,\beta}^{(i)}$ the generators of $R^{\prime\prec \Delta_i}$. 
We let $G$ act on $R^{\prime\prec I}$ setting, for each factor $S_{h_j}$ of $G$ and for each $\sigma \in S_{h_j}$, 
\begin{equation}\label{action}
  \begin{split}
    \sigma(T_{\alpha,\beta}^{(i)}) & := T_{\alpha,\beta}^{(\sigma(i))} , \\
    \sigma(y_i) & := y_{\sigma(i)} .
  \end{split}
\end{equation}
Upon denoting the invariant ring of that action by $(R^{\prime\prec I})^G$, we obtain that 
\begin{equation*}
  Y^I = {\rm Spec}\, (R^{\prime\prec I})^G .
\end{equation*}
The functor represented by $Y^I$ is 
\begin{equation*}
  \begin{split}
    h_{Y^I}:(k{\rm-Alg})&\to({\rm Sets})\\
    B&\mapsto
    \left\lbrace 
    \begin{array}{c}
      \text{sets }
      \left\lbrace \langle J_i \rangle + \langle x_{n}-b_i \rangle: i \in I \right\rbrace \\
      \text{ of ideals in } S \otimes_k B \\
      \text{s.t. } J_i \subset \overline{S} \otimes_k B \text{ is monic with standard set }\Delta_i \\
      \text{ and } b_i-b_{j}\in B^*\text{ for }i\neq j
    \end{array}
    \right\rbrace .
  \end{split}
\end{equation*}

%%%%%%%

\subsection{The disjoint sum over all C4 decompositions}\label{alldecompositions}

Finally, for a fixed $\Delta\in\mathcal{D}_{n}$, let $\mathcal{I}$ be the set of all C4 decompositions 
$\{\Delta_i: i\in I\}$ of $\Delta$ by elements $\Delta_i$ of $\mathcal{D}_{n-1}$. 
Again we use the notation $I\in\mathcal{I}$ when we mean the whole C4 decomposition $\{\Delta_i: i\in I\}$. 
Accordingly, we denote the disjoint sum of all $\widehat{Y}^I$ by 
\begin{equation*}
  \widehat{Y}^\Delta := \coprod_{I \in \mathcal{I}}\widehat{Y}^I .
\end{equation*}
This is an affine scheme whose coordinate ring is
\begin{equation*}
  R^{\prime\prec\Delta} := \prod_{I\in\mathcal{I}}R^{\prime\prec I} .
\end{equation*}
The scheme $\widehat{Y}^\Delta$ is the most important auxiliary object of this paper. 

As we explicitly know all functors $h_{\widehat{Y}^I}$,
we also have the functor 
\begin{equation*}
  \begin{split}
    h_{\widehat{Y}^\Delta}:(k{\rm-Alg})&\to({\rm Sets})
  \end{split}
\end{equation*}
at hand. Indeed, in the category of schemes, the coproduct is given by the disjoint sum. 
It therefore suffices to know the functors $h_{\widehat{Y}^{I}}$ for knowing the whole functor $h_{\widehat{Y}^\Delta}$. 
Explicitly, if $B$ is a $k$-algebra $B$ having no nontrivial idempotents (e.g. a domain, or even a field), 
then $h_{\widehat{Y}^\Delta}$ sends $B$ to the set of all tuples of ideals 
$\bigl( \langle J_i \rangle + \langle x_{n}-b_i \rangle \bigr)_{i \in I}$, indexed by some $I\in\mathcal{I}$, 
such that $J_i \subset \overline{S} \otimes_k B$ is monic with standard set $\Delta_i$ and $b_i-b_{j}\in B^*$ for $i\neq j$.
(We will need that description in the forthcoming sections. 
For verifying its correctness, we consider $e\in B$, the only nonzero element such that $e^2=e$, 
and the identity elements $e_{I}\in R^{\prime\prec I}$. 
A homomorphism $g:\prod_{I}R^{\prime\prec I}\to B$ sends $\sum_{I}e_{I}$ to $e$. 
From the equation $g(e_{I})^2=g(e_{I}^2)=g(e_{I})$, we see that $g(e_{I})$ is either $0$ or $e$. 
For $I\neq I^\prime$, we have $0=g(0)=g(e_{I}e_{I^\prime})=g(e_{I})g(e_{I^\prime})$.
Therefore only one $g(e_{I})$ is $e$, and all other $g(e_{I^\prime})$ are $0$. 
Hence $g$ is really a homomorphism $g:R^{\prime\prec I}\to B$ 
from one factor of $R^{\prime\prec\Delta}$ to $B$, and the zero map on all other factors.)
Note 
\begin{itemize}
  \item the difference between the functors $h_{\widehat{Y}^I}$ and $h_{\widehat{Y}^\Delta}$: 
    In the first, the tuple of ideals is indexed by a fixed set $I$, 
    whereas in the second, the tuple of ideals is indexed by some $I\in\mathcal{I}$; 
  \item that for all $i\in I$, the ideal $J_i$ is just 
  $\langle ({\rm id} \otimes g^\prime)(J^{\prec\Delta_i}) \rangle \subset \overline{S} \otimes_k B$, 
    where $g^\prime$ is the composition
    \begin{equation*}
      \xymatrix{
        g^\prime: R^{\prec \Delta_i} \ar[r] & R^{\prime\prec \Delta_i} \ar[r] & R^{\prime\prec I} \ar[r]^g & B ;
      }
    \end{equation*}
  \item that equivalently, for all $i\in I$, we have 
    $\langle J_i \rangle + \langle x_{n}-b_i \rangle = 
    \langle ({\rm id} \otimes g^{\prime\prime})(J^{\prime\prec\Delta_i}) \rangle$, where 
    where $g^{\prime\prime}$ is the composition
    \begin{equation*}
      \xymatrix{
        g^{\prime\prime}: R^{\prime\prec\Delta_i} \ar[r] & R^{\prime\prec I} \ar[r]^g & B ;
      }
    \end{equation*}
  \item that the scheme $\widehat{Y}^\Delta$ also depends on the ring $k$, 
    so we should actually denote it by $\widehat{Y}^\Delta_k$; and analogously, for $Y^\Delta$. 
    However, we shall not need the dependence on $k$ before Section \ref{theresult}. 
\end{itemize}

Equally as important as $\widehat{Y}^\Delta$ is the scheme 
\begin{equation*}
  Y^\Delta := \coprod_{I \in \mathcal{I}} Y^I = \widehat{Y}^\Delta / G,
\end{equation*}
whose coordinate ring is $(R^{\prime\prec\Delta})^G = \prod_{I\in\mathcal{I}} (R^{\prime\prec I})^G$. 
The transition from the functor $h_{Y^I}$ to the functor $h_{Y^\Delta}$ 
is analogous to the transition from the functor $h_{\widehat{Y}^I}$ to the functor $h_{\widehat{Y}^\Delta}$: 
On a $k$-algebra $B$ having no nontrivial idempotents, 
$h_{Y^\Delta}(B)$ is the set of all sets of ideals 
$\left\lbrace \langle J_i \rangle + \langle x_{n}-b_i \rangle: i \in I \right\rbrace$, indexed by some $I\in\mathcal{I}$, 
such that $J_i \subset \overline{S} \otimes_k B$ is monic with standard set $\Delta_i$ and $b_i-b_{j}\in B^*$ for $i\neq j$.

%%%%%%%

\subsection{Finiteness}

We conclude this section with a trivial but important observation on the schemes we introduced thus far. 

\begin{lmm}
  The schemes $\Hi^{\prec\Delta}_{S/k}$, $U^{\prec\Delta}$, 
  $\widehat{Y}^I$, $Y^I$, $\widehat{Y}^\Delta$ and $Y^\Delta$ are of finite type over ${\rm Spec}\,k$. 
\end{lmm}

\begin{proof}
  Remember that the set $N$ in \eqref{generators} can be chosen finite, e.g. $N=\Delta$. 
  Therefore the affine scheme $\Hi^{\prec\Delta}_{S/k}$ is of finite type over ${\rm Spec}\,k$.
  As for the other schemes in question, the assertion follows from that fact; 
  from the definitions; and from the Gordan-Hilbert Theorem. 
  (The latter is Theorem 3.1 in \cite{dolgachevit}, where it is only stated over fields $k$. 
  However, our schemes are defined over $\mathbb{Z}$, from which ring we may pass to $\mathbb{Q}$, 
  apply the cited theorem, observe that that the result lives over $\mathbb{Z}$, and pass to $k$.)
\end{proof}

%%%%%%%%%%%%%%%%%%%%%%%%%%%%%%%%%%%%%%%%%%%%%%%%%%%%%%%% 

\section{The Connect Four morphism}\label{connectfour}

The aim of this section is to define a morphism of schemes
\begin{equation*}
  \tau:Y^{\Delta}\to\Hi^{\prec\Delta}_{S/k}
\end{equation*}
which is a ``universal form'' of the proof of the main Theorem of \cite{jpaa}. 
Our strategy is to first define a morphism 
\begin{equation*}
  \widehat{\tau}: \widehat{Y}^\Delta \to \Hi^{\prec\Delta}_{S/k}
\end{equation*}
and then to show that $\widehat{\tau}$ is invariant under the action of $G$, thus defining the desired morphism $\tau$. 
The definitions of these morphisms will use some specific properties 
of the term order $\prec$ which we use on $S$, and more generally, on $S \otimes_k B$. 
Remember that $\prec$ is the lexicographic order such that $x_1 \succ \ldots \succ x_n$. 

%%%%%%%

\subsection{Functorial description of the morphism}

Recall from Section \ref{combinatorics} that the main Theorem of \cite{jpaa} states that if $k$ is a field 
and $A \subset k^n$ a finite set, then $D(A)=\sum_{\lambda\in k}D(A_{\lambda})$. 
Its proof can briefly be described as follows.
We fix a $\lambda\in k$, understand $A_{\lambda}$ to be a finite subset of $k^{n-1}$, 
and consider the Gr\"obner basis of the ideal $\overline{I}(A_{\lambda}) \subset \overline{S}$. 
We append the polynomial $x_{n}-\lambda$ to that Gr\"obner basis. 
What we get is the Gr\"obner basis of $I(A_{\lambda}) \subset S$, 
where  we consider $A_{\lambda}$ as a subset of $k^{n-1}\times\{\lambda\} \subset k^n$. 
Then we use a somewhat involved method, based on interpolation and reduction, 
to construct the Gr\"obner basis of the ideal $I(A)=\cap_{\lambda\in k}I(A_{\lambda})$. 
The appearance of the latter intersection motivates the following definition. 

\begin{dfn}\label{deftau}
  The {\it Connect Four morphism of functors}, or \emph{C4 morphism}, is the morphism of functors
  \begin{equation*}
    \widetilde{\tau}:h_{\widehat{Y}^\Delta} \to \HHi^{\prec\Delta}_{S/k}
  \end{equation*}
  defined on each subfunctor $h_{\widehat{Y}^I}$ of $h_{\widehat{Y}^\Delta}$ by the following property: 
  For each $k$-algebra $B$ having no nontrivial idempotents, 
  \begin{equation*}
    \widetilde{\tau}(B):h_{\widehat{Y}^I}(B) \to \HHi^{\prec\Delta}_{S/k}(B)
  \end{equation*}
  is the map which sends a tuple $\bigl( \langle J_i \rangle + \langle x_{n}-b_i \rangle \bigr)_{i \in I}$, 
  where $J_i \subset \overline{S} \otimes_k B$ is monic and $b_i-b_{j}\in B^*$ for $i\neq j$, 
  to the ideal $\cap_{i \in I}( \langle J_i \rangle + \langle x_{n}-b_i \rangle )$ in $S \otimes_k B$. 
\end{dfn}

Note that a priori it is neither clear that the map $\widetilde{\tau}$ is a transformation of functors,
nor that it has the correct range, i.e. that the intersection 
$\cap_{i \in I}( \langle J_i \rangle + \langle y_i-b_i \rangle )$ is an element of 
$\HHi^{\prec\Delta}_{S/k}(B)$. 
Instead of proving this directly, we shall define the announced morphism of schemes 
$\widehat{\tau}: \widehat{Y}^\Delta \to \Hi^{\prec\Delta}_{S/k}$
and subsequently prove that the transformation of functors corresponding to 
$\widehat{\tau}$ is the above defined $\widetilde{\tau}$. 
In fact, we shall define $\widehat{\tau}$ by giving its corresponding ring homomorphism
\begin{equation*}
  \widehat{\tau}^*: R^{\prec\Delta} \to R^{\prime\prec\Delta} .
\end{equation*}
Theorem \ref{connectfourthm} below states that $\widehat{\tau}$ and $\widetilde{\tau}$ correspond to each other. 
The crux of the proof of that theorem is to see what $\widehat{\tau}$ does to the universal object over $\widehat{Y}^\Delta$. 

%%%%%%%

\subsection{Interpolation and reduction of Gr\"obner bases}

For defining the ring homomorphism $\widehat{\tau}^*$, we carry over the techniques of \cite{jpaa} to our situation here. 
The difference between the situation of the cited paper and our situation here is that we now no longer work over a field, 
but rather over more complicated rings. 

We start with a fixed C4 decomposition $\{\Delta_i: i\in I\}$ of $\Delta$. 
The reduced Gr\"obner basis of the monic ideal $J^{\prec\Delta_i}$ in $\overline{S} \otimes_k R^{\prec\Delta_i}$
is formed by polynomials of the form 
\begin{equation}\label{fioverline}
  f_{i,\overline{\alpha}}=\overline{x}^{\overline{\alpha}}
  +\sum_{\overline{\beta}\in\Delta_i,\overline{\beta}\prec\overline{\alpha}}
  c_{i,\overline{\alpha},\overline{\beta}}\overline{x}^{\overline{\beta}}
\end{equation}
where $\overline{\alpha}$ runs through $\mathscr{C}(\Delta_i) \subset \mathbb{N}^{n-1}$.
By Lemma 1 of \cite{strata}, we get a unique polynomial 
$f_{i,\overline{\alpha}}$ in $J^{\prec\Delta_i}$ of shape \eqref{fioverline} for all 
$\overline{\alpha}\in\mathbb{N}^{n-1} \setminus \Delta_i$. 

Next, let $\alpha$ be an arbitrary element of $\mathbb{N}^n \setminus \Delta$. 
We define a partition $I=S(\alpha)\coprod T(\alpha)$ of the indexing set $I$ as follows. 
\begin{equation}\label{salpha}
  S(\alpha) := \{i\in I: q^n(\alpha)\in\Delta_i\} ,
\end{equation}
where $q^n: \mathbb{N}^n \to \mathbb{N}^{n-1}$ is the projection we introduced in Section \ref{combinatorics}, and 
\begin{equation*}
  T(\alpha) := I \setminus S(\alpha)
\end{equation*}
is its complement. In what follows, we use the shorthand notation $\overline{\alpha} = q^n(\alpha)$ 
for the projection of $\alpha$. By definition of the above partition of $I$, for all $i\in T(\alpha)$, 
that projection lies in $\mathbb{N}^{n-1} \setminus \Delta_i$. In particular, for all $i\in T(\alpha)$, 
we get a unique polynomial $f_{i,\overline{\alpha}}\in J^{\prec\Delta_i}$ of shape \eqref{fioverline}. 
We can rewrite that polynomial, setting
\begin{equation*}
  \Gamma(\alpha) := \cup_{i \in T(\alpha)}\{\overline{\beta}\in\Delta_i: \overline{\beta}\prec\overline{\alpha}\}
\end{equation*}
and $c_{i,\overline{\alpha},\overline{\beta}} := 0$ if $\overline{\beta}$ lies in the union 
$\Gamma(\alpha)$ and not in the set $\{\overline{\beta}\in\Delta_i: \overline{\beta}\prec\overline{\alpha}\}$. Thus 
\begin{equation}\label{fg}
  f_{i,\overline{\alpha}} = \overline{x}^{\overline{\alpha}}+\sum_{\overline{\beta}\in\Gamma(\alpha)}
  c_{i,\overline{\alpha},\overline{\beta}}\overline{x}^{\overline{\beta}}
  \in J^{\prime\prec\Delta_i} .
\end{equation}

Consider, for each $i\in T(\alpha)$, the {\it characteristic polynomial
of $y_i$ in $\{y_{j}: j\in T(\alpha)\}$}, i.e., the polynomial 
\begin{equation}\label{chi}
  \chi(T(\alpha),i) := \prod_{j\in T(\alpha) \setminus \{i\}}\frac{x_{n}-y_{j}}{y_i-y_{j}} ,
\end{equation}
which lies in the polynomial ring
\begin{equation*}
  k[x_{n},y_i,\prod_{i\neq j\in I}\frac{1}{y_i-y_{j}}: i\in I] \subset R^{\prime\prec I}[x_{n}] .
\end{equation*}
Upon writing $\chi(T(\alpha),i)$ as a polynomial in the variable $x_{n}$ with coefficients in
$k[y_i,\prod_{i\neq j\in I}\frac{1}{y_i-y_{j}}: i\in I]$, we see that its degree in $x_{n}$ is $\#T(\alpha)-1$.
When evaluating $\chi(T(\alpha),i)$ at $x_{n}=y_i$, the result is $1$;
when evaluating that polynomial at any $x_{n}=y_{j}$, for $i\neq j\in T(\alpha)$, the result is $0$. 
This is the motivation for calling $\chi(T(\alpha),i)$ the characteristic polynomial. 
We also see that 
\begin{equation}\label{sum1}
  \sum_{i \in T(\alpha)} \chi(T(\alpha),i) = 
  \begin{cases} 
    1 & \text{if } T(\alpha) \neq \emptyset , \\
    0 & \text{else}.
  \end{cases}
\end{equation}

Now we build a new polynomial $\theta^{I}_{\alpha}\in S \otimes_k R^{\prime\prec I}$, 
based on the formulas \eqref{fg} and \eqref{chi} above, namely, 
\begin{equation}\label{thetai}
  \theta^{I}_{\alpha} := \overline{x}^{\overline{\alpha}} + 
  \sum_{i \in T(\alpha)}\sum_{\overline{\beta}\in\Gamma(\alpha)}
  \chi(T(\alpha),i)c_{i,\overline{\alpha},\overline{\beta}}\overline{x}^{\overline{\beta}} .
\end{equation}
As $\chi(T(\alpha),i)$ is the characteristic polynomial we described above, 
we see that when evaluating $\theta^{I}_{\alpha}$ at $x_{n}=y_i$, for $i\in T(\alpha)$, 
the result is the polynomial $f_{i,\overline{\alpha}}$. 
Therefore
\begin{equation}\label{thetanonempty}
  \theta^{I}_{\alpha} = 
  \begin{cases}
    \sum_{i \in T(\alpha)} \chi(T(\alpha),i) f_{i,\overline{\alpha}} & \text{if } T(\alpha) \neq \emptyset , \\
    1 & \text{else}.
  \end{cases}
\end{equation}
Finally we also bring $S(\alpha)$ into play, defining 
\begin{equation}\label{phii}
  \phi^{I}_{\alpha} := \theta^{I}_{\alpha}\cdot\prod_{i \in S(\alpha)}(x_{n}-y_i) .
\end{equation}
When evaluating $\phi^{I}_{\alpha}$ at $x_{n}=y_i$, for $i\in S(\alpha)$, the result is $0$. 

As we use the lexicographic order on the polynomial ring $S \otimes_k R^{\prime\prec I}$, 
we see that for all $\alpha \in \mathscr{C}(\Delta)$, 
\begin{itemize}
  \item the leading term of $\phi^{I}_{\alpha}$ is $x^\alpha$; and
  \item the non-leading exponents of $\phi^{I}_{\alpha}$ lie in the union of 
  $\{0,\ldots,\#I-1\}\times\Gamma(\alpha)$ and $\{0,\ldots,\#S(\alpha)-1\}\times\{\overline{\alpha}\}$. 
\end{itemize}
For the first bulleted item, we used equation \eqref{sum1}.
The second item is good but not perfect news --- we generally prefer monic polynomials
with leading exponents in $\mathbb{N} \setminus \Delta$ and non-leading exponents in $\Delta$. 
For getting there, we have to modify our $\phi^{I}_{\alpha}$ without changing the ideal they span. 
We therefore consider the ideal
\begin{equation*}
  J^{I} := \langle \phi^{I}_{\alpha}: 
  \alpha \in \mathscr{C}(\Delta) \rangle \subset S \otimes_k R^{\prime\prec I} .
\end{equation*}

\begin{pro}\label{algorithm}
  Let $J^{I} \subset S \otimes_k R^{\prime\prec I}$ be the ideal defined above. 
  Then for all $\alpha\in\mathbb{N}^n \setminus \Delta$, 
  there exists a monic polynomial $\psi^{I}_{\alpha}\in J^{I}$ 
  whose leading exponent is $\alpha$ and whose non-leading exponents lie in $\Delta$. 
\end{pro}

\begin{proof}
  The proof of this proposition is essentially the same as the proofs of Theorem 9 and Corollary 10 of \cite{jpaa}
  (which two proofs are entangled with each other). 
  Nevertheless, we present the full proof here, instead of leaving it to the readers, 
  as the rings we are using here are not quite the same as those we use in the cited paper. 
  
  In fact, we prove the following statement: For all $\lambda \in \mathscr{C}(\Delta)$ 
  and for all $\alpha\in\cup_{\lambda^\prime\preceq\lambda}(\lambda^\prime+\mathbb{N}^n)$, 
  there exists a polynomial $f_{\alpha}\in J^{I}$ such that the leading term of $f_{\alpha}$ is $x^\alpha$
  and the non-leading exponents of $f_{\alpha}$ lie in 
  $\mathbb{N}^n \setminus \cup_{\lambda^\prime\preceq\lambda}(\lambda^\prime+\mathbb{N}^n)$. 
  (Here, and in the rest of the proof, the union over $\lambda^\prime\preceq\lambda$ 
  always means all $\lambda^\prime\in\mathscr{C}(\Delta)$ such that $\lambda^\prime\preceq\lambda$.)
  The proposition follows from the above statement by taking $\lambda$ 
  to be the maximal element of $\mathscr{C}(\Delta)$. 
  
  The proof consists of two inductions --- the outer induction is over $\lambda\in\mathscr{C}(\Delta)$, 
  and the inner induction is over 
  $\alpha\in\cup_{\lambda^\prime\preceq\lambda}(\lambda^\prime+\mathbb{N}^n)$.
  For the basis of the outer induction, we define $\lambda$ to be the minimal element of $\mathscr{C}(\Delta)$.
  For the basis of the inner induction, we define $\alpha := \lambda$. 
  For these data, we set
  \begin{equation*}
    f_{\alpha} := \phi^{I}_{\alpha} = \prod_{i \in I}(x_{n}-y_i) .
  \end{equation*}
  This establishes the basis of the inner induction.
  
  For the inner induction step, we take a non-minimal $\alpha\in\lambda+\mathbb{N}^n$ 
  and assume that we have found the desired polynomial $f_{\alpha^\prime}$
  for all $\alpha^\prime\in\lambda+\mathbb{N}^n$ such that $\alpha^\prime\prec\alpha$. 
  Non-minimality of $\alpha$ implies the existence an $i\in\{1,\ldots,n\}$ 
  such that $\alpha^\prime := \alpha-e_i$ lies in $\lambda+\mathbb{N}^n$ as well. 
  Clearly $\alpha^\prime\prec\alpha$ holds true, therefore $f_{\alpha^\prime}$ exists. 
  We define $\Gamma$ to be the set of all $\gamma\in\lambda+\mathbb{N}^n$ 
  such that $\gamma-e_i$ is the exponent of some non-leading term of $f_{\alpha^\prime}$. 
  In particular, for all $\gamma\in\Gamma$, we have $\gamma-e_i\prec\alpha^\prime=\alpha-e_i$, 
  hence $\gamma\prec\alpha$, hence $f_{\gamma}$ exists. We set
  \begin{equation}\label{corrected}
    f_{\alpha} := x_i f_{\alpha^\prime} - \sum_{\gamma\in\Gamma}c_{\gamma}f_{\gamma} ,
  \end{equation}
  where $c_{\gamma}$ is the coefficient of $x^{\gamma-e_i}$ in $f_{\alpha^\prime}$. 
  This establishes the inner induction step. 
  
  For the outer induction step, we take a non-minimal $\lambda\in\mathscr{C}(\Delta)$
  and denote by $\lambda^{\prime\prime}$ its predecessor in $\mathscr{C}(\Delta)$. 
  We may assume that we have found the desired polynomial $f_{\alpha^\prime}$ for all 
  $\alpha^\prime\in\cup_{\lambda^\prime\preceq\lambda^{\prime\prime}}(\lambda^\prime+\mathbb{N}^n)$. 
  We have to show the existence of $f_\alpha$ for all 
  $\alpha\in\cup_{\lambda^\prime\preceq\lambda}(\lambda^\prime+\mathbb{N}^n)$. 

  First we note that $f_\alpha$ exists for all 
  $\alpha\in\cup_{\lambda^\prime\preceq\lambda}(\lambda^\prime+\mathbb{N}^n)$ such that
  $\alpha\prec\lambda$. Indeed, in this case $\alpha$ even lies in 
  $\cup_{\lambda^\prime\preceq\lambda^{\prime\prime}}(\lambda^\prime+\mathbb{N}^n)$, 
  since otherwise, $\alpha\in\lambda+\mathbb{N}^n$, hence $\alpha\succeq\lambda$. 
  A priori the non-leading exponents $\gamma$ of the attached $f_{\alpha}$ lie in
  $\mathbb{N}^n \setminus \cup_{\lambda^\prime\preceq\lambda^{\prime\prime}}(\lambda^\prime+\mathbb{N}^n)$.
  Yet in fact they even lie in 
  $\mathbb{N}^n \setminus \cup_{\lambda^\prime\preceq\lambda}(\lambda^\prime+\mathbb{N}^n)$,
  since otherwise, $\gamma\in\lambda+\mathbb{N}^n$, 
  thus $\alpha\prec\lambda\preceq\gamma$, a contradiction. 
  
  Therefore, we have to construct $f_{\alpha}$ for all 
  $\alpha\in\cup_{\lambda^\prime\preceq\lambda}(\lambda^\prime+\mathbb{N}^n)$ 
  such that $\alpha\succeq\lambda$. 
  Again, we do this by induction over $\alpha$, the inner induction.
  For the basis of this induction, we have to consider $\alpha := \lambda$. 
  The polynomial $\phi^{I}_{\alpha}\in J^{I}$ has leading exponent $\alpha$, 
  but its non-leading exponents lie in too large a set. For repairing this, 
  we have to get rid of all terms of $\phi^{I}_{\alpha}$ which lie in the product
  $\{0,\ldots,\#I-1\}\times\Gamma(\alpha)$ and not in 
  $\mathbb{N}^n \setminus \cup_{\lambda^\prime\preceq\lambda}(\lambda^\prime+\mathbb{N}^n)$. 
  (Note that $\{0,\ldots,\#S(\alpha)-1\}\times\{\overline{\alpha}\}$ is a subset of $\Delta$.) 
  Consider the set
  \begin{equation*}
    \Gamma := \bigl( \{0,\ldots,\#I-1\} \times \Gamma(\alpha) \bigr) \cap
    \bigl( \cup_{\lambda^\prime\preceq\lambda}(\lambda^\prime+\mathbb{N}^n) \bigr) .
  \end{equation*}
  The existence of $f_{\gamma}$ is shown for all $\gamma\in\Gamma$, 
  since $\overline{\gamma}\prec\overline{\alpha}$ implies $\gamma\prec\alpha$ in the lexicographic order. 
  Therefore, the polynomial 
  \begin{equation*}
    f_{\alpha} := \phi^{I}_{\alpha} - \sum_{\gamma\in\Gamma}c_{\gamma}f_{\gamma} ,
  \end{equation*}
  where $c_{\gamma}$ is the coefficient of $x^{\gamma}$ in $\phi^{I}_{\alpha}$, has the desired properties. 
  This establishes the inner induction basis. 
  
  For the inner induction step, we take an
  $\alpha\in\cup_{\lambda^\prime\preceq\lambda}(\lambda^\prime+\mathbb{N}^n)$
  such that $\alpha\succ\lambda$, 
  assume that the existence of $f_{\alpha^\prime}$ for all $\alpha^\prime\prec\alpha$
  in $\cup_{\lambda^\prime\preceq\lambda}(\lambda^\prime+\mathbb{N}^n)$, 
  and show the existence of $f_{\alpha}$. 
  In this situation, there exists an $i$ such that $\alpha^\prime := \alpha-e_i$ lies in 
  $\cup_{\lambda^\prime\preceq\lambda}(\lambda^\prime+\mathbb{N}^n)$. 
  Now we define $\Gamma$ to be the set of all 
  $\gamma\in\cup_{\lambda^\prime\preceq\lambda}(\lambda^\prime+\mathbb{N}^n)$
  such that $\gamma-e_i$ is a non-leading exponent of $f_{\alpha^\prime}$. 
  As $\gamma-e_i\prec\alpha^\prime=\alpha-e_i$, we also have $\gamma\prec\alpha$, 
  therefore all $f_{\gamma}$ exist. We define $f_{\alpha}$ by equation \eqref{corrected} as above. 
\end{proof}

\begin{cor}\label{corolla}
  For the length of the $R^{\prime\prec I}$-module $S \otimes_k R^{\prime\prec I}/J^{I}$, we have the inequality 
  \begin{equation*}
    {\rm length}\,S \otimes_k R^{\prime\prec I}/J^{I}\leq\#\Delta .
  \end{equation*}
\end{cor}

\begin{proof}
  Each polynomial $\psi^{I}_{\alpha}\in J^{I}$ is monic with leading exponent $\alpha$ 
  and non-leading exponents in $\Delta$.
  Therefore, the canonical homomorphism of $R^{\prime\prec I}$-modules
  \begin{equation*}
    \oplus_{\beta\in\Delta}R^{\prime\prec I}\cdot x^\beta\to S \otimes_k R^{\prime\prec I}/J^{I}
  \end{equation*}
  is surjective. 
\end{proof}

For a fixed $i\in I$, the set of polynomials $f_{i,\overline{\alpha}}$, 
where $\overline{\alpha}$ runs through any subset of $\mathbb{N}^{n-1}$ containing $\mathscr{C}(\Delta_i)$, 
appended by the polynomial $x_{n}-y_i$, is a Gr\"obner basis of the ideal $J^{\prime\prec\Delta_i}$. 
The polynomials $\phi^{I}_{\alpha}$ spanning $J^{I}$ arise from these Gr\"obner bases, 
where $i$ runs through $I$, by an interpolation process. Subsequently, 
the polynomials $\psi^{I}_{\alpha}$ of Proposition \ref{algorithm} arise from the latter by a reduction process. 
This motivates the name of the present subsection. 
In the next subsection, we will see that the polynomials $\psi^{I}_{\alpha}$, 
where $\alpha$ runs through $\mathscr{C}(\Delta)$, are the reduced Gr\"obner basis of $J^{I}$. 

%%%%%%%

\subsection{The Connect Four ring homomorphism}

We write the polynomials of Proposition \ref{algorithm} as
\begin{equation}\label{psii}
  \psi^{I}_{\alpha}=x^\alpha+\sum_{\beta\in\Delta}c^{I}_{\alpha,\beta}x^\beta ,
  \text{ where }c^{I}_{\alpha,\beta}=0\text{ if }\alpha\prec\beta .
\end{equation}
We use these coefficients, for all C4 decompositions of $\Delta$, 
for all $\alpha\in N$ (where $\Delta \subset N \subset \mathbb{N}^n$,
as in Section \ref{groebnerscheme}) 
and for all $\beta\in\Delta$, for defining a ring homomorphism
\begin{equation*}
  \begin{split}
    \sigma: k[T_{\alpha,\beta},\alpha\in N,\beta\in\Delta]&\to\prod_{I\in\mathcal{I}}R^{\prime\prec I}
    =R^{\prime\prec\Delta}\\
    T_{\alpha,\beta}&\mapsto(c^{I}_{\alpha,\beta})_{I\in\mathcal{I}} .
  \end{split}
\end{equation*}
Remember that $R^{\prec\Delta}=k[T_{\alpha,\beta},\alpha\in N,\beta\in\Delta]/I^{\prec\Delta}$ 
arises as the quotient of the domain of $\sigma$ by the ideal $I^{\prec\Delta}$. 

\begin{thm}\label{connectfourthm}
  The ring homomorphism $\sigma$ factors through $R^{\prec\Delta}$, 
  defining a morphism of schemes
  \begin{equation*}
    \widehat{\tau}: \widehat{Y}^\Delta \to \Hi^{\prec\Delta}_{S/k}
  \end{equation*}
  whose corresponding morphism of functors is $\widetilde{\tau}$. 
\end{thm}

\begin{proof}
  We first show that $\sigma$ factors through $R^{\prec\Delta}$. 
  By the universal property of $R^{\prec\Delta}$ from Lemma \ref{universalring}, 
  we have to show that for all $I\in\mathcal{I}$, 
  the ideal $J^{I} \subset S \otimes_k R^{\prime\prec I}$ is monic
  with standard set $\Delta$, having the reduced Gr\"obner basis $\psi^{I}_{\alpha}$,
  where $\alpha$ runs through $\mathscr{C}(\Delta)$. 
  Note that the polynomials $\psi^{I}_{\alpha}$ take the shape \eqref{psii}.
  It suffices to prove that the quotient $S \otimes_k R^{\prime\prec I}/J^{I}$
  is a free $R^{\prime\prec I}$-module of rank $r=\#\Delta$. 
  Indeed, the existence of monic polynomials with leading coefficients in $\mathscr{C}(\Delta)$
  and non-leading coefficients in $\Delta$ tells us that the set of leading exponents of elements of $J^{I}$ is
  $\mathbb{N}^n \setminus \Delta$ or larger. 
  Also, the collection of all $x^\beta$, where $\beta$ runs through $\Delta$, 
  is a system of generators of $S \otimes_k R^{\prime\prec I}/J^I$. 
  If the set of leading exponents of elements of $J^I$ is strictly larger than $\mathbb{N}^n \setminus \Delta$, 
  the above $x^\beta$ cannot not be a basis of $S \otimes_k R^{\prime\prec I}/J^{I}$. 
  If therefore, these generators do conversely form a basis, 
  the set of leading exponents of elements of $J^{I}$ equals $\mathbb{N}^n \setminus \Delta$. 
  
  So we have to show freeness of rank $r$ of $S \otimes_k R^{\prime\prec I}/J^{I}$. 
  Remember from Section \ref{lowerdim} that for all $i$, 
  the ideal $J^{\prec \Delta_i}$ in $\overline{S} \otimes_k R^{\prec \Delta_i}$ is monic with standard set $\Delta_i$. 
  In particular, we get an isomorphism of $R^{\prec \Delta_i}$-modules
  \begin{equation*}
    \overline{S} \otimes_k R^{\prec \Delta_i}/J^{\prec \Delta_i}\cong
    \oplus_{\overline{\beta}\in\Delta_i}R^{\prec \Delta_i}\cdot\overline{x}^{\overline{\beta}} .
  \end{equation*}
  Also remember that $R^{\prime\prec \Delta_i}=R^{\prec \Delta_i}[y_i]$. 
  Therefore, upon writing $R^{\prime\prec \Delta_i}\cdot J^{\prec \Delta_i}$ for the ideal in
  $\overline{S} \otimes_k R^{\prime\prec \Delta_i}$ spanned by $J^{\prec \Delta_i}$, we get isomorphisms
  \begin{equation*}
    \begin{split}
      & \overline{S} \otimes_k R^{\prime\prec \Delta_i}/R^{\prime\prec \Delta_i}\cdot J^{\prec \Delta_i}
      \cong
      R^{\prec \Delta_i}[\overline{x},y_i]/R^{\prec \Delta_i}[y_i]\cdot J^{\prec \Delta_i}\\
      \cong
      &\oplus_{\overline{\beta}\in\Delta_i}R^{\prec \Delta_i}[y_i] \cdot\overline{x}^{\overline{\beta}}
      \cong
      \oplus_{\overline{\beta}\in\Delta_i}R^{\prime\prec \Delta_i}\cdot\overline{x}^{\overline{\beta}} .
    \end{split}
  \end{equation*}
  Remember the definition of the ideal $J^{\prime\prec \Delta_i} \subset S \otimes_k R^{\prime\prec \Delta_i}$
  from Section \ref{hilbtimesa}, from which we get isomorphisms
  \begin{equation*}
    \begin{split}
      &S \otimes_k R^{\prime\prec \Delta_i}/J^{\prime\prec \Delta_i}
      \cong
      R^{\prime\prec \Delta_i}[\overline{x},x_{n}]/(J^{\prec \Delta_i}+(x_{n}-y_i))\\
      \cong
      &\overline{S} \otimes_k R^{\prime\prec \Delta_i}/J^{\prec \Delta_i}
      \cong
      \oplus_{\overline{\beta}\in\Delta_i}R^{\prime\prec \Delta_i}\cdot\overline{x}^{\overline{\beta}} .
    \end{split}
  \end{equation*}
  Thus the quotient $S \otimes_k R^{\prime\prec \Delta_i}/J^{\prime\prec \Delta_i}$
  is a free module over $R^{\prime\prec \Delta_i}$ of rank $\#\Delta_i$. 
  The same line of arguments shows that 
  $S \otimes_k R^{\prime\prec I}/R^{\prime\prec I}\cdot J^{\prime\prec \Delta_i}$ 
  is a free $R^{\prime\prec I}$-module with basis 
  $\{\overline{x}^{\overline{\beta}}: \overline{\beta}\in\Delta_i\}$.
  Its rank, therefore, is also $\#\Delta_i$. 
  
  We now consider the $R^{\prime\prec I}$-module homomorphism 
  \begin{equation*}
    \begin{split}
      \epsilon:S \otimes_k R^{\prime\prec I} & \to
      \oplus_{i \in I}S \otimes_k R^{\prime\prec I}/R^{\prime\prec I}\cdot J^{\prime\prec \Delta_i}\\
      f(x) & \mapsto \bigl( f(x)+R^{\prime\prec I}\cdot J^{\prime\prec \Delta_i} \bigr)_{i \in I} .
    \end{split}
  \end{equation*}
  
  Our first claim concerning $\epsilon$ is that $J^{I} \subset \ker\epsilon$. 
  For proving this we have to show that all generators $\phi^{I}_{\alpha}$ of $J^{I}$ lie in $\ker\epsilon$. 
  This means that for all $i\in I$, 
  the polynomial $\phi^{I}_{\alpha}$ lies in $R^{\prime\prec I}\cdot J^{\prime\prec \Delta_i}$. 
  If $i\in S(\alpha)$, this is trivial, as $\phi^{I}_{\alpha}$ contains the factor $(x_{n}-y_i)$. 
  If $i\in T(\alpha)$, then in particular $T(\alpha) \neq \emptyset$, 
  and we may use the upper line of \eqref{thetanonempty}
  for studying the factor $\theta^{I}_{\alpha}$ of $\phi^{I}_{\alpha}$,
  \begin{equation*}
    \begin{split}
      \theta^{I}_{\alpha} &= 
      \sum_{i^\prime \in T(\alpha)} \chi(T(\alpha),i^\prime) f_{i^\prime,\overline{\alpha}} \\
      &= \chi(T(\alpha),i) f_{i,\overline{\alpha}} 
      + \sum_{i^\prime \in T(\alpha) \setminus \{i\}} \chi(T(\alpha),i^\prime) f_{i^\prime,\overline{\alpha}}
    \end{split}
  \end{equation*}
  The first summand is contains the factor $f_{i,\overline{\alpha}}$ 
  and therefore lies in $J^{\prime\prec \Delta_i}$. 
  Each of the other summands contains the factor $(x_{n}-y_i)$ 
  and therefore also lies in $J^{\prime\prec \Delta_i}$. 
  The first claim is proved. 
  
  Our second claim is that $\epsilon$ is surjective. 
  For this it suffices to take an arbitrary $i\in I$ and to find an element $a \in S \otimes_k R^{\prime\prec I}$ 
  such that $\epsilon(a)=(0,\ldots,0,b,0,\ldots,0)$, where $b$ is an invertible element of 
  $S \otimes_k R^{\prime\prec I}/R^{\prime\prec I}\cdot J^{\prime\prec \Delta_i}$. 
  We take $a := \chi(T(\alpha),i)$. Then if $i^\prime\neq i$, 
  the factor $(x_{n}-y_i)$ appears in $a$, hence $a\in J^{\prime\prec \Delta_i}$, 
  thus the $i^\prime$-th component of $\epsilon(a)$ vanishes. 
  Upon considering $i$, we see that for all $j\in I \setminus \{i\}$, 
  the factor $(x_{n}-y_{j})=(x_{n}-y_i)+(y_i-y_{j})$ appears in $a$. 
  The first summand, $x_{n}-y_i$, 
  vanishes in $S \otimes_k R^{\prime\prec I}/R^{\prime\prec I}\cdot J^{\prime\prec \Delta_i}$, 
  and the second summand, $y_i-y_{j}$, is invertible in that ring. The claim is proved. 
  
  From the second claim we get an isomorphism of $R^{\prime\prec I}$-modules
  \begin{equation*}
    S \otimes_k R^{\prime\prec I}/\ker\epsilon\cong
    \oplus_{i \in I}S \otimes_k R^{\prime\prec I}/R^{\prime\prec I}\cdot J^{\prime\prec \Delta_i} .
  \end{equation*}
  Each direct summand on the right hand side is a free $R^{\prime\prec I}$-module of rank $\#\Delta_i$. 
  Therefore, $S \otimes_k R^{\prime\prec I}/\ker\epsilon$ is a free $R^{\prime\prec I}$-module 
  of rank $\sum_{i \in I}\#\Delta_i=\#\Delta$. 
  From the first claim, together with Corollary \ref{corolla}, 
  we get an inequality of lengths of $R^{\prime\prec I}$-modules,
  \begin{equation}\label{geq}
    \#\Delta={\rm length}\,S \otimes_k R^{\prime\prec I}/\ker\epsilon\leq
    {\rm length}\,S \otimes_k R^{\prime\prec I}/J^{I}\leq\#\Delta .
  \end{equation}
  This is in fact an equality, and by what we said in the first paragraph of the present proof, 
  we have just proved that $\sigma$ factors through $R^{\prec \Delta}$. 
  We denote the induced homomorphism by
  \begin{equation*}
    \widehat{\tau}^*: R^{\prec \Delta} \to R^{\prime\prec \Delta} .
  \end{equation*}
  This defines the claimed morphism of schemes $\widehat{\tau}: \widehat{Y}^\Delta \to \Hi^{\prec \Delta}_{S/k}$. 
  
  It remains to show that the morphism of functors corresponding to the ring homomorphism $\widehat{\tau}^*$ 
  is the same as the morphism $\widetilde{\tau}$ of Definition \ref{deftau}.
  So let $B$ be a $k$-algebra with no nontrivial idempotents. 
  We consider a fixed element $g$ of $h_{\widehat{Y}^\Delta}(B)$. 
  This is just a homomorphism $g:R^{\prime\prec \Delta}\to B$, 
  and by the arguments of Section \ref{alldecompositions}, 
  in fact a homomorphism $g:R^{\prime\prec I}\to B$ for one fixed $I$. 
  (By abuse of language we denote both homomorphisms by the same letter, $g$.)
  We identify $g$ with the tuple of ideals
  \begin{equation*}
    \bigl( \langle ({\rm id} \otimes g^{\prime\prime})(J^{\prime\prec \Delta_i}) \rangle \bigr)_{i \in I}
    =\bigl( \langle J_i \rangle + \langle x_{n}-b_i \rangle \bigr)_{i \in I} ,
  \end{equation*}
  as we did in Section \ref{alldecompositions}. Denote, 
  for the time being, the morphism of functors corresponding to $\widehat{\tau}$ by 
  $h_{\widehat{\tau}}: h_{\widehat{Y}^\Delta} \to \HHi^{\prec \Delta}_{S/k}$. 
  For finishing the proof of the theorem, 
  we have to show that the image of the ideal $J^{\prec \Delta}$ under the composition
  \begin{equation*}
    \xymatrix{
      S \otimes_k R^{\prec \Delta} \ar[r]^{{\rm id} \otimes \widehat{\tau}^*} & 
      S \otimes_k R^{\prime\prec \Delta} \ar[r]^{{\rm id} \otimes g} & S \otimes_k B ,
    }
  \end{equation*}
  satisfies the identity 
  \begin{equation}\label{intersectionidentity}
    \langle (({\rm id} \otimes g)\circ({\rm id} \otimes \widehat{\tau}^*))(J^{\prec \Delta}) \rangle 
    = \cap_{i \in I}( \langle J_i \rangle + \langle x_{n}-b_i \rangle )
  \end{equation}
  since the ideal on the left hand side is $h_{\widehat{\tau}}(B)$ applied to $g$, 
  and the right hand side is $\widetilde{\tau}(B)$ applied to $g$. 
  
  For showing \eqref{intersectionidentity}, we first take another look at \eqref{geq}.
  Equality holds true in that formula, hence $\ker\epsilon=J^{I}$. 
  Therefore, the definition of $\epsilon$ shows that
  $J^{I}=\cap_{i \in I}R^{\prime\prec I}\cdot J^{\prime\prec \Delta_i}$.
  Also, since $\widehat{\tau}^*$ sends each generator $T_{\alpha,\beta}$ of $R^{\prec \Delta}$ 
  to the tuple $(c^{I}_{\alpha,\beta})_{i \in\mathcal{I}}$ in $R^{\prime\prec \Delta}$, 
  the identity of ideals
  $\langle ({\rm id} \otimes \widehat{\tau}^*)(J^{\prec \Delta}) \rangle = (J^{I})_{I\in\mathcal{I}}$
  in $S \otimes_k R^{\prime\prec \Delta}$ follows. 
  Putting things together, we get
  \begin{equation}\label{intersectionuniversal}
    \langle ({\rm id} \otimes \widehat{\tau}^*)(J^{\prec \Delta}) \rangle
    =(\cap_{i \in I}R^{\prime\prec I}\cdot J^{\prime\prec \Delta_i})_{I\in\mathcal{I}} .
  \end{equation}
  This identity is a ``universal form'' of the identity \eqref{intersectionidentity} we wish to prove. 
  
  Indeed, we can rewrite the ideal on the left hand side of \eqref{intersectionidentity} as 
  \begin{equation*}
    \langle ({\rm id} \otimes g)\circ({\rm id} \otimes \widehat{\tau}^*)(J^{\prec \Delta}) \rangle
    = \langle ({\rm id} \otimes g)((J^{I})_{I\in\mathcal{I}}) \rangle .
  \end{equation*}
  Since ${\rm id} \otimes g$ is a nonzero morphism only on the $I$-th factor of $R^{\prime\prec \Delta}$, 
  that ideal equals
  \begin{equation*}
    \langle ({\rm id} \otimes g)(J^{I}) \rangle
    = \langle ({\rm id} \otimes g)(\cap_{i \in I}R^{\prime\prec I}\cdot J^{\prime\prec \Delta_i}) \rangle .
  \end{equation*}
  The latter ideal is clearly contained in the intersection
  \begin{equation*}
    \cap_{i \in I} \langle ({\rm id} \otimes g)(R^{\prime\prec I}\cdot J^{\prime\prec \Delta_i}) \rangle
    =\cap_{i \in I}( \langle J_i \rangle + \langle x_{n}-b_i \rangle ) .
  \end{equation*}
  Thus we have shown that the left hand side of \eqref{intersectionidentity} is contained in its right hand side.  
  For showing also the other inclusion, we look at the ideal $\langle ({\rm id} \otimes g)(J^{I}) \rangle$. 
  By functoriality of reduced Gr\"obner bases with standard set $\Delta$, 
  that ideal is monic with standard set $\Delta$.
  (A reference for functoriality is \cite{strata}, Section 6, Lemma 2; for seeing this directly in our context here,
  one argues as follows: $J^{I}$ is monic with reduced Gr\"obner basis $\psi^{I}_{\alpha}$ as in \eqref{psii}, 
  where $\alpha$ runs through $\mathscr{C}(\Delta)$. 
  This means that the coefficients $c^{I}_{\alpha,\beta}$ of the various $\psi^{I}_{\alpha}$
  satisfy the quadratic equations in the variables $T_{\alpha,\beta}$ 
  which span the ideal $I^{\prec \Delta}$, see \eqref{generators}.
  As $g$ is a ring homomorphism, also the coefficients of the polynomials 
  $\xi^{I}_{\alpha} := ({\rm id} \otimes g)(\psi^{I}_{\alpha})$ satisfy the same equations. 
  These polynomials span the ideal $\langle ({\rm id} \otimes g)(J^{I}) \rangle \subset S \otimes_k B$, 
  which is therefore monic with standard set $\Delta$.)
  Remember that the first claim about $\epsilon$ from earlier in the present proof
  says that for all generators $\phi^{I}_{\alpha}$ of $J^{I}$ and for all $i\in I$, 
  we have $\phi^{I}_{\alpha}\in R^{\prime\prec I}\cdot J^{\prime\prec \Delta_i}$. 
  Therefore, also all $\psi^{I}_{\alpha}$ lie in $R^{\prime\prec I}\cdot J^{\prime\prec \Delta_i}$. 
  Upon applying ${\rm id} \otimes g$, we see that all $\xi^{I}_{\alpha}$ 
  lie in $\langle ({\rm id} \otimes g)(R^{\prime\prec I}\cdot J^{\prime\prec \Delta_i}) \rangle 
  = \langle J_i \rangle + \langle x_{n}-b_i \rangle$, 
  thus $\xi^{I}_{\alpha} \in \cap_{i \in I}( \langle J_i \rangle + \langle x_{n}-b_i \rangle )$. We get a surjection
  \begin{equation*}
    S \otimes_k B/\langle ({\rm id} \otimes g)(J^{I}) \rangle 
    \to S \otimes_k B/\cap_{i \in I}( \langle J_i \rangle + \langle x_{n} - b_i \rangle )
  \end{equation*}
  of $B$-modules. The module on the left hand side has rank $\#\Delta$. 
  By the Chinese Remainder Theorem, the module on the right hand side is isomorphic to
  $\oplus_{i \in I} S \otimes_k B/( \langle J_i \rangle + \langle x_{n}-b_i \rangle )$
  (for this isomorphism we use that $b_i-b_{j}\in B^*$ for $i\neq j$), 
  hence also has rank $\sum_{i \in I}\#\Delta_i=\#\Delta$. 
  Therefore, the above surjection is an isomorphism. This also shows the missing inclusion
  in \eqref{intersectionidentity}, and we are done.
\end{proof}

The following corollary has been proved along the lines of proving Theorem \ref{connectfourthm}. 
As it is interesting in its own right and was announced at the end of the last subsection, we state it separately here. 
Note that for its proof, we used the Chinese Remainder Theorem, 
as we did for proving the uniqueness assertion, i.e., the reduced Gr\"obner basis property, in Corollary 10 of \cite{jpaa}. 

\begin{cor}
  For all $I\in\mathcal{I}$, the ideal $J^{I} \subset S \otimes_k R^{\prime\prec I}$ is monic
  with reduced Gr\"obner basis $\{\psi^{I}_{\alpha}: \alpha\in\mathscr{C}(\Delta)\}$. 
  In particular, the polynomials of Proposition \ref{algorithm} are unique. 
\end{cor}

Remember that the action of the group $G$ on $\widehat{Y}^I$ 
is given by the action of $G$ on the coordinate ring $R^{\prime\prec\Delta}$, as was defined in \eqref{action}. 

\begin{cor}\label{connectfourmultiset}
  The ring homomorphism $\sigma$ factors through $(R^{\prec\Delta})^G$, 
  defining a morphism of schemes
  \begin{equation*}
    \tau: Y^\Delta \to \Hi^{\prec\Delta}_{S/k} .
  \end{equation*}
  The effect of the corresponding morphism of functors $h_{\tau}: h_{Y^\Delta} \to \Hi^{\prec\Delta}_{S/k}$, 
  evaluated at a $k$-algebra $B$ having no nontrivial idempotents, 
  is the same as the effect of $\widetilde{\tau}: h_{\widehat{Y}^\Delta} \to \Hi^{\prec\Delta}_{S/k}$, 
  except that $h_{\tau}$ is defined on sets 
  $\left\lbrace \langle J_i \rangle + \langle x_{n}-b_i \rangle: i \in I \right\rbrace$ 
  rather than on such tuples indexed by $I$. 
\end{cor}

\begin{proof}
  $\widehat{\tau}$ is defined via the ring homomorphism $\widehat{\tau}^*$, 
  which is defined via the coefficients $c^I_{\alpha,\beta}$ of the polynomial $\psi^I_\alpha$ appearing in \eqref{psii}. 
  Therefore it suffices to show that the polynomials $\psi^I_\alpha$ are $G$-invariant. 
  Those polynomials arise from the polynomials $\phi^I_\alpha$, defined in \eqref{phii}, 
  by the process presented in the proof of Proposition \ref{algorithm}. 
  That proof is essentially polynomial reduction by all $\phi^I_\alpha$, for $\alpha \in \mathscr{C}(\Delta)$. 
  Therefore it suffices to show that the polynomials $\phi^I_\alpha$ are $G$-invariant. 
  
  The definition of $\phi^I_\alpha$ uses, in particular, the partition $I = S(\alpha) \coprod T(\alpha)$, 
  where $S(\alpha)$ is defined in \eqref{salpha}. 
  From that definition, we immediately see that any given $i \in I$ lies in $S(\alpha)$ if, and only if, 
  $\sigma(i)$ lies in $S(\alpha)$. 
  Next, consider the polynomial $\theta^I_\alpha$, defined in \eqref{thetai}. 
  We think of that polynomial as being a linear combination of the terms 
  $c_{i, \overline{\alpha}, \overline{\beta}} \overline{x}^{\overline{\beta}}$, with coefficients $\chi(T(\alpha), i)$. 
  From the definition of the characteristic polynomials $\chi(T(\alpha), i)$, together with the fact that 
  any given $i \in I$ lies in $T(\alpha)$ if, and only if, 
  $\sigma(i)$ lies in $T(\alpha)$, we see that 
  \begin{equation*}
    \sigma(\chi(T(\alpha), i)) = \chi(T(\alpha), \sigma(i)) . 
  \end{equation*}
  Moreover, from \eqref{action} we see that each factor $S_{h_j}$ of $G$ 
  acts on the coefficients of the polynomial in question via
  \begin{equation*}
    \sigma(c_{i, \overline{\alpha}, \overline{\beta}}) = c_{\sigma(i), \overline{\alpha}, \overline{\beta}} .
  \end{equation*}
  These two identities, together with \eqref{thetai}, imply that $\theta^I_\alpha$ is $G$-invariant. 
  Finally, \eqref{phii} shows that $\phi^I_\alpha$ is $G$-invariant. 
\end{proof}

A closer examination of $\tau$ immediately reveals that this morphism is injective as a map of topological spaces. 
We shall now prove Theorem \ref{immersion}, whose statement is stronger than injectivity of $\tau$. 

%%%%%%%

\subsection{Connect Four is an immersion}\label{proofimmersion}

We shall now prove Theorem \ref{immersion}, i.e., show that $\tau$ is an immersion. 
We start by recalling the definition of the {\it Hilbert scheme of $r$ points}, 
more precisely, its corresponding functor, the {\it Hilbert functor of $r$ points}, 
\begin{equation*}
  \begin{split}
    \HHi^{r}_{S/k}:(k{\rm-Alg})&\to({\rm Sets})\\
    B & \mapsto \left\lbrace 
    \begin{array}{c}
      \text{surjective }B\text{-algebra homomorphisms }\\
      \phi: S \otimes_k B \to Q \text{ s.t. }\\
      Q\text{ is a locally free }B\text{-module of rank } r
    \end{array}
    \right\rbrace/\sim .
  \end{split}
\end{equation*}
Here the relation $\sim$ indicates that two surjective $B$-algebra homomorphisms $\phi_{1}: S \otimes_k B \to Q_{1}$ 
and $\phi_{2}: S \otimes_k B \to Q_{2}$ are equivalent if, and only if, 
there exists an isomorphism $\rho:Q_{1}\to Q_{2}$ such that the diagram
\begin{equation*}
  \xymatrix{ 
    S \otimes_k B \ar[rd] \ar[r] & Q_1 \ar[d]^\rho \\ 
    & Q_2 .	
  }
\end{equation*}
commutes. 
It is far from trivial to show that this functor is represented by a scheme 
(see \cite{norge} for a proof). In fact, $\HHi^{r}_{S/k}$ is covered by open subfunctors 
\begin{equation*}
  \begin{split}
    \HHi^{\Delta}_{S/k}:(k{\rm-Alg})&\to({\rm Sets})\\
    B&\mapsto
    \left\lbrace
    \begin{array}{c}
      B\text{-algebra homomorphisms }\phi:S \otimes_k B \to Bx^\Delta\\
      \hspace{.7cm}\text{ s.t. }\phi\circ({\rm id} \otimes \iota)={\rm id}
    \end{array}
    \right\rbrace ,
  \end{split}
\end{equation*}
where we write $Bx^\Delta=\oplus_{\beta\in\Delta}B\cdot x^\beta$, 
and $\iota:Bx^\Delta \to S \otimes_k B$ for the canonical inclusion (see \cite{huibregtse}, \cite{norge} and \cite{strata}). 
Note that $Bx^\Delta$ does not carry a natural structure of a $B$-algebra; 
the functor $\HHi^{\Delta}_{S/k}$ detects all $B$-algebra structures which can be imposed on $Bx^\Delta$
such that they are compatible with the natural $B$-algebra structure on $S \otimes_k B$. 
It turns out that the functor $\HHi^{\Delta}_{S/k}$ is representable by an affine scheme $\Hi^{\Delta}_{S/k}$, 
and that $\HHi^{\prec \Delta}_{S/k}$ is represented by a closed subscheme $\Hi^{\prec \Delta}_{S/k}$ of that scheme. 
We thus obtain an immersion
\begin{equation*}
  \Hi^{\prec \Delta}_{S/k} \hookrightarrow \Hi^r_{S/k}
\end{equation*}
of a locally closed subscheme of $\Hi^r_{S/k}$. 

\begin{proof}[Proof of Theorem \ref{immersion}]
  As the images of the restrictions of $\tau$ to different $Y^{I}$ do not meet each other, 
  it suffices to show that each restriction $\tau\mid_{Y^{I}}:Y^{I}\to\Hi^{\prec \Delta}_{S/k}$ is an immersion. 
  For all $i$, let $r_i := \#\Delta_i$. 
  First we show that for all $i$, the morphism 
  \begin{equation*}
    \tau_i: \Hi^{\prec \Delta_i}_{\overline{S}/k} \times \mathbb{A}^1_k \to \Hi^{r_i}_{S/k} ,
  \end{equation*}
  whose attached transformation of functors $h_{\tau_i}: (k{\rm-Alg})\to({\rm Sets})$ is given, 
  for each $k$-algebra $B$, by the map of sets 
  \begin{equation*}
    \begin{split}
      h_{\tau_i}(B):
      \left\lbrace
      \begin{array}{c}
        \text{pairs of ideals } (J_i, \langle x_{n}-b_i \rangle): \\
        J_i \subset \overline{S} \otimes_k B \text{ is a monic ideal} \\
        \text{ with standard set }\Delta_i, \\
        \text{ and } \langle x_{n}-b_i \rangle \subset S \otimes_k B
      \end{array}
      \right\rbrace
      & \to
      \left\lbrace 
      \begin{array}{c}
      \text{ideals } J \subset S \otimes_k B: \\
      S \otimes_k B/J\text{Ê is locally free} \\
      \text{ of degree }r_i
      \end{array}
      \right\rbrace
      \\
      (J_i, \langle x_{n}-b_i \rangle) & \mapsto \langle J_i \rangle + \langle x_{n}-b_i \rangle .
    \end{split}
  \end{equation*}
  is an immersion. 
  This is nothing but a reprise of rewriting the functor 
  \begin{equation*}
    h_{\Hi^{\prec\Delta_i}_{\overline{S}/k} \times \mathbb{A}^1_k}: (k{\rm-Alg}) \to ({\rm Sets})
  \end{equation*}
  as we did in Section \eqref{hilbtimesa}. Indeed, each pair $(J_i, \langle x_{n}-b_i \rangle)$ 
  appearing in the domain of $h_{\tau_i}(B)$ takes the shape of the pairs in \eqref{1st}, 
  and each ideal $\langle J_i \rangle + \langle x_{n}-b_i \rangle$ appearing in the range of $h_{\tau_i}(B)$ 
  takes the shape of the pairs in \eqref{2nd}. 
  Upon identifying the standard sets $\Delta_i \subset \mathbb{N}^{n-1}$
  and $\Delta_i\times\{0\} \subset \mathbb{N}^n$, 
  we canonically understand the ideal $\langle J_i \rangle + \langle x_{n}-b_i \rangle \subset S \otimes_k B$ 
  to be monic with standard set $\Delta_i\times\{0\}$. 
  Hence a canonical identification 
  \begin{equation*}
    \Hi^{\prec \Delta_i}_{\overline{S}/k}\times\mathbb{A}^1_k = \Hi^{\prec \Delta_i\times\{0\}}_{S/k} . 
  \end{equation*}
  In this sense each ideal $\langle J_i \rangle + \langle x_{n}-b_i \rangle$ appearing in the range of $h_{\tau_i}(B)$
  is just an element of $\Hi^{\prec (\Delta_i \times \{0\})}_{S/k}(B)$. 
  Therefore, $\tau_i$ is just the locally closed immersion 
  \begin{equation*}
    \Hi^{\prec (\Delta_i \times \{0\})}_{S/k} \hookrightarrow \Hi^{r_i}_{S/k}
  \end{equation*}
  we discussed above. 

  Now Lemma \ref{productimmersion} below shows that the morphism $\rho: Y^I \to \Hi^r_{S/k}$, 
  defined by its corresponding transformation of functors, 
  \begin{equation*}
    h_\rho: \left\lbrace (J_i, \langle x_{n}-b_i \rangle): i \in I \right\rbrace 
    \mapsto \cap_{i \in I}( \langle J_i \rangle + \langle x_{n}-b_i \rangle ) ,
  \end{equation*}
  is an immersion. 
  In fact, we apply the cited lemma several times, in the following way: 
  First we decompose the indexing set into $I = I_1 \coprod \ldots \coprod I_m$ as in \eqref{ij}. 
  Remember that for a fixed $I_j$, all elements of the multiset $\{\Delta_i: i \in I_j\}$ agree. 
  We apply Lemma \ref{productimmersion} (b) in the case where 
  \begin{equation*}
    \begin{split}
      Y_i & := \Hi^{\prec (\Delta_i \times \{0\})}_{S/k} , \\
       c_i & := \tau_i, \text{ for } i \in I_j , \text{ and} \\
       U & := \prod_{i \in I_j} \Hi^{\prec (\Delta_i \times \{0\})}_{S/k} \setminus \Lambda_j , \text{ where }\\
      \Lambda_j & := \cup_{i \neq a \in I_j} \mathbb{V}(y_i - y_a) . 
    \end{split}
  \end{equation*}
  The general assumption of the lemma, namely, that none of the schemes $c_i(y_i) \subset \mathbb{A}^n$ meet each other, 
  is satisfied since the diagonal $\Lambda_j$ has been removed from the product. 
  We do the same thing for all parts $I_j$ of $I$, thus obtaining, for all $j \in \{1, \ldots, m\}$, an immersion 
  \begin{equation*}
    \zeta_j: \bigl( (\prod_{i \in I_j} \Hi^{\prec (\Delta_i \times \{0\})}_{S/k}) \setminus \Lambda_j \bigr) / S_{h_j} 
    \hookrightarrow \Hi^{R_j}_{S/k} ,
  \end{equation*}
  where $R_j := \sum_{i \in I_j} \#I_j$. 
  Then we apply Lemma \ref{productimmersion} (a) in the case where 
  \begin{equation*}
    \begin{split}
      Y_j & := \prod_{i \in I_j} \Hi^{\prec (\Delta_i \times \{0\})}_{S/k} \setminus \Lambda_j , \\
       c_j & := \zeta_j, \text{ for } j \in \{1, \ldots, m\}, \text{ and} \\
       U & := \biggl( \prod_{j \in \{1, \ldots, m\}} 
       \Bigl( \bigl( (\prod_{i \in I_j} \Hi^{\prec (\Delta_i \times \{0\})}_{S/k}) \setminus \Lambda_j \bigr) / S_{h_j} \Bigr) \biggr)
       \setminus \widetilde{\Lambda}, \text{ where }\\
      \widetilde{\Lambda} & := \text{image of } \Lambda \text{ under the canonical map}
    \end{split}
  \end{equation*}
  The general assumption of the lemma, namely, that none of the schemes $c_j(y_j) \subset \mathbb{A}^n$ meet each other, 
  is satisfied since all $\Lambda_j$ and $\widetilde{\Lambda}$ have been removed. 
  The particular assumption of part (a) of the lemma is satisfied as 
  $\Delta_i \neq \Delta_a$ for $i \in I_j$, $a \in I_b$, $j \neq b$. 
  Note that the above-defined scheme $U$ is just $Y^I$, 
  so we obtain the desired immersion $\rho: Y^I \hookrightarrow \Hi^r_{S/k}$. 

  Finally, Theorem \ref{connectfourthm} says that 
  the morphism $\rho$ factors through $\Hi^{\prec \Delta}_{S/k} \hookrightarrow \Hi^{r}_{S/k}$, 
  and that the induced morphism $Y^{I}\to\Hi^{\prec \Delta}_{S/k}$ is just $\tau$. 
\end{proof}

\begin{lmm}\label{productimmersion}
  Let $c_i: Y_i \hookrightarrow \Hi^{r_i}_{S/k}$, for $i = 1, \ldots, m$, 
  and $c: U \hookrightarrow \prod_{i = 1}^m Y_i$ be immersions of schemes such that for all $k$-algebras $A$ and  
  for all $A$-valued points $(y_1, \ldots, y_m)$ of $U$, where the $y_i$ are $A$-valued points of $Y_i$, 
  none of the schemes $c_i(y_i) \subset \mathbb{A}^n_A$ meet each other. 
  Let $r := \sum_{i = 1}^m r_i$. 
  Then there exists a morphism $\iota: U \to \Hi^r_{S/k}$ 
  whose attached transformation of functors $(k{\rm-Alg})\to({\rm Sets})$ is given, 
  for each $k$-algebra $A$, by the map of sets 
  \begin{equation*}
    \begin{split}
      \iota(A):U(A) & \to \left\lbrace w \subset \mathbb{A}^n_A \text{ flat of degree } r \text{ over }{\rm Spec}\,A \right\rbrace \\
      (y_1, \ldots, y_m) &\mapsto c_1(y_1) \cup \ldots \cup c_m(y_m) .
    \end{split}
  \end{equation*}  
  Moreover, $\iota$ is, or induces, resp., an immersion under the following hypotheses: 
  \begin{enumerate}
    \item[(a)] Assume that for all $k$-algebras $A$ and for all $A$-valued points 
    $y = (y_1, \ldots, y_m)$ and $z = (z_1, \ldots, z_m)$ of $U$, 
    where the $y_i$ and $z_i$ are $A$-valued points of $Y_i$, we have $c_i(y_i) \neq c_j(z_j)$ unless $y = z$ and $i = j$. 
    Then $\iota$ is an immersion. 
    \item[(b)] Assume that all $Y_i$, and all $c_i$, for $i = 1, \ldots, m$, are the same, 
    and that $U$ is invariant under the symmetric group $S_m$, 
    acting on the full product $\prod_{i = 1}^m Y_i$ by permuting the factors. 
    Then $\iota$ induces a morphism $\overline{\iota}: U/S_m \to \Hi^r_{S/k}$, which is also an immersion. 
  \end{enumerate}
\end{lmm}

\begin{proof}
  Each map of sets $\iota(A)$ is well-defined, 
  as the hypothesis guarantees that $c_1(y_1) \cup \ldots \cup c_m(y_m) \subset \mathbb{A}^n_A$ 
  is of degree $r$ over ${\rm Spec}\,A$. 
  The collection of maps $i(A)$ is functorial when $A$ runs through all $k$-algebras $A$, 
  as the hypothesis guarantees that $c_1(y_1) \cup \ldots \cup c_m(y_m) \subset \mathbb{A}^n_A$ 
  is the coproduct of $c_1(y_1), \ldots, c_m(y_m)$. 
  Therefore $\iota$ is a well-defined morphism of schemes. 
  
  For showing that $\iota$ and $\overline{\iota}$, resp., are immersions, we give ourselves a cartesian square
  \begin{equation}\label{cartesiana}
    \xymatrix{
      X \ar[d]^h \ar[r]^-j & {\rm Spec}\,B \ar[d]^g \\
      U \ar[r]^-\iota & \Hi^r_{S/k} .
    }
  \end{equation}
  Under hypothesis (a), the problem is to show that the morphism $j$ appearing in \eqref{cartesiana} is an immersion. 
  Under hypothesis (b), we will consider an analogous cartesian diagram with $\iota$ replaced by $\overline{\iota}$. 
  The problem will also be to show that the upper horizontal morphism, 
  which we will denote by $\overline{j}: X/S_m \to {\rm Spec}\,B$, is an immersion. 
  In what follows, we shall construct the sought-for subschemes $X$ and $X/S_m$ of ${\rm Spec}\,B$. 
  
  The first part of our proof is independent of hypotheses (a) and (b). 
  We evaluate diagram \eqref{cartesiana} in an affine test scheme ${\rm Spec}\,A$, obtaining
  \begin{equation}\label{bigdiagram}
    \xymatrix{ 
      {\rm Spec}\,A \ar@/_/[ddr]_b \ar@/^/[drr]^a \ar@{-->}[dr]^-u \\ 
      & X \ar[d]^h \ar[r]^-j & {\rm Spec}\,B \ar[d]^g \\
      & U \ar[r]^-\iota & \Hi^r_{S/k} .	
    }
  \end{equation}
  For all $i$, we denote the composition of $b$ with the projection $U \to Y_i$ by
  $b_i: {\rm Spec}\,A \to Y_i$. 
  Upon composing that morphism with the immersion $c_i: Y_i \hookrightarrow \Hi^{r_i}_{S/k}$, 
  we obtain a collection of $A$-valued points of $\Hi^{r_i}_{S/k}$, 
  that is to say, a collection of closed subschemes $w_i \subset \mathbb{A}^n_A$, for $i = 1, \ldots, m$. 
  The morphism $g$ is a $B$-valued point $w$ of $\Hi^r_{S/k}$. 
  By Yoneda's Lemma, $g$ corresponds to an element of $\HHi^r_{S/k}({\rm Spec}\,B)$, 
  thus to the equivalence class of a surjective $B$-algebra homomorphism $\phi: S \otimes_k B \to Q$
  such that $Q$ is a locally free $B$-module of rank $r$. 
  After localizing $B$ in a suitable finite collection of elements $f_i\in B$
  such that the ideal spanned by all $f_i$ is the unit ideal in $B$, 
  we may assume that $Q$ is free of rank $r$. 
  Let $q_1, \ldots, q_r$ be a basis of $Q$. 
  Similarly, each $w_i$ corresponds to an element of $\HHi^{r_i}_{S/k}({\rm Spec}\,A)$, 
  thus to the equivalence class of a surjective $B$-algebra homomorphism $\phi_i: S \otimes_k A \to P_i$
  such that $P_i$ is a locally free $A$-module of rank $r_i$. 
  Commutativity of \eqref{bigdiagram} means that $Q \otimes_B A = P_1 \oplus \ldots \oplus P_m$. 
  Consider, for all $i$, the composition of the canonical map, ${\rm can}$, 
  with the $i$-th projection, $\pi_i$, 
  \begin{equation*}
    \xymatrix{ 
      p_i: Q \ar[r]^-{\rm can} & Q \otimes_B A \ar[r]^-{\pi_i} & P_i .
    }
  \end{equation*}
  Take $x \in p_i^{-1}(P_i)$, $y \in  p_j^{-1}(P_j)$, for $i \neq j$, 
  then the equality $p_i(x)p_j(y) = 0$ holds in $Q \otimes_B A$. 
  As the sum $\pi_1 + \ldots + \pi_m$ is the identity, 
  that equality means that ${\rm can}(x){\rm can}(y) = 0$ in $Q \otimes_B A$. 
  Let us write the product $xy \in Q$ in terms of our basis, 
  \begin{equation*}
    xy = \sum_{k = 1}^r d(x,y)_k q_k .
  \end{equation*}
  Then each $d(x,y)_k$ is killed when tensoring with $A$. 
  Accordingly, we consider the following ideal in $B$, 
  \begin{equation*}
    I := \left\langle
    \begin{array}{c}
      d(x,y)_a: \\
      {\rm Spec}\,A \text{ runs through all affine schemes as in } \eqref{bigdiagram}, \\
      x \in p_i^{-1}(P_i), y \in  p_j^{-1}(P_j), \text{ where } i \neq j \in \{1, \ldots, m\}, \\
      \text{ and } a \in \{1, \ldots, r\}
    \end{array}
    \right\rangle .
  \end{equation*}
  We obtain a $B/I$-algebra $Q \otimes_B (B/I)$, which is free of rank $r$ as a module over $B/I$. 
  Hence a homomorphism $\overline{\phi}: S \otimes_k (B/I) \to Q$, 
  which corresponds to a morphism $\overline{g}: {\rm Spec}\, B/I \to \Hi^r_{S/k}$.
  Moreover, the algebra $Q \otimes_B (B/I)$ splits as a direct sum,
  \begin{equation*}
    Q \otimes_B (B/I) = Q_1 \oplus \ldots \oplus Q_m ,
  \end{equation*}
  where $Q_i$ is free of rank $r_i$ as a module over $B/I$.
  Hence, for each $i$, a homomorphism $\phi_i: S \otimes_k (B/I) \to Q_i$, 
  which corresponds to a morphism $\overline{g}_i: {\rm Spec}\, (B/I) \to \Hi^{r_i}_{S/k}$. 
  
  For computing the fibered product $X$, we may replace $g$ by $\overline{g}$. 
  Indeed, given any evaluation of \eqref{cartesiana} as in \eqref{bigdiagram}, 
  we may interpret the composition $g \circ a = \iota \circ b$ as an $A$-valued point of $\Hi^r_{S/k}$, 
  and we have seen that $g \circ a$ definitely factors through $\overline{g}$. 
  Therefore we may replace cartesian diagram \eqref{cartesiana} by cartesian diagram 
  \begin{equation}\label{cartesianb}
    \xymatrix{
      X \ar[d] \ar[r]^-j & {\rm Spec}\,B/I \ar[d]^{\overline{g}} \\
      U \ar[r]^-\iota & \Hi^r_{S/k} .
    }
  \end{equation}
  Moreover, we consider, for each $i$, the cartesian diagram 
  \begin{equation*}
    \xymatrix{
      X_i \ar[d]^{h_i} \ar[r]^-{j_i} & {\rm Spec}\,B/I \ar[d]^{\overline{g}_i} \\
      Y_i \ar[r]^-{c_i} & \Hi^{r_i}_{S/k} .
    }
  \end{equation*}
  As $c_i$ is an immersion, so is $j_i$. 
  Let $\overline{a}$ be the composition of morphism $a$ from \eqref{bigdiagram} with the canonical map $B \to B/I$. 
  Then the morphisms $\overline{a}$ and $b$ make the exterior square of diagram
  \begin{equation}\label{bigdiagrami}
    \xymatrix{
      {\rm Spec}\,A \ar@/_/[ddr]_{b_i} \ar@/^/[drr]^{\overline{a}} \ar@{-->}[dr]^{u_i} \\ 
      & X_i \ar[d]^{h_i} \ar[r]^-{j_i} & {\rm Spec}\,B/I \ar[d]^{\overline{g}_i} \\
      & Y_i \ar[r]^-{c_i} & \Hi^{r_i}_{S/k} ,
    }
  \end{equation}
  commute, hence, by the cartesian property, a unique morphism $u_i$ is induced. 
  Commutativity of \eqref{bigdiagram} guarantees that these morphisms are compatible in the sense that 
  $u_i$ also factors through $X_j$, for all $j \neq i$, 
  more precisely, $u_i = u_j$ as morphisms from ${\rm Spec}\, A$ 
  to the scheme-theoretic intersection 
  \begin{equation*}
    X := X_1 \cap \ldots \cap X_m . 
  \end{equation*}
  We obtain a morphism $u: {\rm Spec}\, A \to X$. 
  Moreover, we define $j$ to be the restriction of any $j_i$ to $X$, 
  and $h$ to be $(h_1, \ldots, h_m)$. 
  The range of that last morphism is $U \subset Y_1 \times \ldots \times Y_m$, 
  as is the range of $b = (b_1, \ldots, b_m)$.
  We obtain a commutative diagram of shape \eqref{bigdiagram}. 
  Note that at this point it is not clear that the scheme $X$ 
  we just defined is the fibered product as in \eqref{cartesiana}. 
  
  At this point, assume hypothesis (a) satisfied. 
  We claim that in this case, $X$ really is the fibered product as in \eqref{cartesiana}. 
  What we have to show is the uniqueness of the above-defined morphism $u$. 
  The composition $\overline{g} \circ \overline{a} = \iota \circ b$ is an $A$-valued point of $\Hi^r_{S/k}$, 
  thus a closed subscheme $Z \subset \mathbb{A}^n_A$. 
  Consider all decompositions of $Z$ into closed subschemes, $Z = Z_1 \cup \ldots \cup Z_m$ 
  such that each $Z_i$ corresponds to an $A$-valued point of $\Hi^{r_i}_{S/k}$. 
  Our hypothesis guarantees that the map of sets $\iota(A)$ is injective. 
  Therefore, not only is the above decomposition of $Z$ unique, even every $Z_i$ appearing in it is unique
  (as opposed to unique up to permutations). 
  In other words, the homomorphism $\phi$ obtained from diagram \eqref{bigdiagram} 
  uniquely determines each homomorphism $\phi_i$. 
  Thus each $u_i$ is uniquely determined, and so is $u$. 
  
  Now assume hypothesis (b) satisfied. 
  As $U$ is quasi-projective and $S_m$ is finite, the quotient $U/S_m$ exists as a scheme. 
  For all $\sigma \in S_m$, we denote by $\sigma: U \to U$ the corresponding automorphism of $U$. 
  Then the morphism $\iota$ has the property that $\iota \circ \sigma = \iota$. 
  Therefore there exists a unique morphism $\overline{\iota}: U/S_m \to \Hi^r_{S_k}$ such that 
  $\iota$ is the composition of the canonical map $U \to U/S_m$ and $\overline{\iota}$ 
  (see Definition 0.5 and Proposition 0.1 of \cite{git}, or Section 6.1 of \cite{dolgachevit}). 
  Moreover, $S_m$ acts on the collection of all $X_i$ by permutation, hence an action of $S_m$ on $X$. 
  The quotient $X/S_m$ exists as a scheme, and analogously as above, we have a morphism 
  $\overline{j}: X/S_m \to {\rm Spec}\, B/I$. 
  Let $\overline{b}$ be the composition of $b$ and the canonical map $U \to U/S_m$, 
  and let $\overline{u}$ be the composition of $u$ and the canonical map $X \to X/S_m$. Moreover, 
  the morphism $h = (h_1, \ldots, h_m)$ induces the morphism $\overline{h}$ appearing in the commutative diagram 
  \begin{equation*}
    \xymatrix{ 
      {\rm Spec}\, A \ar@/_/[ddr]_{\overline{b}} \ar@/^/[drr]^{\overline{a}}
      \ar@{-->}[dr]^-{\overline{u}} \\ 
      & X/S_m \ar[d]^{\overline{h}} \ar[r]^{\overline{j}} & {\rm Spec}\, B/I \ar[d]^{\overline{g}} \\ 
      & U/S_m \ar[r]^-{\overline{\iota}} & \Hi^r_{S/k} .
    }
  \end{equation*}
  We show that $\overline{u}$ is unique, 
  as this will prove that $\overline{\iota}$ is an immersion. 
  Consider, as in the last paragraph, 
  all decompositions of $Z$ into closed subschemes, $Z = Z_1 \cup \ldots \cup Z_m$ 
  such that each $Z_i$ corresponds to an $A$-valued point of $\Hi^{r_i}_{S/k}$.
  Under hypothesis (b), the map of sets $\iota(A)$ is not injective. 
  Therefore, we consider the map of multisets induced by $\iota(A)$. 
  It turns out that this map coincides with $\overline{\iota}(A)$, 
  the evaluation of $\overline{\iota}$ at $A$. 
  (Note that the range of $\overline{\iota}(A)$ coincides with the range of $\iota(A)$.)
  Now $\overline{\iota}(A)$ indeed is injective, 
  and therefore, the multiset of all $Z_i$ appearing the above decomposition of $Z$ unique. 
  In other words, the homomorphism $\phi$ obtained from diagram \eqref{bigdiagram} 
  uniquely determines the multiset of homomorphisms $\phi_i$. 
  Thus the multiset of $u_i$ is uniquely determined, and so is $\overline{u}$. 
\end{proof}

%%%%%%%%%%%%%%%%%%%%%%%%%%%%%%%%%%%%%%%%%%%%%%%%%%%%%%%% 

\section{The Gr\"obner stratum of reduced points}\label{reducedmoduli}

%%%%%%%

\subsection{The Gr\"obner functor of reduced points}

The datum of the equivalence class of a surjective $B$-algebra homomorphism $\phi: S \otimes_k B \to Q$ 
is equivalent to the datum of the morphism of affine schemes $p:Z={\rm Spec}\,Q\to{\rm Spec}\,B$. 
Here $Z$ is a closed subscheme of affine space $\mathbb{A}^n_B$. 
Local freeness of $Q$ translates to flatness and surjectivity of $p$. 
Therefore, the Hilbert functor of $r$ points can be reformulated as follows:
\begin{equation*}
  \begin{split}
    \HHi^{r}_{S/k}:(k{\rm-Alg})&\to({\rm Sets})\\
    B&\mapsto
    \left\lbrace
    \begin{array}{c}
      \text{closed subschemes }Z \subset \mathbb{A}^n_B: \\
      p:Z\to{\rm Spec}\,B \text{ is finite flat and surjective} \\
      \text{ of degree }r
    \end{array}
    \right\rbrace . 
  \end{split}
\end{equation*}
Now consider $(\mathbb{A}^n_k)^r$, the $r$-fold product of affine $n$-space over ${\rm Spec}\,k$. 
The symmetric group $S_{r}$ acts on that product by permuting the factors. Let 
$x^{(i)}=(x^{(i)}_{1},\ldots,x^{(i)}_{n})$ denote the coordinates on the $i$-th factor of $(\mathbb{A}^n_k)^r$, 
and let 
\begin{equation*}
  \Lambda^\prime := \cup_{i\neq j\in\{1,\ldots,r\}}\mathbb{V}(x^{(i)}_{1}-x^{(j)}_{1},\ldots,x^{(i)}_{n}-x^{(j)}_{n})
\end{equation*}  
be the large diagonal in $(\mathbb{A}^n_k)^r$. 
Then clearly $S_r$ also acts on $(\mathbb{A}^n_k)^r \setminus \Lambda^\prime$. 
The quotient by that action,
\begin{equation*}
  \Hi^{r,\et}_{S/k} = ((\mathbb{A}^n_k)^r \setminus \Lambda^\prime)/S_{r} ,
\end{equation*}
is a scheme, as $(\mathbb{A}^n_k)^r \setminus \Lambda^\prime$ is quasi-projective and $S_{r}$ is finite. 
Its dimension is $nr$. Moreover, by \cite{bertin}, Section 2.1, Proposition 2.4, this scheme represents the functor
\begin{equation*}
  \begin{split}
    \HHi^{r,\et}_{S/k}:(k{\rm-Alg})&\to({\rm Sets})\\
    B&\mapsto
    \left\lbrace 
    \begin{array}{c}
      \text{ closed subschemes }Z \subset \mathbb{A}^n_k: \\
      p:Z\to{\rm Spec}\, B \text{ is finite \'etale and surjective} \\
      \text{ of degree }r
    \end{array}
    \right\rbrace .
  \end{split}
\end{equation*}
Upon comparing this definition to the reformulation of $\HHi^{r}_{S/k}$ we gave above,
we see that the additional requirement here is unramifiedness of $p$. 
The scheme $\HHi^{r,\et}_{S/k}$ is the {\it \'etale part} of $\HHi^{r}_{S/k}$, 
or the {\it Hilbert scheme of reduced points}. 
It is the moduli space of families of $r$ distinct points in affine space. 

\begin{dfn}
  We define the following two functors $(k{\rm-Alg})\to({\rm Sets})$, 
  \begin{equation*}
    \begin{split}
      \HHi^{\Delta,\et}_{S/k}:B&\mapsto\HHi^{\Delta}_{S/k}(B)\cap\HHi^{r,\et}_{S/k}(B) ,\\
      \HHi^{\prec\Delta,\et}_{S/k}:B&\mapsto\HHi^{\prec\Delta}_{S/k}(B)\cap\HHi^{r,\et}_{S/k}(B) ,
    \end{split}
  \end{equation*}
  and call $\HHi^{\prec \Delta,\et}_{S/k}$ the 
  {\it Gr\"obner stratum of reduced points with standard set $\Delta$}. 
\end{dfn}

Note that functoriality of $\HHi^{\Delta,\et}_{S/k}$ and $\HHi^{\prec\Delta,\et}_{S/k}$ is clear, 
as an intersection of two functors from an arbitrary category to the category of sets is always a functor. 

%%%%%%%

\subsection{Representability}

\begin{pro}\label{etale}
  $\HHi^{\Delta,\et}_{S/k}$ is an open subfunctor of $\HHi^{\Delta}_{S/k}$, and 
  $\HHi^{\prec\Delta,\et}_{S/k}$ is an open subfunctor of $\HHi^{\prec\Delta}_{S/k}$. 
  In particular, these functors are representable by open subschemes
  $\Hi^{\Delta,\et}_{S/k} \subset \Hi^{\Delta}_{S/k}$ and 
  $\Hi^{\prec\Delta,\et}_{S/k} \subset \Hi^{\prec\Delta}_{S/k}$, resp. 
\end{pro}

\begin{proof}
  Only the first part of the theorem requires proof, as an open subfunctor of a functor representable
  by a scheme $X$ is always representable by an open subscheme of $X$. 
  
  We first show the assertion for $\HHi^{\Delta,\et}_{S/k}$. 
  We have to show that for all affine schemes ${\rm Spec}\,B$ over ${\rm Spec}\,k$ and all morphisms of functors
  $g:h_{{\rm Spec}\,B}\to\HHi^{\Delta,\et}_{S/k}$, the upper horizontal arrow in the cartesian diagram
  \begin{equation*}
    \xymatrix{ 
      \mathscr{G} \ar[d] \ar[r] & h_{{\rm Spec}\,B} \ar[d]_g \\ 
      \HHi^{\Delta,\et}_{S/k} \ar[r]^-j & \HHi^\Delta_{S/k} .	
    }
  \end{equation*}
  where $j$ is the canonical inclusion of functors, 
  corresponds to the inclusion of an open subscheme of ${\rm Spec}\,B$ into ${\rm Spec}\,B$. 
  This means that there exists an ideal $I \subset B$ such that 
  the morphism of functors $\mathscr{G}\to h_{{\rm Spec}\,B}$ isomorphic to
  $h_{{\rm Spec}\,B \setminus {\rm Spec}\,B/I}\to h_{{\rm Spec}\,B}$. 
  In other words, we have to find an ideal $I$ such that for all affine test schemes ${\rm Spec}\,A$ over ${\rm Spec}\,B$, 
  we have
  \begin{equation*}
    \mathscr{G}({\rm Spec}\,A)=\{a\in h_{{\rm Spec}\,B}({\rm Spec}\,A): \alpha(I)\cdot A=A\} .
  \end{equation*}
  Here the morphism of schemes $a:{\rm Spec}\,A\to{\rm Spec}\,B$ corresponds to the 
  ring homomorphism $\alpha:B\to A$. 

  By Yoneda's Lemma, our given $g$ corresponds to an element $\phi$ of $\HHi^{\Delta}_{S/k}(B)$, 
  i.e., a surjective $B$-algebra homomorphism $\phi: S \otimes_k B \to Q$ such that the composition
  with the canonical inclusion $\iota:Bx^\Delta \to S \otimes_k B$ is an isomorphism of $B$-modules. 
  Therefore 
  \begin{equation*}
    \mathscr{G}({\rm Spec}\,A)=\left\lbrace (a,b)\in 
    h_{{\rm Spec}\,B}({\rm Spec}\,A)\times\HHi^{\Delta,\et}_{S/k}({\rm Spec}\,A): 
    g(a)=j(b) \right\rbrace .
  \end{equation*}
  We make the following identifications regarding the objects appearing in that set:
  \begin{itemize}
    \item From the identification of $a:{\rm Spec}\,A\to{\rm Spec}\,B$ and $\alpha:B\to A$, 
      we get an identification of $g(a)$ and $\phi\otimes\alpha: S \otimes_k B \otimes_{B}A = S \otimes_k A \to Q\otimes_{B}A$.
    \item $b\in\HHi^{\Delta,\et}_{S/k}({\rm Spec}\,A)$ corresponds to a $A$-algebra homomorphism 
      $\beta: S \otimes_k A \to Q^\prime$ such that the composition 
      $Ax^\Delta \to S \otimes_k A \to Q^\prime$ is an isomorphism and 
      $p^\prime:Z^\prime={\rm Spec}\,Q^\prime\to{\rm Spec}\,A$ is \'etale and surjective. 
    \item Therefore, the condition $g(a)=j(b)$ uniquely determines
      $\beta=\phi\otimes\alpha$. In particular, $b$ is uniquely determined. 
    \item The composition $Ax^\Delta \to S \otimes_k A \to Q^\prime$
      of the canonical inclusion with $\beta$ is automatically an isomorphism.
  \end{itemize}
  Therefore 
  \begin{equation*}
      \mathscr{G}({\rm Spec}\,A)=
      \left\lbrace
      \begin{array}{c}
        k\text{-algebra homomorphisms }\alpha:B\to A \text{ s.t. } \\
        \phi\otimes\alpha: S \otimes_k A \to Q\otimes_{B}A\text{ corresponds to \'etale surjective} \\
        p^\prime:Z^\prime={\rm Spec}\,S \otimes_k A/\ker(\phi\otimes\alpha)\to{\rm Spec}\,A
      \end{array}
      \right\rbrace .
  \end{equation*}
  
  We are ready to construct the desired ideal $I \subset B$. 
  The $k$-algebra homomorphism $\pi: B \to S \otimes_k B/\ker\phi$ corresponds to the morphism
  $p:Z={\rm Spec}\,S \otimes_k B/\ker\phi\to{\rm Spec}\,B$ of affine schemes. By SGA 1, Expos\'e I, Corollaire 3.3 \cite{sga1}, 
  the locus in $Z$ where $p$ is unramified is open in $Z$. 
  As by assumption $p$ is flat, that locus is the locus in $Z$ where $p$ is \'etale. 
  We denote its complement by $Y$, a closed subset of $Z$. 
  As the morphism $p$ is finite, it is universally closed (EGA II, 6.1.10 \cite{ega2}), 
  and therefore $p(Y)$ is closed in ${\rm Spec}\,B$. We give this closed set the reduced subscheme structure, 
  thus write it as the closed subscheme $p(Y)={\rm Spec}\,B/I$, for a suitable ideal $I \subset B$. 
    
  From the above description of $\mathscr{G}({\rm Spec}\,A)$, we get, 
  for each morphism $p^\prime:Z^\prime\to{\rm Spec}\,A$ appearing in that description, a cartesian diagram
  \begin{equation}\label{surjective}
    \xymatrix{ 
      Z^\prime \ar[d]_{p^\prime} \ar[r]^r & Z \ar[d]^p \\ 
      {\rm Spec}\,A \ar[r]^a & {\rm Spec}\,B .	
    }
  \end{equation}
  We claim that $p^\prime$ is \'etale and surjective if, and only if, 
  $a({\rm Spec}\,A)$ is contained in ${\rm Spec}\,B \setminus {\rm Spec}\,B/I$. 
  As the latter assertion is equivalent to $A=A\cdot\alpha(I)$, a proof of the claim will finish the proof of 
  the assertion for $\HHi^{\Delta,\et}_{S/k}$ of the proposition.
  Now $p^\prime$ is unramified (hence \'etale) in a point $x^\prime\in Z^\prime$ if, 
  and only if, $\Omega^1_{Z^\prime/{\rm Spec}\,A}(x^\prime)=0$, 
  where $\Omega^1_{Z^\prime/{\rm Spec}\,A}$ is the sheaf of differentials 
  (see SGA 1, Expos\'e I, \S1 \cite{sga1}). 
  By the value of that sheaf in $x^\prime$, we are referring to its stalk. 
  However, by Nakayama's Lemma, that stalk is zero if, and only if, 
  its tensor product with the residue field in $x^\prime$ is zero. We may therefore identify
  \begin{equation*}  
    \Omega^1_{Z^\prime/{\rm Spec}\,A}(x^\prime)
    =\Omega^1_{Z^\prime/{\rm Spec}\,A}\otimes_{\mathcal{O}_{Z}}\kappa(x^\prime) .
  \end{equation*}
  In the cited section of SGA 1, it is also shown that $\Omega^1_{Z^\prime/{\rm Spec}\,A}$ 
  is well behaved under base change, i.e., 
  $\Omega^1_{Z^\prime/{\rm Spec}\,A}=r^*\Omega^1_{Z/{\rm Spec}\,B}$. 
  Therefore, upon writing $x=r(x^\prime)$, we get
  \begin{equation*}  
    \Omega^1_{Z^\prime/{\rm Spec}\,A}(x^\prime)
    =(r^*\Omega^1_{Z/{\rm Spec}\,B})(x^\prime)
    =\Omega^1_{Z/{\rm Spec}\,B}(x)\otimes_{\kappa(x)}\kappa(x^\prime) .
  \end{equation*}
  As the homomorphism $\kappa(x)\to\kappa(x^\prime)$ is just a field extension, 
  this equation shows that the $\kappa(x^\prime)$-vector space 
  $\Omega^1_{Z^\prime/{\rm Spec}\,A}(x^\prime)$ is zero if, and only if, the $\kappa(x)$-vector space 
  $\Omega^1_{Z/{\rm Spec}\,B}(x)$ is zero. 
  In other words, $p^\prime$ is unramified (hence \'etale) in $x^\prime$ if, and only if, 
  $p$ is unramified (hence \'etale) in $x$. 
  
  Therefore, the following statements are equivalent. 
  \begin{itemize}
    \item $p^\prime$ is \'etale;
    \item for all $x^\prime\in Z^\prime$, $p^\prime$ is \'etale in $x^\prime$;
    \item for all $x^\prime\in Z^\prime$, the point $x=r(x^\prime)$ does not lie in $p^{-1}({\rm Spec}\,B/I)$;
    \item $p\circ r(Z^\prime) \subset {\rm Spec}\,B \setminus {\rm Spec}\,B/I$;
    \item $a\circ p^\prime(Z^\prime) \subset {\rm Spec}\,B \setminus {\rm Spec}\,B/I$; and
    \item $a({\rm Spec}\,A) \subset {\rm Spec}\,B \setminus {\rm Spec}\,B/I$. 
  \end{itemize}
  (For the transition between the second and the third bulleted items, 
  we use the fact that a point $x$ in the fiber of $p$ over some $y\in{\rm Spec}\,B$ lies in the image of $r$ if, 
  and only if, all points in the fiber of $p$ over $y$ lie in the image of $r$. 
  This follows from EGA I, 3.4.8 \cite{ega1}.)
  The proof of the assertion for $\HHi^{\Delta,\et}_{S/k}$ of the proposition is complete.
  
  The proof for $\HHi^{\prec \Delta,\et}_{S/k}$ follows the same line of argument. 
  We have to show that the upper horizontal arrow in the cartesian diagram
  \begin{equation*}
    \xymatrix{ 
      \mathscr{G}^\prime \ar[d] \ar[r] & h_{{\rm Spec}\,B} \ar[d]^g \\ 
      \HHi^{\prec \Delta,\et}_{S/k} \ar[r]^j & \HHi^{\prec \Delta}_{S/k} .
    }
  \end{equation*}
  corresponds to the inclusion of an open subscheme of ${\rm Spec}\,B$ in ${\rm Spec}\,B$. 
  Analogously as above, we can rewrite the fibered product as
  \begin{equation*}
    \mathscr{G}^\prime({\rm Spec}\,A) = 
    \left\lbrace
    \begin{array}{c}
      k\text{-algebra homomorphisms }\alpha:B\to A\text{ s.t. } \\
      \phi \otimes \alpha: S \otimes_k A \to Q\otimes_{B}A\text{ corresponds to \'etale surjective}\\
      p^\prime:Z^\prime={\rm Spec}\,S \otimes_k A/\ker(\phi\otimes\alpha)\to{\rm Spec}\,A\text{, where}\\
      \ker(\phi\otimes\alpha)\text{ is monic with standard set }\Delta
    \end{array}
    \right\rbrace .
  \end{equation*}
  The above proof of openness of the morphism $\mathscr{G}\to h_{{\rm Spec}\,B}$ 
  goes through also for the morphism $\mathscr{G}^\prime\to h_{{\rm Spec}\,B}$.
\end{proof}

Note that in the above proof we did not mention surjectivity of $p^\prime$ in \eqref{surjective}, 
implicitly taking it for granted. 
Indeed, the morphism $p$ at the right hand side of \eqref{surjective} 
corresponds to an element of $\HHi^{\Delta,\et}_{S/k}(B)$, 
the image of which under the map $\HHi^{\Delta,\et}_{S/k}(\beta)$ corresponds to $p^\prime$. 
Therefore $p^\prime$ lies in $\HHi^{\Delta,\et}_{S/k}(A)$ and is thus surjective by definition. 
However, the reason for surjectivity of $p^\prime$, which is necessary for functoriality of $\HHi^{\Delta,\et}_{S/k}$, 
is found in EGA I, 3.4.8, the same reference we cited in the above proof. 

%%%%%%%%%%%%%%%%%%%%%%%%%%%%%%%%%%%%%%%%%%%%%%%%%%%%%%%% 

\section{The Connect Four morphism on reduced points}\label{connectonreduced}

%%%%%%%

\subsection{Restriction to the Gr\"obner stratum of reduced points}\label{restriction}

Now we restrict the C4 morphism $\tau:Y^{\Delta}\to\Hi^{\prec \Delta}_{S/k}$
of Corollary \ref{connectfourmultiset} to that open subscheme of $Y^{\Delta}$ 
which parametrizes \'etale rather than flat families. 
More precisely we consider, for each C4 decomposition $\{\Delta_i: i\in I\}$ of $\Delta$, the open subscheme 
\begin{equation*}
  Y^{I,\et} := \Bigl( \prod_{i \in I}\Hi^{\prec \Delta_i,\et}_{\overline{S}/k}\times\mathbb{A}^1_k \setminus \Lambda \Bigr) / G
\end{equation*}
of $Y^{I}$, where the {\it large diagonal}, $\Lambda$, is defined as in \eqref{large}; 
and the disjoint sum over all these schemes, 
\begin{equation*}
  Y^{\Delta,\et} := \coprod_{I\in\mathcal{I}}Y^{I,\et} .
\end{equation*}

\begin{pro}\label{restriction0}
  The restriction of the C4 morphism to $Y^{\Delta,\et}$ factors through 
  $\Hi^{\prec \Delta,\et}_{S/k}$, thus defining a morphism 
  \begin{equation*}
    \tau^\et:Y^{\Delta,\et}\to\Hi^{\prec \Delta,\et}_{S/k} .
  \end{equation*}
\end{pro}

\begin{proof}
  We use the description of $\tau$ by its corresponding morphism of functors $h_{\tau}$, 
  see Corollary \ref{connectfourmultiset}.
  Let $B$ be a $k$-algebra with no nontrivial idempotents, and let $\phi$ be an element of $Y^{\Delta,\et}(B)$,
  given by a set of ideals $\left\lbrace \langle J_i \rangle + \langle x_{n}-b_i \rangle: i \in I \right\rbrace$ 
  as in Section \ref{alldecompositions}, 
  for one C4 decomposition $\{\Delta_i: i\in I\}$ of $\Delta$. 
  To say that this set lies in $Y^{\Delta,\et}(B)$ means that, in addition to the constraints on 
  the ideals $J_i \subset \overline{S} \otimes_k B$ and the differences $b_i-b_{j}$ we discussed in Sections 
  \ref{minusbadpoints} and \ref{alldecompositions}, each morphism
  \begin{equation*}
    {\rm Spec}\,S \otimes_k B/(\langle J_i \rangle + \langle x_{n}-b_i \rangle )\to{\rm Spec}\,B
  \end{equation*}
  is also \'etale. By the Chinese Remainder Theorem, we have an isomorphism of $k$-algebras
  \begin{equation*}
    S \otimes_k B/\cap_{i \in I} \bigl( \langle J_i \rangle + \langle x_{n}-b_i \rangle \bigr) \cong 
    \oplus_{i \in I}S \otimes_k B / \bigl( \langle J_i \rangle + \langle x_{n}-b_i \rangle \bigr) ,
  \end{equation*}
  hence an isomorphism of schemes
  \begin{equation*}
    {\rm Spec}\,S \otimes_k B/\cap_{i \in I} \bigl( \langle J_i \rangle + \langle x_{n}-b_i \rangle \bigr) \cong
    \coprod_{i \in I}{\rm Spec}\,S \otimes_k B / \bigl( \langle J_i \rangle + \langle x_{n}-b_i \rangle \bigr) .
  \end{equation*}
  Therefore, also the morphism 
  \begin{equation*}
    {\rm Spec}\,S \otimes_k B/\cap_{i \in I} \bigl( \langle J_i \rangle + \langle x_{n}-b_i \rangle \bigr)\to{\rm Spec}\,B
  \end{equation*}
  is \'etale. This means that the image of $\phi$ under $h_{\tau}(B)$ lies in $\Hi^{\prec \Delta,\et}_{S/k}$.
\end{proof}

%%%%%%%

\subsection{Galois descent}

Proposition \ref{descent} below will be the key to proving our main theorem. 
As a preparation, we need the following lemma, which is in fact a corollary to the main theorem of \cite{jpaa}. 

\begin{lmm}\label{invariance}
  Let $F$ be a field, $\overline{F}$ its algebraic or separable closure and $A \subset \overline{F}^n$ a finite set. 
  For $\lambda \in \overline{F}$, we write $A_{\lambda} := A\cap\{x_{n} = \lambda\}$. 
  Let $D(A)\in\mathcal{D}_{n}$ be the standard set (w.r.t. $\prec$) 
  of the ideal in $S \otimes_k \overline{F}$ defining $A$. 
  Let $D(A_{\lambda})\in\mathcal{D}_{n-1}$ be the standard set (w.r.t. $\prec$) 
  of the ideal in $\overline{S} \otimes_k \overline{F}$ defining $A_{\lambda}$, 
  where we view the latter set as a subset of $\overline{F}^{n-1}$. 
  Then for all $\sigma\in{\rm Gal}(\overline{F}/F)$, the following equalities hold,
  \begin{enumerate}
    \item[(i)] $D(A)=D(\sigma A)$;
    \item[(ii)] $D(A_{\lambda})=D((\sigma A)_{(\sigma\lambda)})$.
  \end{enumerate}
\end{lmm}

\begin{proof}
  Assertion (i) holds for $n-1$ (more precisely, for all $A \subset \overline{F}^{n-1}$) if, and only if,
  assertion (ii) holds for $n$. 
  Indeed, take $A \subset \overline{F}^{n-1}$ and assume that (ii) is proven for $n$. 
  We embed $A$ into $\overline{F}^n$ by sending each $a\in A$ to $(a,\lambda)$, 
  where $\lambda$ is a fixed element of $F$ and call the result $\widehat{A}_\lambda$. 
  Under the usual identification, we get $D(A)=D(\widehat{A}_{\lambda})$. 
  By (ii), that set equals $D((\sigma\widehat{A})_{(\sigma\lambda)})$. 
  As $\lambda$ lies in $F$, we have $(\sigma\widehat{A})_{\sigma\lambda}=(\sigma\widehat{A})_{\lambda}$,
  hence $D((\sigma\widehat{A})_{\sigma\lambda})=D((\sigma\widehat{A})_{\lambda})=D(\sigma A)$.
  Conversely, take $A \subset \overline{F}^{n}$ and assume that (i) is proven for $n-1$.
  The intersection $(\sigma A)_{\sigma\lambda}$ consist of all $\sigma a=(\sigma a_{1},\ldots,\sigma a_{n})$
  such that $\sigma a_{n}=\sigma\lambda$. The latter condition is equivalent to $a_{n}=\lambda$. 
  Therefore, $(\sigma A)_{\sigma\lambda}$ is just the transform under $\sigma$ of $A_{\lambda}$. 
  Upon understanding $A_{\lambda}$ to be a subset of $\overline{F}^{n-1}$, 
  we may apply (i) and get $D(A_{\lambda})=D((\sigma A)_{(\sigma\lambda)})$. 
  
  We show assertion (i) by induction over $n$. The case $n=1$ is trivial, 
  as $D(A)=\{0,\ldots,\#A-1\}=\{0,\ldots,\#(\sigma A)-1\}$. 
  For $n>1$, the main theorem of \cite{jpaa} tells us that $D(A)=\sum_{\lambda\in\overline{F}}D(A_{\lambda})$. 
  As we may assume that (ii) holds for $n$, we get
  $D(\sigma A)=\sum_{\lambda\in\overline{F}}D((\sigma A)_{(\sigma\lambda)})
  =\sum_{\lambda\in\overline{F}}D(A_{\lambda})=D(A)$.
\end{proof}

\begin{pro}\label{descent}
  The morphism $\tau^\et:Y^{\Delta,\et}\to\Hi^{\prec \Delta,\et}_{S/k}$ is bijective on closed points.
\end{pro}

\begin{proof}
  Injectivity follows from Theorem \ref{immersion}.
  As for surjectivity, we take a closed point $y\in\Hi^{\prec \Delta,\et}_{S/k}$
  and consider its residue field $F := \kappa(y)$. 
  Thus $y$ is an $F$-valued point of $\Hi^{\prec \Delta,\et}_{S/k}$. 
  After base change to the algebraic or separable closure $\overline{F}/F$, the product
  $y\times_{{\rm Spec}\,F}{\rm Spec}\,\overline{F}$ contains an $\overline{F}$-valued point of
  $\Hi^{\prec \Delta,\et}_{S/k}\times_{{\rm Spec}\,k}{\rm Spec}\,\overline{F}
  =\Hi^{\prec \Delta,\et}_{S \otimes_k \overline{F}/\overline{F}}$. 
  That point corresponds to a set of $r$ closed $\overline{F}$-rational points in $\mathbb{A}^n_{\overline{F}}$
  which we denote by $A \subset \overline{F}^n$. 
  For the attached standard set, we clearly have $D(A)=\Delta$. 
  The transition from $A$ back to $y$ is established by Galois descent: 
  By Lemma \ref{invariance} (i), the standard set of $A$ is invariant under the action of the Galois group of $F$. 
  In concrete terms, for all $\sigma\in{\rm Gal}(\overline{F}/F)$, we have
  \begin{equation*}
    D(\sigma A)=\Delta .
  \end{equation*}
  Therefore the set $(\sigma A)_{\sigma\in{\rm Gal}(\overline{F}/F)}$ 
  defines an $\overline{F}$-valued point of $\Hi^{\prec \Delta,\et}_{S/k}$. 
  That point is just $y$. 

  We can also, however, slice $A$ into a number of horizontal pieces, as we did in Lemma \ref{invariance}. 
  Let $\{\lambda_i: i\in I\}$ be the set of values taken by the $n$-th coordinates of elements of $A$.
  Thus $\{\lambda_i: i\in I\}$ is a finite subset of $\overline{F}$. 
  Upon interpreting each slice $A_{\lambda_i}$ as being a subset of $\overline{F}^{n-1}$,
  we denote the corresponding ideal by $J_{\lambda_i} \subset \overline{S} \otimes_k \overline{F}$, 
  and the standard set of that ideal as $\Delta_i := D(A_{\lambda_i})\in\mathcal{D}_{n}$. 
  The set of ideals 
  \begin{equation}\label{kfamily}
    \left\lbrace \langle J_{\lambda_i} \rangle + \langle x_{n}-\lambda_i \rangle : i \in I \right\rbrace
  \end{equation}
  in $S \otimes_k \overline{F}$ is an $\overline{F}$-valued point of $Y^{I,\et}\times_{{\rm Spec}\,k}{\rm Spec}\,\overline{F}$. 
  Moreover, by Lemma \ref{invariance} (ii), the collection of standard sets $D(A_{\lambda_i})$, 
  where $i$ runs through $I$, is invariant under the action of the Galois group of $F$. 
  In concrete terms, for all $\sigma\in{\rm Gal}(\overline{F}/F)$, we have an equality of sets, 
  \begin{equation*}
    \left\lbrace D((\sigma A)_{(\sigma\lambda_i)}): i\in I \right\rbrace = \left\lbrace \Delta_i: i\in I \right\rbrace .
  \end{equation*}
  Upon writing $\widehat{J}$ for the tuple in \eqref{kfamily}, 
  invariance tells us that the set $\{ \sigma \widehat{J} : \sigma \in {\rm Gal}(\overline{F}/F) \}$ 
  defines an $\overline{F}$-valued point of $Y^{I,\et}$, 
  therefore in particular an $\overline{F}$-valued point of $Y^{\Delta,\et}$.
  The image under $\tau$ of that point is just $y$. 
\end{proof}

%%%%%%%

\subsection{Proving the main results}\label{theresult}

We can now prove Theorem \ref{mainthm}, i.e., 
count the number of components of $\Hi^{\prec \Delta,\et}_{S/k}$ and determine its dimension. 
After that, we will make some statements on that the number of components of $\Hi^{\prec \Delta}_{S/k}$ and its dimension. 

\begin{proof}[Proof of Theorem \ref{mainthm}]
  (i). By Theorem \ref{immersion} and Proposition \ref{restriction0}, $\tau^\et$ is an immersion. 
  In particular, as a map of topological spaces, 
  $\tau^\et$ is a homeomorphism from $Y^{\Delta,\et}$ to a subspace of $\Hi^{\prec \Delta,\et}_{S/k}$. 
  We claim that this subspace is all of $\Hi^{\prec \Delta,\et}_{S/k}$. 
  
  For proving this, we first reduce to the situation where $k=\mathbb{Z}$. 
  The construction of $\Hi^{\prec \Delta}_{S/k}$ given in Section 8 of \cite{strata}, shows that 
  $\Hi^{\prec \Delta}_{S/k}$ arises from $\Hi^{\prec \Delta}_{\mathbb{Z}[x]/\mathbb{Z}}$ by base change, 
  \begin{equation*}
    \Hi^{\prec \Delta}_{S/k}=\Hi^{\prec \Delta}_{\mathbb{Z}[x]/\mathbb{Z}}
    \times_{{\rm Spec}\,\mathbb{Z}}{\rm Spec}\,k .
  \end{equation*}
  The reason for this is the fact that $\Hi^{\prec \Delta}_{S/k}$ 
  is the closed subscheme of a certain affine space over $k$ defined by an ideal with integer coefficients. 
  Analogously, the non-\'etale part of $\Hi^{\prec \Delta}_{S/k}$ is defined by an ideal with integer coefficients. 
  Therefore, $\Hi^{\prec \Delta,\et}_{S/k}$ arises from $\Hi^{\prec \Delta,\et}_{\mathbb{Z}[x]/\mathbb{Z}}$
  by base change. We see that it suffices to prove the claim in the case where $k=\mathbb{Z}$. 
  
  From the construction given in Section 8 of \cite{strata},
  we also see that $\Hi^{\prec \Delta,\et}_{\mathbb{Z}[x]/\mathbb{Z}}$ is of finite type over ${\rm Spec}\,\mathbb{Z}$. 
  Now if the image of $\tau^\et$ were smaller than $\Hi^{\prec \Delta,\et}_{\mathbb{Z}[x]/\mathbb{Z}}$, 
  its complement would clearly be a subscheme of finite type over ${\rm Spec}\,\mathbb{Z}$. 
  By Hilbert's Nullstellensatz, however, such a scheme contains closed points. 
  This is a contradiction to Proposition \ref{descent}, and the claim is proved. 
    
  (ii). We use induction over $n$. 
  We start with \eqref{quotient}, the definition of $Y^I$, from which we see that 
  \begin{equation}\label{y0quot}
    Y^{I,\et} = \widehat{Y}^{I,\et} / G .
  \end{equation}
  As the canonical map $\widehat{Y}^{I,\et} \to Y^{I,\et}$ has finite fibers, 
  it follows that in the passage from $\widehat{Y}^{I,\et}$ to $Y^{I,\et}$, 
  the dimensions of all components of the respective schemes does not change. 
  Now for $n=1$, we have 
  \begin{equation*}
    \Hi^{\prec \Delta,\et}_{S/k}=\Hi^{r,\et}_{S/k}=(\mathbb{A}_k^r \setminus \Lambda)/S_{r} .
  \end{equation*}
  This is an irreducible scheme 
  of relative dimension $r=\#\Delta=\#q_{1}(\Delta)$ over ${\rm Spec}\,k$. 
  For $n>1$, we use the fact that $\tau^\et$ is a homeomorphism; we use the equality $\dim Y^{I,\et} = \dim Y^I$; 
  and we determine the dimension of the subscheme $Y^I$ of $Y^\Delta$, for an arbitrary $I$.
  From \eqref{full}, we obtain that 
  \begin{equation*}
    \dim(Y^{I})=\sum_{i \in I}\dim(\Hi^{\prec \Delta_i,\et}_{\overline{S}/k}+1)
    =\sum_{i \in I}\dim(\Hi^{\prec \Delta_i,\et}_{\overline{S}/k})+\#q_{n}(\Delta) .
  \end{equation*}
  Here we used the identity $\#I=\#q_{n}(\Delta)$, mentioned in Section \ref{subsectiondecomps}. 
  By induction hypothesis, we have
  \begin{equation*}
    \dim(\Hi^{\prec \Delta_i,\et}_{\overline{S}/k})=\sum_{j=1}^{n-1}\#\overline{q}_{j}(\Delta_i) ,
  \end{equation*}
  for all $i\in I$, where $\overline{q}_{j}$ is the analogue of $q_{j}$ of \eqref{qj}, with $n$ replaced by $n-1$. 
  It is easy to see that 
  $\sum_{i \in I}\sum_{j=1}^{n-1}\#\overline{q}_{j}(\Delta_i) + \#q_n(\Delta) = \sum_{j = 1}^n \#q_{j}(\Delta)$,
  from which the assertion on the relative dimension follows. 
  Equidimensionality follows as well, as we did not pose any restriction on the C4 decomposition, or, equivalently,
  the indexing set $I$, we started with. 
  
  (iii) and (iv). This is also proved by induction over $n$. 
  For $n=1$, the scheme $\Hi^{\prec \Delta,\et}_{S/k}$ is an open subscheme of 
  affine space $\mathbb{A}^r_k$. 
  Therefore, that scheme has exactly one connected, and irreducible, resp., component. 
  This is in accord with $d(\Delta) = 1$.
  If $n\geq2$, we argue as follows.
  The number of connected, and irreducible, resp., components of $\Hi^{\prec\Delta,\et}_{S/k}$
  equals the sum of the number of connected, and irreducible, resp., components of $Y^{I,\et}$, 
  where $I$ runs through all (indexing sets of) C4 decompositions of $\Delta$. 
  We once more use the characterization \eqref{y0quot} of $Y^{I,\et}$, rewriting that scheme as
  \begin{equation*}
    \begin{split}
      Y^{I,\et} & = \bigl( ( \prod_{i \in I}\Hi^{\prec \Delta_i,\et} \times \mathbb{A}^1_k ) \setminus \Lambda \bigr) / G \\
       & = \biggl( \prod_{j = 1}^m
       \Bigl( \bigl( (\prod_{i \in I_j} \Hi^{\prec \Delta_i,\et}_{S/k} \times \mathbb{A}^1_k) 
       \setminus \Lambda_j \bigr) / S_{h_j} \Bigr) \biggr)
       \setminus \widetilde{\Lambda} , \text{ where } \\
      \Lambda_j & := \cup_{i \neq a \in I_j} \mathbb{V}(y_i - y_a) , \text{ and } \\
      \widetilde{\Lambda} & := \text{ image of } \Lambda \text{ under the canonical map.}
    \end{split}
  \end{equation*}
  Here we also used the decomposition \eqref{ij} of $I$ into $I_j$. 
  
  Let us first discuss irreducible components of $Y^{I,\et}$. 
  In the above description of $Y^{I,\et}$, we may put 
  each $\Lambda_j$ back into the product $(\prod_{i \in I_j} \Hi^{\prec \Delta_i,\et}_{S/k} \times \mathbb{A}^1_k)$, 
  thereby keeping the number of irreducible components as it is. 
  As $\mathbb{A}^1_k$ is irreducible, we may dispose of each factor $\mathbb{A}^1_k$. 
  Moreover, we may put $\widetilde{\Lambda}$ back into the space. 
  Therefore, it suffices to determine the number of irreducible components of the space
  \begin{equation*}
    Z^I := \prod_{j = 1}^m
       \bigl( \prod_{i \in I_j} \Hi^{\prec \Delta_i,\et}_{S/k} / S_{h_j} \bigr) .
  \end{equation*}
  By induction hypothesis, 
  we may assume that the number of irreducible components of $\Hi^{\prec \Delta_i,\et}_{S/k}$ equals $d(I_j)$, 
  the C4 decomposition number of $\Delta_i$, for any $i \in I_j$. 
  Denote by $Z_1, \ldots, Z_{d(I_j)}$ the irreducible components. 
  Then it is not hard to see that the decomposition of $Z^I$ into irreducible components is given by
  \begin{equation}\label{zi}
    Z^I = \bigcup_{A_1, \ldots ,A_m} 
    \prod_{j=1}^m \bigl( \prod_{a_j \in A_j} Z_{a_j} / S_{h_j} \bigr) ,
  \end{equation}  
  where each $A_j$ runs through all multisets in $I_j$ of size $h_j$, 
  and $S_{h_j}$ acts on $\prod_{a_j \in A_j} Z_{a_j}$ by permuting only those factors 
  In particular, it follows that the number of irreducible components of $Z^I$ is given by 
  \begin{equation*}
    \prod_{j = 1}^m {d(I_j) + h_j - 1 \choose h_j}
  \end{equation*}
  Upon summing over all C4 decompositions of $\Delta$, indexed by various $I$, 
  we see that the number of irreducible components of $\Hi^{\prec \Delta,\et}_{S/k}$
  satisfies the functional equation \eqref{functionald}. 
  Analogously as in the proof of Lemma \ref{graphlemma}, 
  we obtain that the number of irreducible components of $\Hi^{\prec \Delta,\et}_{S/k}$ equals $d(\Delta)$. 
  
  We now show that for the space $\Hi^{\prec \Delta,\et}_{S/k}$, 
  its number of connected components equals its number of irreducible components. 
  We first make the transition back from $Z^I$ to $Y^{I,\et}$, 
  putting each factor $\mathbb{A}^1_k$ back in, removing each $\Lambda_j$ and removing $\widetilde{\Lambda}$. 
  Then the decomposition \eqref{zi} of $Z^I$ corresponds to a decomposition of $Y^{I,\et}$, 
  \begin{equation*}
    Y^{I,\et} = \cup_{A_1, \ldots ,A_m} 
    \biggl( \prod_{j=1}^m 
    \Bigl( \bigl( (\prod_{a_1 \in A_1} Z_{a_1} \times \mathbb{A}^1_k) \setminus \Lambda_1 \bigr) / S_{h_1} \Bigr) 
    \biggr) \setminus \widetilde{\Lambda} .
  \end{equation*}
  By construction of $\Lambda_j$ and $\widetilde{\Lambda}$, this union is in fact a coproduct. 
  Now the claim follows from an elementary observation: 
  Assume that a space $X$ admits a decomposition into irreducible components, $X = X_1 \cup \ldots \cup X_l$. 
  Then if this union is a coproduct, it is also a decomposition into connected components. 
\end{proof}

%%%%%%%

\subsection{Non-generalizations}\label{negative}

Remember that in Section \ref{original}, 
we mentioned the question whether or not the number of irreducible components of 
$\Hi^{\prec\Delta}_{S/k}$ equals $d(\Delta)$. 
If so, we would have a generalization of the first statement of Theorem \ref{mainthm} (iii) 
from $\Hi^{\prec \Delta,\et}_{S/k}$ to $\Hi^{\prec\Delta}_{S/k}$. 
(The scheme in question has only one connected component, as is shown in \cite{strata},
thus there is no hope for generalizing the second statement.)
However, that generalization does not apply: 

\begin{cor}\label{fullprime}
  For large $n$, the scheme $\Hi^{\prec \Delta}_{S/k}$ in general contains more than $d(\Delta)$
  irreducible components, whose relative dimension is in general larger than $\sum_{j=1}^n\#q_{j}(\Delta)$. 
\end{cor}

\begin{proof}
  Let us denote by $C_{\Delta,i}$, where $i=1,\ldots,a(\Delta)$, 
  the irreducible components of $\Hi^{\prec \Delta,\et}_{S/k}$; 
  and by $D_{\Delta,j}$, where $j=1,\ldots,b(\Delta)$, the irreducible components of $\Hi^{\prec \Delta}_{S/k}$. 
  As $\Hi^{\prec \Delta,\et}_{S/k} \subset \Hi^{\prec \Delta}_{S/k}$, 
  the closure in $\Hi^{\prec \Delta}_{S/k}$ of each $C_{\Delta,i}$ is equal to some $D_{\Delta,j}$. 
  It follows that $a(\Delta)\leq b(\Delta)$, and $a(\Delta)=b(\Delta)$ if, and only if, 
  each $D_{\Delta,j}$ arises from a $C_{\Delta,i}$ by taking the closure. 
  Moreover, the dimension of each $D_{\Delta,j}$, which arises from a $C_{\Delta,i}$ by taking the closure cannot exceed $nr$, 
  the dimension of $\Hi^{r,\et}_{S/k}$. However, from Theorem 2 of \cite{strata}, 
  we know that at the level of spaces of closed points, the identity 
  \begin{equation}\label{rdelta}
    \Hi^{r}_{S/k}=\coprod_{\Delta}\Hi^{\prec \Delta}_{S/k}
  \end{equation}
  holds, where the disjoint sum goes over all standard sets $\Delta$ of size $r$. 
  From equation (1) of \cite{iarrobino}, 
  we know that $\dim {\rm Hilb}^{r}_{S/k} > nr$ if $n = 3$ and $r \geq 102$ or $n = 4$ and $r \geq 25$. 
  It follows that $a(\Delta)<b(\Delta)$ for some $\Delta$. 
  
  As for the assertion of the relative dimension, we use \eqref{rdelta} once more. 
  As the relative dimension of the full Hilbert scheme of $r$ points is in general larger than $nr$, 
  there exists a $\Delta$ such that $\Hi^{\prec \Delta}_{S/k}$ has a relative dimension strictly larger than $nr$. 
  However, $\sum_{j=1}^n\#q_{j}(\Delta)\leq nr$ for all $\Delta$. 
\end{proof}

The proof of Corollary \ref{fullprime} fails if one replaces $\Hi^{\prec \Delta}_{S/k}$ by the scheme
\begin{equation*}
  \mathscr{G}^{\prec \Delta}_{S/k} := \Hi^{\prec \Delta}_{S/k} \cap \mathscr{G}^r_{S/k} ,
\end{equation*}
where $\mathscr{G}^r_{S/k}$ is the {\it good component} of $\Hi^r_{S/k}$, 
i.e., the scheme-theoretic closure of $\Hi^{r,\et}_{S/k}$ inside $\Hi^r_{S/k}$. 
(See \cite{ekedahlskjelnes} and \cite{rydhskjelnes} for constructions of the good component.)
Indeed, the dimension of the good component equals $nr$, and therefore, the above arguments do not lead to a contradiction. 
Therefore, there is hope that the number of irreducible components of $\mathscr{G}^{\prec \Delta}_{S/k}$ equals $d(\Delta)$. 
In other words, the hope is that the statement of Theorem \ref{mainthm} (iii) 
generalizes from $\Hi^{\prec \Delta,\et}_{S/k}$ to $\mathscr{G}^{\prec \Delta}_{S/k}$. 
However, the following example shows that also that generalization does not apply:
Consider the standard set
\begin{equation*}
  \Delta := \{0, e_1, e_2, e_3, 2e_1, e_1 + e_2 \} \in \mathcal{D}_3. 
\end{equation*}
$\Delta$ is of size $6$. 
Therefore Theorem 1.1 of \cite{velasco} shows that $\mathscr{G}^{\prec \Delta}_{S/k} = \Hi^{\prec \Delta}_{S/k}$. 
It is easy to see that $d(\Delta) = 1$. 
But with the help of Macaulay2 \cite{M2}, 
one finds that $\Hi^{\prec \Delta_1}_{S/k}$ has two irreducible components, 
one of the expected dimension, $\#q_1(\Delta_1) + \#q_2(\Delta_1) + \#q_3(\Delta_1) = 11$, 
the other of dimension $10$. 
Therefore, this specific $\Delta$ provides a counterexample to the conjectured generalization. 
(Another counterexample is provided by 
\begin{equation*}
  \Delta^\prime := \{0, e_1, e_2, e_3, e_1 + e_2, 2e_2 \} \in \mathcal{D}_3 .
\end{equation*}
The two standard sets $\Delta$ and $\Delta^\prime$ are the only three-dimensional standard sets of size $6$ 
providing counterexamples to the conjectured generalization.)

%%%%%%%%%%%%%%%%%%%%%%%%%%%%%%%%%%%%%%%%%%%%%%%%%%%%%%%% 

\bibliography{components.bib}
\bibliographystyle{amsalpha}

\end{document}